\documentclass[reqno,letterpaper,11pt]{amsart}
\usepackage{times}
\usepackage{amsfonts}
\usepackage{setspace}
\usepackage{amssymb}
\usepackage[all]{xy}

\usepackage[bookmarksopen=true]{hyperref}

\newcommand{\reg}{{\rm reg}}

\newcommand{\rR}{{\rm R}}
\newcommand{\ol}{\overline}

\newcommand{\Af}{\mathbb{A}_F}
\newcommand{\Ae}{\mathbb{A}_E}

\newcommand{\Gal}{{\rm Gal}}
\newcommand{\diag}{{\rm diag}}
\newcommand{\Fr}{{\rm Fr}}

\newcommand{\antidiag}{{\rm antidiag}}

\newcommand{\K}{\mathcal{K}}

\newcommand{\CC}{\mathbb{C}}

\newcommand{\Hom}{{\rm Hom}}

\newcommand{\lp}{\left(}
\newcommand{\rp}{\right)}
\newcommand{\la}{\left\langle}
\newcommand{\ra}{\right\rangle}
\newcommand{\lmx}{\begin{matrix}}
\newcommand{\rmx}{\end{matrix}}
\newcommand{\lsm}{\begin{smallmatrix}}
\newcommand{\rsm}{\end{smallmatrix}}
\newcommand{\tr}{{\rm tr}\;}

\newcommand{\ve}{\varepsilon}
\newcommand{\ep}{\epsilon}
\newcommand{\vp}{\varpi}

\newcommand{\GL}{{\rm GL}}

\newcommand{\PGSp}{{\rm PGSp}}

\newcommand{\Sp}{{\rm Sp}}
\newcommand{\Mp}{{\rm Mp}}

\newcommand{\HH}{\mathbf{H}(\Af)}

\newcommand{\N}{{\rm N}}

\newcommand{\bs}{\backslash}

\newcommand{\dps}{\displaystyle}
\newcommand{\n}{m}
\renewcommand{\K}{l}

\newcommand{\rH}{H'}
\newcommand{\Un}{{\rm U}}

\newcommand{\rG}{G'}

\newcommand{\rM}{M'}

\newcommand{\Aep}{\mathbb{A}_{E'}}

\newcommand{\rf}{f'}

\newcommand{\evp}{}
\newcommand{\st}{{\rm St}}

\providecommand{\mb}[1]{\mathbf{#1}}
\providecommand{\mc}[1]{\mathcal{#1}}
\providecommand{\mf}[1]{\mathfrak{#1}}
\providecommand{\mbb}[1]{\mathbb{#1}}

\providecommand{\abs}[1]{\left\vert{#1}\right\vert}

\providecommand{\idc}[1]{C_{#1}}

\CompileMatrices

\newtheorem{generic}{}[section]
\newtheorem{theorem}[generic]{Theorem}

\newtheorem{lemma}[generic]{Lemma}
\newtheorem{proposition}[generic]{Proposition}

\numberwithin{equation}{section}
\numberwithin{figure}{section}

\theoremstyle{remark}
\newtheorem*{remark}{\sc Remark}

\theoremstyle{plain}
\newtheorem*{thm*}{\sc Theorem}
\newtheorem*{propa}{Proposition A}%
\newtheorem*{propb}{Proposition B}%
\newtheorem*{propc}{Proposition C}%
\newtheorem*{propd}{Proposition D}%

\title{Cyclic odd degree base change lifting\\for unitary groups in three variables}

\author{Ping-Shun Chan}
\address{Max-Planck-Institut f\"ur Mathematik\\Vivatsgasse 7\\53111 Bonn\\GERMANY}
\curraddr{Department of Mathematics\\University of California, San Diego\\9500 Gilman Drive \# 0112\\La Jolla, CA 92093\\USA}
\email{pschan@caa.columbia.edu}
\author{Yuval Z. Flicker}
\address{Max-Planck-Institut f\"ur Mathematik\\Vivatsgasse 7\\53111 Bonn\\GERMANY}
\curraddr{Department of Mathematics\\The Ohio State University\\231 W 18th Avenue\\Columbus, OH 43210\\USA}
\email{flicker@math.osu.edu}
\date{\today}

\begin{document}


\maketitle

\markboth{PING-SHUN CHAN and YUVAL Z. FLICKER}{BASE CHANGE FOR UNITARY GROUPS}

\pagestyle{headings}

\begin{abstract}
Let $F$ be a number field or a $p$-adic field of odd residual characteristic.
Let $E$ be a quadratic extension of $F$, and $F'$ an odd degree cyclic field extension of $F$.
We establish a base-change functorial lifting of automorphic (resp.\ admissible) 
representations from the unitary group $\Un(3, E/F)$ to the unitary group $\Un(3, F'E/F')$.
As a consequence, we classify, up to certain restrictions,
the packets of $\Un(3, F'E/F')$ 
which contain irreducible automorphic (resp.\ admissible) representations invariant under the action 
of the Galois group $\Gal(F'E/E)$. 
We also determine the invariance of individual representations.
This work is the first study of base change into an algebraic group
whose packets are not all singletons, and which does not satisfy
the rigidity, or ``strong multiplicity one,'' theorem. Novel phenomena are encountered:
e.g.\ there are invariant packets where not every irreducible 
automorphic (resp.\ admissible) member is Galois-invariant. 
The restriction that the residual characteristic of the local fields be odd may be removed once the multiplicity one theorem
for $\Un(3)$ is proved to hold unconditionally without restriction on the dyadic places.
\end{abstract}
\section{Introduction}\label{sec:intro}
Let $F$ be a number field or a local nonarchimedean field of odd residual characteristic.
In the local case,
the restriction on the residual characteristic of $F$ 
may be removed once the multiplicity one theorem for $\Un(3)$ is proved to hold for all automorphic representations
without restriction on their dyadic local components.
Let $E$ be a quadratic extension of $F$, with $\alpha$ the generator of $\Gal(E/F)$.
For any $g = (g_{ij})$ in $\GL(3, E)$, or in the ad\`ele group $\GL(3, \Ae)$ if $E$ is a number field, 
put $\alpha(g) := (\alpha (g_{ij}))$ and
\[
\sigma(g) :=  J\,\alpha ({}^t g^{-1}) J^{-1},\quad
J = \lp\lsm &&1\\&-1&\\1&&\rsm\rp.
\]
Let $\mb{G} = \Un(3, E/F)$ be the unitary group in three variables with respect to $E/F$ and the Hermitian form
$(x, y) \mapsto x J\, {}^t\!\alpha(y)$.  
It is quasisplit and its group of $F$-points is:
\[
\Un(3, E/F)(F)
= \left\{
g \in \GL(3, E) : \sigma(g) = g
\right\}.
\]
Let $\rR_{E/F}\mb{G}$ be the $F$-group obtained from $\mb{G}$ by the restriction of scalars functor from $E$ to $F$.
Thus, $(\rR_{E/F}\mb{G})(F)$ is $\GL(3, E)$.

The base-change lifting from $\mb{G}$ to $\rR_{E/F}\mb{G}$ is established in \cite{F}
(some of whose results we summarize in Section \ref{sec:summaryF} of the Appendix, per kind suggestion of the referee).
If $F$ is a number field, then the lifting is a one-to-one correspondence between the stable packets 
of automorphic representations of $\mb{G}(\Af)$ 
and the $\sigma$-invariant, discrete spectrum, automorphic representations of $\GL(3, \Ae)$.  
Here, for a representation $(\pi, V)$ of a group $M$ and an automorphism $\gamma$ of $M$, 
we say that $\pi$ is {\bf $\gamma$-invariant} if $(\pi, V)$ is equivalent to $(\gamma\pi, V)$, 
where $\gamma\pi : m \mapsto \pi(\gamma(m))$ for all $m \in M$.  
If $F$ is local, the lifting is a one-to-one correspondence
between the packets of admissible representations of $\mb{G}(F)$ 
and the $\sigma$-invariant, $\sigma$-stable, admissible representations of $\GL(3, E)$.
Implicit is a definition of packets.  
We shall in due course recall these results in further detail.

Our work too is a study of base change for $\mb{G}$, albeit in a different sense.
Namely, we consider a nontrivial cyclic field extension $F'$ of $F$ of odd degree $n$,
and we establish the base-change lifting from the unitary group $\mb{G}$ associated with $E/F$
to the unitary group $\mb{\rG} := \rR_{F'/F}\mb{G}$ associated with $E'/F'$, 
where $E'$ is the compositum field $F'E$.  Note that $\mb{\rG}(F)$ is $\Un(3, E'/F')(F')$.
If $F$ is a number field (resp.\ local nonarchimedean field),
we classify, in terms of base change from $\mb{G}$,
the {\bf invariant packets} of $\mb{\rG}(\Af)$ (resp.\ $\mb{\rG}(F)$), 
namely those packets which contain an automorphic (resp.\ admissible) representation invariant under the action of $\Gal(E'/E)$.

In the nonarchimedean case, we show that every member of every invariant local packet is $\Gal(E'/E)$-invariant,
except for one case when $n = 3$.  
In this case, there is an invariant local packet $\{\pi'\}$ on $\Un(3, E'/F')$
consisting of four representations,
and only the unique generic member is Galois invariant (see Proposition \ref{prop:n3case}).

Let $\beta$ denote a generator of $\Gal(F'/F)$.  Denote by $\alpha$ and $\beta$ also the generators of
$\Gal(E'/F')$ and $\Gal(E'/E)$ whose restrictions to $E$ and $F'$ are $\alpha$ and $\beta$.
Thus, $\Gal(E'/F) = \la \alpha, \beta\ra$, with $\alpha \beta = \beta \alpha$.
We have the following system of field extensions, where each edge is labeled by the generator of the
Galois group of the corresponding field extension:
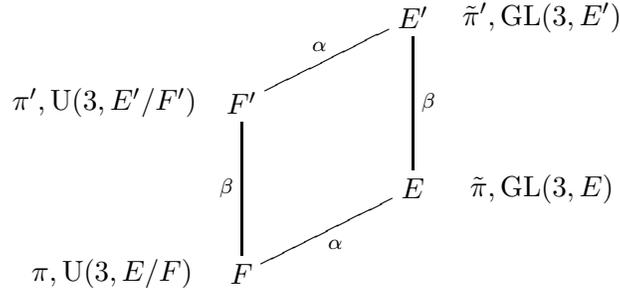
\[
\xymatrix@!=.7pc{
& & & E' & \quad\quad\quad \tilde{\pi}', \GL(3, E') \\
\pi', \Un(3, E'/F')\quad\quad\quad\;\; & F'\ar@{-}[rru]^\alpha  & & &\\
& & & E\ar@{-}[uu]_\beta & \quad\quad\quad \tilde{\pi}, \GL(3, E) \\
\pi, \Un(3, E/F) \quad\quad\quad & F\ar@{-}[uu]^\beta \ar@{-}[rru]_\alpha & & &
}
\]
Also recorded in the diagram above is our convention for denoting the representations of the groups.  
Namely, a representation of a general linear group is marked with a $\sim$, 
and a representation of a group obtained via the restriction-of-scalars functor $\rR_{F'/F}$ is marked with a prime $\prime$.

Unlike the base-change lifting from $\GL(m, F)$ to $\GL(m, F')$ (see \cite{BCGL2}, \cite{AC}; $m$ any positive integer), 
where each $\Gal(F'/F)$-invariant automorphic or local admissible representation of $\GL(m, {F'})$ is a lift from $\GL(m, F)$,
our case of base change for $\Un(3)$ involves two twisted endoscopic groups, $\mb{G}$ {\it and} $\mb{H} := \Un(2, E/F)$.  

In the global case, 
we say that a (quasi-) packet $\{\pi\}$ of $\mb{G}(\Af)$ or $\mb{H}(\Af)$ {\bf weakly lifts}
to a (quasi-) packet $\{\pi'\}$ of $\mb{\rG}(\Af)$ if, for almost all places $v$ of $F$,
the unramified $v$-component $\pi_v$ of $\{\pi\}$ 
lifts to the unramified $v$-component $\pi'_v$ of $\{\pi'\}$ according to the underlying base-change $L$-group homomorphism.
We say that $\{\pi\}$ {\bf lifts} (or ``strongly lifts'' for emphasis) 
to $\{\pi'\}$ if for {\it all} $v$ the base-change lift $(\{\pi\}_v)'$ of the $v$-component $\{\pi\}_v$ 
(which is a local packet) of $\{\pi\}$ is the $v$-component $\{\pi'\}_v$ of $\{\pi'\}$.
The definition of local base-change lifting at a place $v$ where $F'/F$ is split is simple: 
A local packet $\{\pi_v\}$ of $G_v$ {\bf lifts} to $\otimes_{i = 1}^n\{\pi_v\}$ of $\rG_v = \otimes_{i = 1}^n G_v$,
and a local packet $\{\rho_v\}$ of $H_v$ lifts to $\otimes_{i = 1}^n \pi(\{\rho_v\})$, where $\pi(\{\rho_v\})$ 
is the lift of $\{\rho_v\}$ to $G_v$ as defined in \cite[Part 2.\ Sec.\ III.2.3.\ Cor.]{F} (See Proposition \ref{prop:u2u3localpacket}
in the Appendix).
Note that if $v$ is archimedean, then $F'/F$ is split at $v$, since $n = [F':F]$ is odd.
In the nonsplit case, where the place $v$ is necessarily nonarchimedean, 
we define local lifting, in the context of $G_v$, $H_v$ and $\rG_v$, 
in terms of the local character identities summarized below.
Namely, one local packet lifts to another if they satisfy one of these local character identities.

We say that a (quasi-) packet is {\bf discrete spectrum} if it contains a discrete spectrum automorphic representation.
We say that a discrete spectrum (quasi-) packet of $\mb{\rG}(\Af) = \Un(3, E'/F')(\mbb{A}_{F'})$ 
is {\bf $\beta$-invariant} if it contains a $\beta$-invariant, discrete spectrum, automorphic, irreducible representation.
The proof of the following main global result is achieved at the end of the paper:
\begin{theorem}[Main Global Theorem]
\label{thm:classification}
Assuming that the multiplicity one theorem holds for $\Un(3)$,
each discrete spectrum $($quasi-$)$ packet $\{\pi\}$ 
of $\mb{G}(\Af)$ or $\mb{H}(\Af)$ weakly lifts to a discrete spectrum $($quasi-$)$ packet 
$\{\pi'\}$ of $\mb{\rG}(\Af)$.  This weak lifting is in fact a lifting.  
Its image consists precisely of those discrete spectrum $($quasi-$)$ packets of $\mb{\rG}(\Af)$ which are $\beta$-invariant.

\begin{remark}
(i) Our proof of the existence of the weak lifting does not require a new use of the trace formula.  
It relies mainly on the lifting results in \cite{F}.  Note that not all members of a 
$\beta$-invariant (quasi-) packet are necessarily $\beta$-invariant.

(ii) To show that every $\beta$-invariant (quasi-) packet of $\mb{\rG}(\Af)$ is in the image of the weak lifting,
we use a twisted trace formula, and the assumed multiplicity one theorem for $\Un(3)$.
These tools are also needed to establish the local lifting, and thus the (strong) global lifting.
It has been shown in \cite{F} that the multiplicity one theorem for $\Un(3)$ holds for
any automorphic representation each of whose dyadic local components lies in the packet of 
a constituent of a parabolically induced representation.
It is likely that the same proof applies in the remaining cases, but this has not been checked as yet.

\end{remark}
\end{theorem}

\subsection*{Summary of Local Character Identities}
\renewcommand{\K}{E}
Let $F, F', E, E'$ now be local nonarchimedean fields of odd residual characteristic.
We denote in roman font the group of $F$-points 
of an algebraic $F$-group, e.g.\ $G := \mb{G}(F)$.  
Let $\mb{\rH} = \rR_{F'/F}\mb{H}$.  Let:
\[
\begin{split}
\K^F &= \{x \in \K^\times : \N_{\K/F}x = 1\},\\
{\K'}^{F'} &= \{x \in {\K'}^\times: \N_{\K'/F'}x = 1\}.
 \end{split}
\]
For an irreducible admissible representation $(\pi, V)$ of a $p$-adic algebraic group $M$, a linear map $A : V \rightarrow V$, 
and a smooth, compactly supported modulo center function $f$ on $M$, 
which transforms under the center $Z(M)$ of $M$ via the inverse of the central character of $\pi$, 
put:
\[
\la \pi, f\ra_A := \tr \pi(f)A.
\]
Here, $\pi(f)$ is the convolution operator $\int_{Z(M)\bs M}\pi(m)f(m)\, dm$, 
and $dm$ is a fixed Haar measure on $M$, implicit in the notations.
The operator $\pi(f)$ is of trace class, since it has finite rank.
For simplicity we put $\la \pi, f\ra := \la \pi, f\ra_{{\rm id}}$, where ${\rm id}$ is the identity map on $\pi$.
If $\{\pi\}$ is a local packet of representations, 
we let $\la \{\pi\}, f \ra$ denote the sum $\sum_{\pi}\la \pi, f\ra$ over the members of $\{\pi\}$.

We now summarize our results in the nonarchimedean case.
Functions on the groups are assumed to be smooth, compactly supported modulo center.
They are said to be matching if their orbital integrals match in the sense of \cite[p.\ 17]{KS}.
\begin{propa}
Let $\pi'$ be the cuspidal $\rG$-module which is the sole member of the singleton local packet $\{\pi'(\tilde{\pi}')\}$
which lifts to a $\sigma$-invariant cuspidal representation $\tilde{\pi}'$ of $\GL(3, \K')$.
Then, the representation $\pi'$ is $\beta$-invariant if and only if there exists a cuspidal $G$-module $\pi$,
and a nonzero intertwining operator $A \in \Hom_G(\pi', \beta\pi')$, such that
the following character identity holds for matching functions $f$ and $\rf$ on $G$ and $\rG$, respectively$:$
\[
\la \pi', \rf\ra_A = \la \pi, f\ra.
\]
In the case where $\pi'$ is $\beta$-invariant,
$\pi$ is the sole member of the singleton local packet $\{\pi(\tilde{\pi})\}$ which lifts to
a $\sigma$-invariant cuspidal representation $\tilde{\pi}$ of $\GL(3, E)$, 
which in turn lifts to $\tilde{\pi}'$ via base change from $\GL(3, E)$ to $\GL(3, E')$.
{\rm (See Proposition \ref{prop:traceglobaldiscrete1}.)}
\end{propa}

Let $\rho$ be a cuspidal representation of $H$ which is not monomial, 
namely does not belong to any packet which is the endoscopic lift of a pair of characters of $\K^F$ 
(see Proposition \ref{prop:u1u2localpacket}).
Let $\pi(\rho)$ denote the local packet of $G$ which is the endoscopic lift of $\rho$.
The $H$-module $\rho$ base-change lifts to a $\Gal(E'/E)$-invariant cuspidal representation $\rho'$ 
of $\rH$.
Let $\pi'(\rho')$ denote the local packet of $\rG$ which is the endoscopic lift of $\rho'$ from $\rH$.  
It consists of two cuspidal representations $\pi'^+$ and $\pi'^-$.
\begin{propb}
There exist nonzero operators 
$A^\dagger \in \Hom_{\rG}(\pi'^\dagger, \beta\pi'^\dagger)$$(\dagger$ is $+$ or $-)$, 
such that
the following system of character identities holds for matching functions$:$
\[
\begin{split}
2 \la \pi'^+, \rf\ra_{A^+} &= \la \pi(\rho), f\ra + \la \rho, f_H\ra,\\
2 \la \pi'^-, \rf\ra_{A^-} &= \la \pi(\rho), f\ra - \la \rho, f_H\ra.
\end{split}
\]
In particular, both $\pi'^+$ and $\pi'^-$ are $\beta$-invariant.
{\rm (See Proposition \ref{prop:tracelocalunstable}.)}
\end{propb}

Let $\theta_1, \theta_2, \theta_3$ be distinct characters of $\K^F$.
For $i = 1, 2, 3$, 
let $\theta_i'$ denote the character $\theta_i\circ\N_{\K'/\K}$ of $\K'^{F'}$.
For $i \neq j \in \{1, 2, 3\}$,
let $\rho_{i, j} = \rho(\theta_i, \theta_j)$ denote the packet of representations of $H$ 
which is the endoscopic lift of the unordered pair $(\theta_i, \theta_j)$. 
Likewise, let  $\rho_{i, j}' = \rho(\theta_i', \theta_j')$ denote the packet of representations of $\rH$
which is the endoscopic lift of the unordered pair $(\theta_i', \theta_j')$.
Let $\{\pi\} := \{\pi(\theta_i, \theta_j)\}$  denote the unstable packet of $G$ 
which is the endoscopic lift of $\rho_{i, j}$ from $H$.
Let $\{\pi'(\theta_i', \theta_j')\}$  denote the unstable packet of $\rG$ 
which is the endoscopic lift of $\rho'_{i, j}$ from $\rH$ (see Proposition \ref{prop:u2u3localpacket}).
The packets $\{\pi'(\theta_i', \theta_j')\}$ are all equivalent to the same packet 
$\{\pi'\} = \{\pi'_a, \pi'_b, \pi'_c, \pi'_d\}$ consisting of four cuspidal representations of $\rG$.
We index the unique generic representation in $\{\pi'\}$ by $a$.
\begin{propc}
There exist nonzero intertwining operators 
$A_{*} \in \Hom_{\rG}(\pi'_*, \beta\pi'_*)$ $(* = a, b, c, d)$, 
and a way to index $\{\rho_{i, j} : {i \neq j \in \{1, 2, 3\}}\}$ as $\{\rho_1, \rho_2, \rho_3\}$,
such that the following system of local character identities holds for matching functions$:$
\[\begin{split}
4 \la \pi'_a, \rf\ra_{A_a} &= \la \{\pi\}, f\ra
+ \la \rho_1, f_H\ra + \la \rho_2, f_H\ra + \la \rho_3, f_H\ra,\\
4 \la \pi'_b, \rf\ra_{A_b} &= \la \{\pi\}, f \ra
- \la \rho_1, f_H\ra - \la \rho_2, f_H\ra + \la \rho_3, f_H\ra,\\
4 \la \pi'_c, \rf\ra_{A_c} &= \la \{\pi\}, f\ra
- \la \rho_1, f_H\ra + \la \rho_2, f_H\ra  - \la \rho_3, f_H\ra,\\
4 \la \pi'_d, \rf\ra_{A_d} &= \la \{\pi\}, f\ra
+ \la \rho_1, f_H\ra - \la \rho_2, f_H\ra - \la \rho_3, f_H\ra.
\end{split}\]
{\rm (See Proposition \ref{prop:uustablecaseii}.)}
\end{propc}

Suppose $n = 3$.
In the global situation,
the cuspidal monomial automorphic representations of $\GL(3, \Ae)$, associated with non-$\beta$-invariant characters of $C_{E'}$,
base-change lift to parabolically induced representations of $\GL(3, \mbb{A}_{E'})$.  
Consequently, there are stable packets of $\Un(3, E/F)(\Af)$ 
which lift to {\it unstable} packets of $\Un(3, E'/F')(\mbb{A}_{F'})$.  
This phenomenon in turn leads to the following result in the local case:
\begin{propd}
Suppose $n = 3$.
Let $\theta$ be a character of $\K'^{F'}$ such that $\theta \neq \beta\theta$.
Let $\rho' = \rho'(\theta, \beta\theta)$ be the local packet $($of cardinality $2$$)$ of $\rH$ 
which is the lift of  $\theta\otimes\beta\theta$ from $\K'^{F'} \times \K'^{F'}$.
Let $\{\pi'\} = \{\pi'_a, \pi'_b, \pi'_c, \pi'_d\}$ be the packet of $\rG$
with central character $\omega = \theta\circ\N_{E'/E}$, which is the lift of $\rho'$,
where $\pi'_a$ is the unique generic representation in the packet.
Then, $\pi'_a$ is the \textbf{only} 
$\beta$-invariant representation in the packet, and there exists a cuspidal representation $\pi$ of $G$
and a nonzero intertwining operator $A \in \Hom_{\rG}(\pi'_a, \beta\pi'_a)$, such that$:$
\[
\la \pi'_a, \rf\ra_A = \la \pi, f\ra
\]
for all matching functions.
More precisely, $\pi$ is the sole member of the singleton packet of $G$ 
which lifts to the cuspidal monomial representation $\pi(\kappa'^2\tilde{\theta})$ 
of $\GL(3, \K)$ associated with $\kappa'^2\tilde{\theta}$, where
$\tilde{\theta}(z) := \theta(z/\alpha(z))$, $z \in \K'^\times$, and $\kappa$ is a character of $\K'^\times$ 
whose restriction to $F'^\times$ is the quadratic character corresponding to the field extension $\K'/F'$ via local class field theory.
In particular, $\pi(\kappa'^2\tilde{\theta})$ is $\sigma$-invariant.
{\rm (See Proposition \ref{prop:n3case}.)}
\end{propd}

\renewcommand{\K}{l}

In both the global and local cases, we make extensive use of the results in \cite{F}, 
where much of the heavy work pertaining to this article has been done.  

\subsection*{Acknowledgment}
The first author would like to thank the MPIM-Bonn and UC San Diego for their generous support throughout the preparation of this article.
The second author also wishes to thank the MPIM-Bonn, and the Humboldt foundation, for their kind support while working on this project.
We thank the referee for the careful reading of this manuscript,
and Wee Teck Gan and Dipendra Prasad for helpful conversations on the theta correspondence.

\section{Global Case}\label{sec:globalcase}
In this section, $F$ is a number field.
We set up the trace formula machinery and establish  global character identities,
which are instrumental in deducing the local character identities in Section \ref{sec:localtraceidentities},
and the base-change lifting of automorphic representations from $\mb{G}$ and $\mb{H}$ to $\mb{\rG}$.  
As an intermediate step, we first prove the weak base-change lifting, 
which simply follows from book-keeping of a system of $L$-group homomorphisms
(Proposition \ref{prop:weaklift}).
\subsection{Definitions and Notations}
For a number field $L$, let $\mbb{A}_L$ denote its ring of ad\`eles, 
and $C_L$ the id\`ele class group $L^\times \bs \mbb{A}_L^\times$.

For a character $\mu$ of a group $M$ and an automorphism $\gamma$ of $M$,
we let $\gamma\mu$ denote the character defined by
$
(\gamma\mu)(m) = \mu(\gamma(m))$, $m \in M.
$
So, for example, if $\mu$ is a character of $C_E$, 
then $(\alpha\mu)(x) = \mu(\alpha(x))$ for all $x \in C_E$.

Put
\[
C_E^F := \Un(1, E/F)(F)\, \bs\, \Un(1, E/F)(\Af).
\]  
The embedding of the group $\Un(1, E/F)(\Af)$ in $\Ae^\times$ induces an embedding $C_E^F \hookrightarrow C_E$.
We define $C_{E'}^{F'}$ similarly.
Fix once and for all a $\beta$-invariant character $\omega'$ of $C_{E'}^{F'}$.
Noting that the center $\mb{Z}_{\rG}(\Af)$ of $\mb{\rG}(\Af)$ is isomorphic to the group 
$\Un(1, E'/F')(\mbb{A}_{F'})$,
we identify $\omega'$ with a character of $\mb{Z}_{\rG}(\Af)$.
By Lemmas \ref{lemma:globalcentralchars} and \ref{lemma:globalnormindex} in the Appendix, 
$\omega'$ determines a unique character $\omega$ of $C_E^F$ such that $\omega' = \omega\circ\N_{E'/E}$.
We identify $\omega$ with a character of the center $\mb{Z}_{\mb{G}}(\Af)$ of $\mb{G}(\Af)$.

Let $\rho$ be the right-regular representation of $\mb{G}(\Af)$ on the space of 
functions $\phi$ in $L^2(\mb{G}(F)\bs\mb{G}(\Af))$ satisfying $\phi(z g) = \omega(z)\phi(g)$ for all $z \in \mb{Z}_{\mb{G}}(\Af)$, 
$g \in \mb{G}(\Af)$.
By an automorphic representation of $\mb{G}(\Af)$
we mean an irreducible representation which is equivalent to a subquotient of $\rho$.
In particular, the automorphic representations of $\mb{G}(\Af)$ which we consider all have central character
$\omega$.  We define an automorphic representation of $\mb{\rG}(\Af)$ similarly, restricting our attention to those which have
central character $\omega'$.

Strictly speaking, our definition of an automorphic representation is not the same as that in the traditional sense, 
namely, a representation which acts on the space of automorphic forms.  Nonetheless, the discrete spectrum of the
right-regular representation on the space of square-integrable functions coincide with that on the space of automorphic forms.  

For an automorphic representation $(\pi', V)$ of $\mb{\rG}(\Af)$,
where $V$ is the vector space of $\pi'$, let $(\beta \pi', V)$ denote the representation:
\[
\beta\pi' : g \mapsto \pi'(\beta (g)),\quad \forall g \in \GL(3, \Ae),
\]
where $\beta(g_{ij}) := (\beta (g_{ij}))$.
We say that the representation $(\pi', V)$ is {\bf $\beta$-invariant} if $(\pi', V) \cong (\beta \pi', V)$.
That is, there exists a vector space automorphism $A$ of $V$ such that $\pi'(g)A = A\pi'(\beta(g))$ for all $g \in \mb{\rG}(\Af)$.

\subsection{Endoscopy, $L$-Group Homomorphisms}\label{sec:endoscopy}
We let ${}^L\! M$ denote the $L$-group of an algebraic group $\mb{M}$.
All of the algebraic $F$-groups which we consider are quasisplit over $F$ and split over a finite extension of $F$.  
For the Weil component of an $L$-group, instead of the full Weil group
we consider only the Weil group associated with the smallest field extension over which the group is split.

Recall that $n := [F':F] = [E':E]$.
The Weil group $W_{E'/F}$ is an extension of $\Gal(E'/F) = \la \alpha, \beta\ra$ by $C_{E'}$ (see \cite{T}).
The $L$-group of ${\mb \rG}$ is 
\[
{}^L \rG = \GL(3, \mbb{C})^n\rtimes W_{E'/F},
\]
where $C_{E'} \subset W_{E'/F}$ acts trivially on the identity component 
$\hat{G}' := {}^L \rG^0 = \GL(3, \CC)^n$,
and the actions of $\alpha, \beta \in W_{E'/F}$ on $\hat{G}'$ are given by:
\[
\alpha(g_1,\ldots, g_n) = (\theta(g_1),\ldots, \theta(g_n)),
\quad 
\beta(g_1,\ldots, g_n) = (g_n, g_1, g_2, \ldots, g_{n-1}),\\
\]
for all $(g_1,\ldots g_n) \in \GL(3, \mbb{C})^n$, where 
\[
\theta(g) := J\; {}^t g^{-1} J,\quad J = \lp\lsm &&1\\&-1&\\1&&\rsm\rp.
\]
The action of $\beta$ on ${}^L \rG$ induces an automorphism, also denoted by $\beta$, of the algebraic group $\mb{\rG}$, 
which in turn induces the automorphism $\beta : (g_{ij}) \mapsto (\beta(g_{ij}))$ of $\mb{\rG}(\Af)$.

In analogy to base change for $\GL(m)$, base change with respect to $F'/F$ for $\mb{G}$
is a case of twisted endoscopic lifting with respect to the automorphism $\beta$ of $\mb{\rG}$.
In \cite{KS},
the elliptic regular part of the geometric side of the twisted trace formula of a connected reductive group is stabilized,
and shown to be equal to the sum of the geometric sides of the stabilized nontwisted trace formulas of 
the twisted elliptic endoscopic groups, 
at least for global test functions whose orbital integrals are supported on elliptic regular elements.

In this work, by a twisted (or ``$\beta$-twisted'' for emphasis) endoscopic group we mean a quasisplit reductive $F$-group 
whose $L$-group is isomorphic to
\[
Z({}^L \rG, s\beta) := \{ g \in {}^L G' : s \beta(g) s^{-1} = g \},
\] 
the $\beta$-twisted centralizer in ${}^L \rG$ of some semisimple element $s$ in $\hat{G}'$.
We say that an endoscopic group $\mb{M}$ is {\bf elliptic} if $Z(\hat{M})^{\Gal(\bar{F}/F)}$ 
is contained in $Z(\hat{G})$.  Here, $\hat{M}$ denotes the identity component of the $L$-group of $\mb{M}$, 
$Z(\hat{M})$ denotes its center, 
and $Z(\hat{M})^{\Gal(\bar{F}/F)}$ 
denotes the subgroup of elements in $Z(\hat{M})$ fixed by the action of $\Gal(\bar{F}/F)$.

Observe that any element $(g_1, g_2, g_3, \ldots, g_n)$ in $\hat{G}'$ is $\beta$-conjugate
(i.e.\ conjugate in $\hat{G}'\rtimes\la\beta\ra$)
to \[(g_n g_{n - 1}\cdots g_2 g_1, I, I,\ldots, I),\]
where $I$ is the $3 \times 3$ identity matrix.
In this work, it suffices to consider the twisted elliptic endoscopic groups of $\mb{\rG}$ 
only up to the equivalence of endoscopic data (see \cite[Chap.\ 2]{KS}).
Hence, we may assume that the semisimple element
$s \in \hat{G}$ has the form $s = (s_1, I, I, \ldots, I)$, where $s_1 = \diag(a, b, c)$.
The condition $s\beta(g)s^{-1} = g$ for all $g \in {}^L G'$ implies in particular that:
\[
(s_1\theta(s_1^{-1}), I, I,\ldots, I; \alpha)
= s (I, I, \ldots, I; \alpha) s^{-1}
= (I, I, \ldots, I; \alpha).
\]
Thus, $s_1 = \diag(a, \pm 1, a^{-1})$ for some $a \in \mbb{C}^\times$.
If $a \neq \pm 1$, then the twisted endoscopic group associated with $s$ is not elliptic.
Hence, up to the equivalence of endoscopic data, 
$\mb{\rG}$ has two $\beta$-twisted elliptic endoscopic groups.
One is $\mb{G} = \Un(3, E/F)$.  Its $L$-group is isomorphic to the twisted centralizer of
$s = (I, \ldots, I) \in \GL(3, \CC)^n$.  The other is
$\mb{H} = \Un(2, E/F) \times \Un(1, E/F)$.  It corresponds to the element:
\[s = (\diag(-1, 1, -1), I, I,\ldots, I).\]
Here, we go by the convention:
\[
\Un(2, E/F)(F) := \left\{
g \in \GL(2, E) : \lp\lsm & 1\\-1&\rsm\rp \alpha ({}^t g^{-1})\lp\lsm & -1\\1&\rsm\rp = g
\right\}.
\]

We now describe the $L$-group of each twisted endoscopic group and its accompanying $L$-homomorphism into ${}^L \rG$.
Their computation is straightforward given the work of \cite{KS}.  Hence, we only give the final results without
elaborating on the intermediate work.  
For the Weil components of these $L$-groups, we use:
\[
W_{E/F} = \la z, \alpha : z \in C_E,\; \alpha^2 \in C_F - \N_{E/F}C_E,\; \alpha z = \alpha(z) \alpha\ra.
\]
$\mb{G}$:\\
${}^L G = \GL(3, \mbb{C}) \rtimes W_{E/F}$.  The action of $\alpha \in W_{E/F}$
on $\GL(3, \CC)$ is defined by $\alpha(g) = \theta(g)$, and $C_E \subset W_{E/F}$ acts trivially.
The $L$-homomorphism $b_G: {}^L G \rightarrow {}^L \rG$ is defined as follows:
\begin{gather*}
b_G(g) = \Delta g := \underbrace{(g, g,\ldots, g)}_{\text{diagonal embedding}} \in \GL(3, \CC)^n ,\quad \forall g \in \GL(3, \CC);\\
b_G(w) = w,\quad \forall\, w \in W_{E/F}.
\end{gather*}
$\mb{H}$:\\
${}^L H = \left[\GL(2, \CC) \times \CC^\times\right] \rtimes W_{E/F}$.
The action of $\alpha$ on $\GL(2, \CC)\times\CC^\times$
is defined by:
\[
\alpha (g, x) = (\theta_2(g), x^{-1}), \quad \forall\, (g, x) \in \GL(2, \CC)\times \CC^\times,
\] 
where $\theta_2(g) := \lp\lsm &1\\-1&\rsm\rp {}^t g^{-1} \lp\lsm &-1\\1&\rsm\rp$.
The $L$-homomorphism $e_H : {}^L H \rightarrow {}^L \rG$ is given by:
\begin{gather*}
e_H((g, x)) = \Delta [g, x];\; [g, x] := \lp\lsm a & & b\\& x &\\c & & d\rsm\rp,
\quad g = \lp\lsm a&b\\c&d\rsm\rp \in \GL(2, \CC),\, x \in \CC^\times;\\
e_H(\alpha) = \Delta [\diag(1, -1), 1] \rtimes \alpha,\\
e_H(z) = \Delta[\diag(\kappa(z), \kappa(z)), 1] \rtimes z,\quad \forall\, z \in C_E \subset W_{E/F}.
\end{gather*}
Here, $\kappa$ is a fixed character of $C_E$ such that $\kappa|_{C_F} = \ve_{E/F}$, 
the quadratic character of $\idc{F}$ associated to the number field extension $E/F$ via global class field theory.

A representation $\pi$ of $\mb{H}(\Af)$ has the form $\rho \otimes \chi$, where $\rho$ is a representation of the group 
$\Un(2, E/F)(\Af)$ and $\chi$ is a character of $\Un(1, E/F)(\Af)$.
Since the automorphic representations of $\mb{\rG}(\Af)$ which we study all have the same fixed central character $\omega'$, 
the representations which lift to them from the endoscopic groups all have the same fixed central character $\omega$.
Thus, $\chi$ is uniquely determined by $\rho$.  
We often abuse notation and simply write $\mb{H} = \Un(2, E/F)$, $\pi = \rho$.

In \cite{F}, 
the classification of the packets of 
$\mb{G}(\Af)$ is expressed
in terms of the Langlands functorial lifting from $\Un(2, E/F)$ to $\mb{G}$, and from $\mb{G}$ to $\rR_{E/F}\mb{G}$.  
These liftings correspond to a system of $L$-homomorphisms among the $L$-groups.
Our work makes extensive use of these results; hence, it is important to see how our system of liftings
``fits in'' with those of \cite[Sec.\ 1]{F}.  This is described by the diagram of $L$-homomorphisms in Figure \ref{fig:lgroups}, 
where the $L$-homomorphisms of \cite{F} and base change for $\GL(m)$ 
($m = 2, 3$; see \cite{BCGL2}, \cite{AC}) are represented by dotted arrows. 
Except for the two curved arrows labeled as $b_s$ and $b_s'$, the diagram is commutative.
\setcounter{figure}{1}
\begin{figure}[htp]
\begin{gather*}
\xymatrix@!=3.5pc{
{}^L \rR_{E'/F} \Un(2, E/F) \ar@{.>}[rrr]_*+++\txt{\bf{}}^{i'} &&& {}^L \rR_{E'/F}\Un(3, E/F)\\
& **[l]
\substack{{}^L \rR_{F'/F} H = {}^L\!\rH}
\ar@{.>}[lu]_{b_u'}\ar@{.>}[r]^{e'} \ar@/^1pc/@{.>}[ul]|-{b_s'} 
& **[r]\substack{{}^L \rG  = {}^L \rR_{F'/F} G} \ar@{.>}[ru]^{b'}&\\
& {}^L H \ar@{.>}[dl]^{b_u} \ar@/_1pc/@{.>}[dl]|-{b_s}  \ar[u]^{b_H} 
\ar[ru]|{e_H}_*+{\txt{\bf{}}} 
\ar@{.>}[r]_e & {}^L G \ar[u]_{b_G} \ar@{.>}[dr]^b &\\
{}^L \rR_{E/F}\Un(2, E/F) \ar@{.>}[uuu]^{b_{2}}_*++++++++{\txt{\bf{}}}\ar@{.>}[rrr]_i^*+++\txt{\bf{}} &&& 
{}^L \rR_{E/F}\Un(3, E/F) \ar@{.>}[uuu]_{b_{3}}^*++++++++\txt{\bf{}}
}
\end{gather*}
\caption{}\label{fig:lgroups}
\end{figure}

The functorial lifts corresponding to the $L$-homomorphisms in Figure \ref{fig:lgroups} are as follows:
\renewcommand{\labelenumi}{(R\arabic{enumi})\quad}
\newcounter{Lcount}
\begin{list}{\hspace{36pt}$\bullet$\;}{\usecounter{Lcount} \leftmargin 0mm}
\item $i, i'$
correspond to normalized parabolic induction from the standard $(2, 1)$-parabolic subgroup of $\GL(3)$, 
i.e.\  normalized induction from the
upper triangular parabolic subgroup with Levi component isomorphic to $\GL(2)\times\GL(1)$.
Here,
it is implicitly understood that the central character is fixed; hence, the character of the $\GL(1)$-component
of the Levi subgroup is uniquely determined by the representation of the $\GL(2)$-component.
\item $b_{m}$ ($m = 2, 3$)
corresponds to base change for $\GL(m)$ with respect to the number field extension $E'/E$ (see \cite{BCGL2}, \cite{AC}).
\item $b_u$ (resp.\ $b_s$) corresponds to the unstable (resp.\ stable) endoscopic lifting from $\Un(2, E/F)$  to $\GL(2, E)$, 
as established in \cite{FU2} (see Proposition \ref{prop:u2gl2globalpacket} in the Appendix).  
For both maps, we let $b_u'$, $b_s'$ denote their counterparts
in the case of the algebraic groups obtained via the restriction-of-scalars functor $\rR_{F'/F}$.
\item $b$
corresponds to the base-change lifting from $\Un(3, E/F)$ to $\GL(3, E)$ as established in \cite{F} 
(see Theorem \ref{thm:u3gl3stableglobalpacket}).
The map $b'$ is the counterpart to $b$ for the groups obtained via restriction of scalars.
\item $e$
and its restriction-of-scalars analogue $e'$ correspond to the endoscopic lifting from $\Un(2)$ to $\Un(3)$ 
(see Proposition \ref{prop:u2u3gl3globalpacket}).
\item $b_H$ 
corresponds to base change from the group $\mb{H}$ to $\mb{\rH} := \rR_{F'/F} \mb{H}$, yet to be established.
We have ${}^L\rH = \GL(2, \CC)^n\rtimes \Gal(E'/F)$, 
where the Galois action on the identity component $\hat{H}'$ is given by:
\begin{gather*}
\alpha(g_1, g_2, \ldots, g_{n}) = (\theta_2(g_1), \theta_2(g_2), \ldots, \theta_2(g_n)),\\
\beta(g_1, g_2,\ldots, g_n) = (g_n, g_1, g_2, \ldots, g_{n-1}).
\end{gather*}
The $L$-homomorphism $b_H$ is defined by:
\begin{gather*}
b_H(g) = (g,\ldots, g), \quad g \in \GL(2, \CC);\\
b_H(\alpha) = \alpha.
\end{gather*}
\end{list}

\subsection{Trace Formula}
The strategy to establish global functorial lifting via the trace formula technique is well-known.  First, we relate
the traces of automorphic representations to orbital integrals, via the twisted Arthur's trace formula (\cite{CLL}).  
We then make use of the stabilization of the elliptic regular part of the geometric side of the twisted trace formula of $\mb{\rG}$, 
as shown in \cite{KS},
in order to relate the traces of the automorphic representations of $\mb{\rG}(\Af)$ to those of its twisted endoscopic groups.
Note that the Kottwitz-Shelstad trace formula is shown in \cite{KS} to hold provided that each test function on a given group
transfers to matching test functions on its endoscopic groups.  
This conjecture on the existence of matching functions has since been proved due to the work of 
Ngo (\cite{Ngo}) and Waldspurger (\cite{W}, \cite{Wchar}, \cite{Watordue}).
Hence, the stabilization of the elliptic regular parts of the geometric sides of 
the (twisted) trace formulas is now known to hold without need of further qualification.

\subsubsection{Trace Formula, Geometric Side}\label{sec:tfgeometric}
In the case of $\beta$-twisted endoscopy for $\mb{\rG}$, the main result of \cite{KS} implies that the following equation holds:
\begin{equation}\label{eq:kstf}
T_e(\rG, \rf \times \beta) = 1 \cdot ST_e(G, f) + \frac{1}{2} \cdot ST_e(H, f_H).
\end{equation}
Here, $T_e(\rG, \rf\times\beta)$ is the {\it elliptic regular} part of the geometric side of the twisted trace formula of $\mb{\rG}$.
It is the sum of the $\beta$-twisted orbital integrals of a
smooth, compactly supported mod center test function $f'$ at the elliptic regular elements of $\mb{\rG}(F)$.
The symbol $ST_e(G, f)$ (resp.\ $ST_e(H, f_H)$) denotes the elliptic regular part of the stable trace formula of $\mb{G}$
(resp.\ $\mb{H}$).
It is the sum of the stable orbital integrals of the test function $f$ (resp.\ $f_H$) matching $\rf$, 
at the elliptic regular elements of $\mb{G}(F)$ (resp.\ $\mb{H}(F)$).
We refer to \cite[p. 75]{KS} (resp.\ \cite[p. 113]{KS})
for the definition of a twisted (resp.\ stable) orbital integral.
The coefficients $1$ and $1/2$ are computed
from a recipe given explicitly in \cite[p. 115]{KS}.

There is an obstacle which one must negotiate in using the Kottwitz-Shelstad formula to establish functorial liftings.
Namely, equation \eqref{eq:kstf} only relates the orbital integrals associated with elliptic
regular elements.  Consequently, each term in the equation may not fully coincide with the geometric side of a trace
formula for general test functions.  So, equation \eqref{eq:kstf} 
may not be readily used to relate the spectral sides of the trace formulas of the groups under consideration.
One way to overcome this is to restrict the space of test functions to those which have two elliptic local components,
as was done in \cite{FPGSP4}, for example.  
An undesirable consequence of this approach is that under it one can study only those automorphic
representations which have at least two elliptic local components.  In the case of lifting from $\Un(3)$ to $\rR_{E/F}\Un(3)$, 
one may use instead an argument analogous to one used in \cite[Part 2.\ II.\ 4]{F}, 
for the lifting of representations from $\Un(3, E/F)$ to $\GL(3, E)$.  A crude outline of the argument is as follows:

Fix a nonarchimedean place $u$ of $F$ which splits in $E$.
We consider global test functions whose local components at $u$
are Iwahori functions (i.e.\ biinvariant under the standard Iwahori subgroup) which vanish on the singular elements.
Applying the last proposition of \cite{FKsimpletf} to $\Un(3)$, it is shown that for such test functions
the discrete parts of the spectral sides of the (twisted) trace formulas of $\GL(3, E)$, $\Un(3, E/F)$ and  $\Un(2, E/F)$
satisfy an equation parallel to the Kottwitz-Shelstad formula.  Using the fact that $\Un(3, E/F)(F_u)$ (resp.\ $\Un(2, E/F)(F_u)$)
is isomorphic to $\GL(3, F_u)$ (resp.\ $\GL(2, F_u)$) at the place $u$, 
it is then shown that the equation holds for all matching global test functions.

The argument outlined above relies on no special properties of the groups 
which are not possessed by those studied this work, hence it readily applies to our case.
Via an argument borrowed from \cite[Part 2.\ II. 4]{F}, we can then ``separate by eigenvalues'' and derive global character
identities relating the terms which occur in the discrete parts of the spectral sides of the trace formulas.

\subsubsection{Discrete Spectral Terms in the Twisted Trace Formula}\label{sec:tfspectral}
Define an operator $T_\beta$ on $L^2 := L^2(\mb{\rG}(F)\bs \mb{\rG}(\Af))$ as follows:
\[
\lp T_\beta\,\varphi \rp(g) = \varphi(\beta (g)),\quad 
\forall \varphi \in L^2,\; g \in \mb{\rG}(\Af).
\]
Fix once and for all a Haar measure $dg'$ on $\mb{\rG}(\Af)$.  
Let $(\pi', V)$ be an irreducible, discrete spectrum, automorphic representation of $\mb{\rG}(\Af)$.
Thus, it is an irreducible subrepresentation of $L^2$.
Let $\rf$ be a smooth, compactly supported mod center test function 
on $\mb{\rG}(\Af)$ which transforms under the center $\mb{Z}_{\rG}(\Af)$ of $\mb{\rG}(\Af)$ 
via the inverse of the central character $\omega'$ of $\pi'$.
We define the $\beta$-twisted convolution operator $\pi'(\rf \times \beta)$
on $(\pi', V)$ as follows:
\[
\pi'(\rf \times \beta)v
= \int_{\mb{Z}_{\rG}(\Af)\bs \mb{\rG}(\Af)} \pi'(g')\rf(g') T_\beta\, v\; dg',\quad v \in V.
\]
It is of trace class, since $\pi'$ is admissible.
We let $\tr \pi'(\rf \times \beta)$ denote its trace.
It is easy to see that $\pi'$ is $\beta$-invariant if its
$\beta$-twisted character $f' \mapsto \tr \pi'(f'\times \beta)$ defines a nonzero distribution on the
space of smooth, compactly supported mod center functions on $\mb{\rG}(\Af)$.
Conversely, suppose $\pi'$ is $\beta$-invariant and occurs in the discrete spectrum of $\mb{\rG}(\Af)$.
The operator $T_\beta$ intertwines $\pi'$ with a subrepresentation $W$ of $L^2(\mb{\rG}(F)\bs\mb{\rG}(\Af))$ 
which is equivalent to $\beta \pi'$.  
If for each dyadic place $v$ of $F$,
the representation $\pi'_v$ belongs to a local packet containing a constituent of a parabolically induced representation,
then $\pi'$ occurs in the discrete spectrum with multiplicity one
(\cite[Part 2.\ Cor.\ III.5.2.2(1)]{F}), and $W$ must coincide with the space of $\pi'$.
This implies that the distribution $\rf \mapsto \tr \pi'(\rf \times \beta)$ is nonzero.
If the multiplicity one theorem holds for $\Un(3, E'/F')$ without any condition on the dyadic local
components of the representations, 
then a discrete spectrum  automorphic representation of $\mb{\rG}(\Af)$ 
is $\beta$-invariant if and only if its $\beta$-twisted character is nonzero.

For characters $\mu$ of $C_E$ and $\eta$ of $C_E^F$, let $I(\mu, \eta)$ denote the automorphic representation of $\mb{G}(\Af)$
which is parabolically induced, with normalization, from the representation:
\[
\mu\otimes\eta : \lp\lsm a & * & *\\ & b & * \\ & & \alpha(a)^{-1}\rsm\rp \mapsto \mu(a)\eta(b), \quad
a \in C_E,\; b \in C_E^F,
\]
of the upper triangular Borel subgroup.  Here, $\alpha$ is the generator of the Galois group $\Gal(E/F)$.
Similarly, given characters $\mu'$ of $C_{E'}$ and $\eta'$ of $C_{E'}^{F'}$, 
let $I(\mu', \eta')$ denote the automorphic representation of $\mb{\rG}(\Af)$ parabolically induced, 
with normalization, from the representation:
\[
\mu'\otimes\eta' : \lp\lsm a & * & *\\ & b & * \\ & & \alpha(a)^{-1}\rsm\rp \mapsto \mu'(a)\eta'(b), \quad
a \in C_{E'}, \quad b \in C_{E'}^{F'},
\]
of the Borel subgroup of $\mb{\rG}(\Af)$, where $\alpha$ is the generator of $\Gal(E'/F')$.
Let $I(\mu)$ denote the representation of $\mb{H}(\Af) = \Un(2, E/F)(\Af)$
parabolically induced, with normalization, 
from the following representation of the Borel subgroup
of $\mb{H}(\Af)$:
\[
\mu : \lp\lsm a & * \\& \alpha (a)^{-1}\rsm\rp \mapsto \mu(a), \quad a \in C_E.
\]

\begin{lemma}\label{lemma:tfrGspectral}
The discrete part $I(\rG, \rf\times\beta)$
of the spectral side of the $\beta$-twisted trace formula of $\mb{\rG}$ is the sum of the following terms$:$
\renewcommand{\labelenumi}{(R\arabic{enumi})\quad}
\newcounter{rGcount}
\begin{list}{\;\;{\rm (G$'$\arabic{rGcount})}\quad}{\usecounter{rGcount} \leftmargin 0mm}
\item
$\dps
\sum_{\substack{\pi' = \beta\pi'}} \tr \pi'(\rf \times \beta).
$\\
The sum is over the $\beta$-invariant, irreducible, discrete spectrum, automorphic representations $\pi'$ of $\mb{\rG}(\Af)$. 
\item
$\dps
\frac{1}{4}
\sum_{\substack{\mu'\cdot\alpha\beta\mu' = 1\\ \mu'=\beta\mu', \eta'=\beta\eta'}}
\!\!\tr M I(\mu', \eta')(\rf\times\beta).
$\\
The sum is over the characters $\mu'$ of $C_{E'}$ and $\eta'$ of $C_{E'}^{F'}$ which satisfy the conditions specified.
The symbol $M$ denotes an intertwining operator on the possibly reducible induced representation $I(\mu', \eta')$.
\end{list}
\end{lemma}
\noindent The computation of these expressions 
follows from a technical procedure described in \cite[Lec.\ 15]{CLL}.
To simplify the exposition, we include it in Appendix \ref{appendix:DOR}.  

The upcoming sections often make reference to the results in \cite{FU2} and \cite{F}.
A summary of the results in these works is given in Section \ref{sec:summaryF} of the Appendix.

\subsubsection{Global Packets}\label{sec:globalpackets}
The spectral side of a stable trace formula is indexed by packets and quasi-packets (see \cite{LL}, \cite{Arthur}).
A definition of \mbox{(quasi-)}packets for $\mb{G} = \Un(3, E/F)$ is provided in \cite[p.\ 217]{F}: 
Let $\{\pi\}$ be a set of representations which is a {\bf restricted tensor product:}
\[
\otimes_v' \{\pi_v\} := \{ \otimes_v \pi_v : {\pi_v \in \{\pi_v\} \text{ for all } v,
\pi_v \text{ is unramified for almost all } v}\},
\]
where the tensor product is over all places $v$ of $F$, and for each place $v$ 
the set $\{\pi_v\}$ is a {\bf local (quasi-)packet} of admissible representations of $\mb{G}(F_v)$.
The local packets of $\mb{G}(F_v)$ are defined in \cite{F} via twisted local character identities,
associated with the lifting from $\Un(3, E/F)$ to $\GL(3, E)$ (see Proposition \ref{prop:u3gl3localpacket}).
If $\{\pi\}$ weakly lifts (that is, almost all local components lift)
via the $L$-homomorphism $b$ (see Figure \ref{fig:lgroups})
to a generic or one-dimensional automorphic representation of $\GL(3, \Ae)$, 
we say that $\{\pi\}$ is a {\bf packet}.
If $\{\pi\}$ weakly lifts to a non-generic and non-one-dimensional automorphic representation of $\GL(3, \Ae)$,
we say that it is a {\bf quasi-packet}.
The (quasi-)packets of $\Un(2)$ are defined likewise (\cite{FU2}).
We say that a packet containing a discrete spectrum automorphic representation
is {\bf stable} if each of its members occurs with equal positive multiplicity
in the discrete spectrum of the group.
We say that it is {\bf unstable} if not all of its members belong to the discrete spectrum.
Unstable packets are in principle lifts from proper endoscopic groups.

For an element $g$ in $\GL(3, \Ae)$ (resp.\ $\GL(2, \Ae)$), let 
\[
\sigma (g) = J\; \alpha({}^t g^{-1}) J^{-1},\;
\text{ where } J := \lp\lsm &&1\\&-1&\\1&&\rsm\rp \; \lp \text{ resp.\ } \lp\lsm &1\\-1&\rsm\rp \rp.
\]
We define automorphisms $\sigma'$ on $\GL(3, \mbb{A}_{E'})$ and $\GL(2, \mbb{A}_{E'})$ in exactly the same way,
but with $\alpha$ viewed as the generator of $\Gal(E'/F')$.
For $m = 2$ or $3$, we say that an automorphic representation $(\pi, V)$ of $\GL(m, \Ae)$ is 
$\sigma$-invariant if $(\pi, V)$ is equivalent to $(\sigma \pi, V)$, where $(\sigma \pi, V)$ is the
$\GL(m, \Ae)$-module defined by
$
\sigma \pi : g \mapsto \pi(\sigma (g))$, $g \in \GL(m, \Ae)$.

By \cite[p.\ 217, 218]{F} (see Theorem \ref{thm:u3gl3stableglobalpacket}), 
a packet containing a discrete spectrum automorphic representation of $\mb{G}(\Af)$ is
stable if and only if it lifts to a $\sigma$-invariant, discrete spectrum, automorphic representation of 
$\GL(3, \Ae)$.  It is unstable if and only if it is equal to the lift $\pi(\{\rho\})$
of a packet $\{\rho\}$ on the proper endoscopic group $\mb{H}$ of $\mb{G}$.  By Proposition \ref{prop:u2gl2globalpacket},
a packet containing a discrete spectrum representation of $\mb{H}(\Af)$ 
is stable if and only if it lifts to a discrete spectrum, automorphic representation of
$\GL(2, \Ae)$.  Otherwise, it is the lift $\rho(\theta_1, \theta_2)$
of an unordered pair $(\theta_1, \theta_2)$ of distinct characters of $C_E^F$. 
In other words, $\rho(\theta_1, \theta_2)$ 
is a lift from the endoscopic group $\Un(1, E/F)\times\Un(1, E/F)$ of $\mb{H}$.

\subsubsection{Discrete Spectral Terms in the Stable Trace Formulas}\label{sec:DORstf}
The discrete parts $SI(G, f)$, $SI(H, f_H)$
of the spectral sides of the stable trace formulas of $\mb{G}$, $\mb{H}$
are computed in \cite[Part 2.\ Sec.\ II.3]{F} and \cite[Sec.\ 3]{FU2}, respectively.
For a packet $\{\pi\}$ of automorphic representations, let $\tr\!\!\{\pi\}(f)$ denote the
sum of characters $\sum_{\pi \in \{\pi\}} \tr\! \pi(f)$.  For a given smooth, 
compactly supported mod center function $f$, 
only finitely many terms in the sum are nonzero.
We say that a packet is {\bf discrete spectrum} if it contains a discrete spectrum automorphic representation.
The discrete part $SI(G, f)$ of the stable trace formula of $\mb{G}(\Af)$ is the sum of the following terms:
\renewcommand{\labelenumi}{G\arabic{enumi}\quad}
\newcounter{Gcount}
\begin{list}{\hspace{36pt}{\rm (G\arabic{Gcount})}\quad}{\usecounter{Gcount} \leftmargin 0mm}
\item
$\dps
\sum_{\text{Stable, discrete spectrum }\substack{\{\pi\}}} \tr\{\pi\}(f).
$
\\
The sum is over all stable discrete spectrum (quasi-) packets $\{\pi\}$ of automorphic representations of $\mb{G}(\Af)$.
By virtue of its stability, $\{\pi\}$ is not of the form $\pi(\{\rho\})$ for any packet $\{\rho\}$ of $\mb{H}(\Af)$.
\item
$\dps
\frac{1}{2}
\sum_{\{\rho\} \neq \rho(\theta_1, \theta_2)}
\tr \pi(\{\rho\})(f).
$
\\
The sum is over the discrete spectrum packets $\{\rho\}$ of $\mb{H}(\Af)$ which are not in the image of the endoscopic
lifting from $\Un(1, E/F) \times \Un(1, E/F)$. %
\item
$\dps
\frac{1}{4}
\sum_{\{\pi\} = \pi(\rho(\theta_1, \theta_2))}
\tr\{\pi\}(f).
$
\\
The sum is over the equivalence classes of packets of the form $\pi(\rho(\theta_1, \theta_2))$
for some pair $(\theta_1$, $\theta_2)$ of characters of $C_{E}^{F}$, such that
$\theta_1, \theta_2, \omega/\theta_1\theta_2$ are distinct.  
Here, $\omega$ is the fixed character of the center of $\mb{G}(\Af)$, which is identified with $C_E^F$.
\item
$\dps
- \frac{1}{8} \sum_{\substack{\mu, \eta\\\mu\alpha\mu = 1\\\mu|_{C_F} \neq 1}} 
\!\!\tr I(\mu, \eta) (f).
$
\\
The sum is over the characters $\mu$ of $C_E$ and $\eta$ of $C_E^F$ 
which satisfy the conditions specified.
\end{list}

By Proposition 1 in \cite[Sec.\ 3]{FU2}, 
the discrete part $SI(H, f_H)$ of the stable trace formula of $\mb{H}(\Af)$ is the sum of the following terms: 
\renewcommand{\labelenumi}{(H\arabic{enumi})\quad}
\newcounter{Hcount}
\begin{list}{\hspace{36pt}{\rm (H\arabic{Hcount})}\quad}{\usecounter{Hcount} \leftmargin 0mm}
\item
$\dps
\sum_{\{\rho\} \neq \rho(\theta_1, \theta_2)}
\tr \{\rho\}(f_H).
$\\
The sum is over the discrete spectrum packets $\{\rho\}$ of $\Un(2, E/F)$ 
which are not of the form $\rho(\theta_1, \theta_2)$ for
any characters $\theta_1$, $\theta_2$ of $C_E^F$.
\item
$\dps
\frac{1}{2}\sum_{\substack{\{\theta_1, \theta_2\}\\ \theta_1 \neq \theta_2}} 
\tr \rho(\theta_1, \theta_2)(f_H).
$\\
The sum is over all unordered pairs $(\theta_1, \theta_2)$ of distinct characters of $C_E^F$.
\item
$\dps
- \frac{1}{4} \sum_{\substack{ \mu\cdot\alpha\mu = 1 \\ \mu|_{C_F} = 1}}
\tr I(\mu, \eta)(f_H).
$\\
The sum is over all characters $\mu$ of $C_E$ satisfying the specified conditions.
\end{list}

\subsection{Separation by Eigenvalues}\label{sec:globaltf}
For an algebraic group $\mb{M}$ defined over $F$ and a place $v$ of $F$, put $M_v := \mb{M}(F_v)$.
Let $E_v$ denote $E\otimes_F F_v$. 
Let $S$ be a finite set of places of $F$. 
Let $t_G(S) = \{t_v : {v \notin S}\}$ be a set of conjugacy classes in ${}^L G$; 
or rather, each $t_v$ is a conjugacy class in the local version ${}^L G_v = \hat{G} \rtimes W_{E_v/F_v}$ of ${}^L G$.
Via the $L$-homomorphism $b_G$ (see Figure \ref{fig:lgroups}), each $t_v$ lifts to
a conjugacy class ${t'_v}$ in ${}^L \rG$.  We thus obtain: \evp 
\[b_G(t_G(S)) := t_{\rG}(S) = \{t'_v : v \notin S\}\]
in ${}^L\rG$.  We say that $b_G(t_{G}(S))$ is the {\bf lift} of $t_G(S)$.
We extend to $\mb{H}$ the  analogous notions of such a set of conjugacy classes and its lifting to ${}^L\rG$.

Given an automorphic representation $\pi = \otimes_v \pi_v$ (tensor product over all places of $F$) of $\mb{G}(\Af)$,
its local component $\pi_v$ is an unramified representation of $G_v$ 
for every place $v$ outside of some finite set of places $S(\pi)$ which contains all the archimedean places.
Since an unramified representation $\pi_v$ corresponds to a Hecke conjugacy class $t_v = t(\pi_v)$ in the $L$-group (\cite{B}), 
the representation $\pi$ defines \evp  
$t(\pi, S) = \{t(\pi_v) : v \notin S\}$ for any finite set of places $S$ containing $S(\pi)$.
Any two members of a (quasi-) packet define the same $t(\pi, S)$ for some finite set $S$.
We often suppress $S$ from the notation and write simply $t(\pi)$.

We say that one (quasi-) packet $\{\pi_1\}$ {\bf weakly lifts} to another (quasi-) packet $\{\pi_2\}$
if, for some finite set of places $S$, the set of conjugacy classes $t(\pi, S)$ 
associated with a member $\pi_1$ of $\{\pi_1\}$ lifts to $t(\pi_2, S)$ for some member $\pi_2$ of $\{\pi_2\}$.
In \cite{F}, a (quasi-) packet $\{\pi\}  = \otimes_v'\{\pi_v\}$ of automorphic representations of $\mb{G}(\Af)$ 
is said to {\bf lift} to an automorphic representation $\tilde{\pi}$ of $\GL(3, \Ae)$ if $\{\pi\}$ weakly lifts to $\tilde{\pi}$,
and for each place $v$ the local packet $\{\pi_v\}$ satisfies a specific local character identity involving $\tilde{\pi}_v$.
Similarly, we say that a (quasi-) packet $\{\rho\}$ of $\mb{H}(\Af)$ lifts to $\{\pi\}$ of $\mb{G}(\Af)$ 
if the former weakly lifts to the latter, 
and the local packets $\{\rho_v\}$ and $\{\pi_v\}$ satisfy the local character identities summarized in 
\cite[p.\ 214, 215]{F}.

For each place $v$ of $F$, let $\beta_v$ be the image of $\beta \in \Gal(E'/E)$ in the group 
$\Gal(E'_v/E_v)$.
Let $\pi'_v$ be an irreducible admissible representation of $\rG_v$.
Write $\beta_v\pi_v'$ for the representation $g \mapsto \pi'_v(\beta_v(g))$.  
By definition, $\pi'_v$ is $\beta_v$-invariant
if there exists a nonzero $\rG_v$-equivariant map $A(\pi'_v)$
in $\Hom_{\rG_v}(\pi'_v, \beta_v\pi'_v)$.  The map $A(\pi'_v)$ is called an {\bf intertwining operator}.
Since, for any nonzero operator $B$ in $\Hom_{\rG_v}(\pi'_v, \beta_v\pi'_v)$, the map $B^{-1}\circ A(\pi'_v)$ 
intertwines the irreducible admissible $\pi'_v$ with itself, $A(\pi'_v)$ is unique up to a scalar by Schur's Lemma.
Since $\beta_v^n = 1$, Schur's Lemma also implies that $A(\pi'_v)^n$ is a scalar multiplication, which we normalize to be the identity map.  
If in addition $\pi_v'$ and $E'_v/F'_v$ are unramified, we scale $A(\pi'_v)$ 
so that it sends each $\mb{\rG}(\mc{O}_v)$-fixed vector of $\pi'_v$ to itself.  
Here, $\mc{O}_v$ denotes the ring of integers of $F_v$.
If $\pi'_v$ is not $\beta_v$-invariant, we put $A(\pi'_v) := 0$.

We make the following observation regarding the case where $v$ splits completely in $F'$: 
We have $E'_v = E'\otimes_F F_v = \prod_{i = 1}^n E_v$.
The group $\rG_v = \mb{\rG}(F_v)$ is the direct product $\prod_{i = 1}^n\Un(3, E'_v/F'_v)$, 
and $\beta_v$ acts on $\rG_v$ via:
\[\beta_v(g_1,\ldots, g_n) = (g_n,g_1,\ldots, g_{n-1}),\quad \forall\,(g_1,\ldots, g_n) \in \rG_v.\]
A representation $\pi'_v$ of $\rG_v$ is of the form $\otimes_{i = 1}^n\pi'_i$, 
where $\pi'_i$ ($1 \leq i \leq n$) is a representation of $\Un(3, E'_v/F'_v)$.
Hence,
\[
(\beta_v\pi'_v)(g_v) = \pi'_1(g_n)\otimes\pi'_2(g_1)\otimes\dots\otimes\pi'_n(g_{n - 1}),
\quad \forall\, g_v = (g_1,\ldots,g_n) \in \rG_v.
\]
So, $\pi'_v$ is $\beta_v$-invariant if and only if
the $\pi'_i$'s are all equivalent.  
For a $\beta_v$-invariant representation $\pi'_v$ of $\rG_v$, we let 
$A(\pi'_v)$ be the operator in $\Hom_{\rG_v}(\pi'_v, \beta_v\pi'_v)$ 
which sends each vector $\xi_1\otimes\dots\otimes \xi_n$ in $\otimes_{i = 1}^n\pi'_v$
to the vector $\xi_n\otimes\xi_1 \otimes\dots\otimes \xi_{n-1}$.

There are additional cases to consider according to how the place $v$ splits in $E$ and $F'$.
For instance, if $n$ is not prime, $v$ may not split completely in $F'$.  If $v$ splits in $E$, 
then $\rG_v = \GL(3, F'_v)$, which further splits into a product of groups if $v$ splits in $F'$.
We leave it as an exercise for the reader to formulate what it means for $\pi'_v$ to be $\beta_v$-invariant in such cases.

Let $v$ be a place of $F$.
For a smooth, compactly supported mod center function
$\rf_v$ on $\rG_v$, satisfying $\rf_v(zg) = \omega_v'^{-1}(z)\rf(g)$ for all $z$ in the center $Z(\rG_v)$ of $\rG_v$, 
let $\tr\pi'_v(\rf_v \times \beta_v)$ denote the trace of the twisted convolution operator
\[
\int_{Z(\rG_v)\bs \rG_v} \pi'_v(g'_v)\rf_v(g'_v) A(\pi'_v)\,dg_v'.
\]
Here, $dg_v'$ is a fixed Haar measure on $\rG_v$ 
such that $\mb{\rG}(\mc{O}_v)$ has volume one for almost all $v$,
and the tensor product $\otimes_v\, dg_v'$ over all places of $F$ 
coincides with a fixed Tamagawa measure $dg'$ on $\mb{\rG}(\Af)$.

Let $\pi' = \otimes_v \pi'_v$ be an
automorphic representation which occurs in the discrete spectrum of $\mb{\rG}(\Af)$ with multiplicity one.
Recall the definition of the operator $T_\beta$ on $L^2 = L^2(\mb{\rG}(F)\bs \mb{\rG}(\Af))$:
\[
\lp T_\beta\varphi\rp(g) = \varphi(\beta (g)),\quad 
\forall \varphi \in L^2,\; g \in \mb{\rG}(\Af).
\]
It defines an intertwining operator $T(\pi')$ in $\Hom_{\mb{\rG}(\Af)}(\otimes_v\pi'_v,\otimes_v\beta_v\pi'_v)$
which differs from the operator $\otimes_v A(\pi'_v)$ by at most an $n$-th root of unity $\ep(\pi')$,
since $T_\beta^n = 1$.
Note that $T(\pi') \neq 0$ if and only if $\pi'$ is $\beta$-invariant.

\label{page:betatwist}

From henceforth, we assume that each global test function $f$ is a tensor product $\otimes_v f_v$ of local functions.
Let $\pi', \pi$ and $\pi_H$ be irreducible, discrete spectrum, automorphic representations of the groups $\mb{\rG}(\Af), \mb{G}(\Af)$ and
$\mb{H}(\Af)$, respectively.
To simplify the notation, put 
\[
\la \pi', \rf \times \beta\ra_v := \la \pi'_v, \rf_v\times\beta_v\ra := \tr \pi'_v(\rf_v \times \beta_v),
\]
\[
\la \pi', \rf \times \beta\ra  := \tr \pi'(\rf\times\beta) = \ep(\pi')\prod_{v} \la \pi_v', \rf_v \times \beta_v\ra,
\]
\[
\la \pi, f \ra := \tr \pi(f) = \prod_{v}\la \pi_v, f_v\ra,
\]
and extend the notation analogously to $\pi_H$.  Here, the products are over all places $v$ of $F$.
Recall that a (quasi-) packet $\{\pi\}$ of automorphic representations is a restricted tensor product $\otimes_v' \{\pi_v\}$ 
of local (quasi-) packets.
We put
\[
\la \{\pi\}, f \ra := \sum_{\pi \in \{\pi\}} \la \pi, f \ra = \prod_{v} \la \{\pi_v\}, f_v \ra,
\]
where $\la \{\pi_v\}, f_v \ra := \sum_{\pi_v \in \{\pi_v\}} c(\pi_v) \la \pi_v, f_v\ra$.  
The coefficients $c(\pi_v)$ are equal to $1$ if $\{\pi_v\}$ is a local packet, 
but they may not be constant over the members of $\{\pi_v\}$ if $\{\pi_v\}$ is a quasi-packet (see \cite[p.\ 214, 215]{F}).

Let $S$ be a set of places of $F$.
For an object which is a tensor product of local components over a set of places containing $S$,
let subscript $S$ denote the
tensor product over the places in $S$.  For example, 
$f_S := \otimes_{v \in S} f_v$, and $\la \pi, f\ra_S := \prod_{v \in S}\la \pi_v, f_v\ra$.

Fix a finite set of places $S$ containing all the archimedean places 
and the nonarchimedean places where at least one of the number field extensions $E/F$, $F'/F$ is ramified.
Fix a collection of conjugacy classes  $t_{\rG} = t_{\rG}(S) = \{t'_v : v \notin S\}$ in ${}^L \rG$.
We choose the matching test functions $\rf, f, f_H$ such that, 
at every place $v \notin S$, their local components $\rf_v, f_v, f_{H, v}$
are spherical and correspond to one another via the duals of the $L$-homomorphisms $b_G$, $e_H$ (see \cite{B}).
In particular, for any unramified representations $\pi'_v, \pi_v$ and $\pi_{H, v}$ of $\rG_v, G_v$ and $H_v$
whose corresponding conjugacy classes in the $L$-groups are related via the $L$-homomorphisms, 
the Satake transforms ${\rf_v}^\vee(\pi'_v)$, $f_v^\vee(\pi_{v})$ and $f_{H, v}^\vee(\pi_{H, v})$ coincide.
The Fundamental Lemma (\cite{Ngo}) asserts that such $\rf_v, f_v, f_{H, v}$ have matching orbital integrals.
Using equation \eqref{eq:kstf} and the twisted Arthur's trace formula, and
borrowing the technique of \cite[Part 2.\ II.\ 4.4]{F} (see Section \ref{sec:tfgeometric}), 
we may separate the terms in the discrete sums $I(\rG, \rf\times\beta)$, $SI(G, f)$, $SI(H, f_H)$
by eigenvalues (or more precisely by Hecke conjugacy classes), and obtain an expression of the form:
\begin{multline}\label{eq:sepevp}
\sum_{\pi' \in \{\pi'\}} m(\pi') \ep(\pi')\la \pi', \rf\times \beta\ra_S\\
=
\sum_{\{\pi\}} m(\{\pi\}) \la \{\pi\}, f\ra_S + 
\frac{1}{2} \sum_{\{\rho\}} m(\{\rho\})\la \{\rho\}, f_H \ra_S.
\end{multline}
The sum on the left is over the $\beta$-invariant automorphic representations $\pi'$ of $\mb{\rG}(\Af)$
with the property that $\pi'_v$ is unramified and parameterized by $t'_v$ for each $v \notin S$.
In particular, these representations all lie in the same packet $\{\pi'\}$.
Each sum on the right is over the (quasi-) packets of $\mb{G}(\Af)$, resp.\ $\mb{H}(\Af)$, 
which weakly lift to $\{\pi'\}$.  
The coefficients $m(\pi'), m(\{\pi\})$ and $m(\{\rho\})$ 
are those associated with the contributions
of the representations/packets to the discrete parts of the spectral sides of the (twisted) trace formulas.
The values of $m(\{\pi\})$ are recorded in Section \ref{sec:DORstf}.  
They are equal to $0$, $1$, $1/2$, $1/4$, or $-1/8$.
The values of $m(\{\rho\})$ are also recorded in that section.  They are equal to $0$, $1$, $1/2$ or $-1/4$. 
The values of $m(\pi')$ shall be described in due course.

Note that we can produce via the same procedure a global trace identity of the form \eqref{eq:sepevp} 
if we fix at the beginning a set of conjugacy classes in ${}^L G$ or ${}^L H$ instead.  

\begin{proposition}\label{prop:classification1}
Suppose the multiplicity one theorem holds for $\Un(3, E'/F')$.
Let $\pi'$ be a $\beta$-invariant, discrete spectrum, automorphic representation of the group $\mb{\rG}(\Af)$.
Then, $\pi'$ belongs to a $($quasi-$)$ packet which is the weak lift of $($quasi-$)$ packet$($s$)$ of automorphic representations of 
$\mb{G}(\Af)$ and/or $\mb{H}(\Af)$.
\end{proposition}
\begin{proof}
Apply equation \eqref{eq:sepevp} to the set of conjugacy classes $t(\pi', S)$ associated with $\pi'$, 
for some finite set of places $S$ outside of which the local components of $\pi'$ 
and the number field extensions $E/F$, $F'/F$ are unramified.
By the multiplicity one theorem for $\Un(3, E'/F')$, the $\beta$-invariance of $\pi$ implies that $\la \pi', \rf \times \beta\ra_S$
is nonzero.  Moreover, by the linear independence of twisted characters, 
the sum of the twisted characters associated with the $\beta$-invariant representations which contribute to 
the left-hand side of \eqref{eq:sepevp} must be nonzero.
Consequently, there must be a representation which has nonzero contribution to the right-hand side.
\end{proof}

\subsection{Global Character Identities}\label{sec:discretespectrumpackets}
With the machinery \eqref{eq:sepevp} in place, we may now establish concrete cases of weak global lifting and 
global character identities.
We let $S$ denote a finite set of places of $F$ which includes all the archimedean places 
and the nonarchimedean ones where at least one of the field extensions $E/F$, $F'/F$ is ramified.
We shall consider below representations which are unramified at all places $v \notin S$.
We let $\mb{\rH} := \rR_{F'/F}\mb{H}$.

Recall from Section \ref{sec:endoscopy} that there is a character $\kappa$ of $C_E$
associated with the $L$-homomorphism $b_u$.  This character $\kappa$ is trivial on $\N_{E/F} C_E$ but not on $C_F$.
Likewise, the $L$-homomorphism $b_u'$ is associated with the character $\kappa' := \kappa\circ\N_{E'/E}$ of $C_{E'}$.

We identify the Weil group $W_{E/F}$ with $\{(z, \tau) : z \in C_E,\, \tau \in \Gal(E/F)\}$, where
$C_E = E^\times$ if $E$ is a $p$-adic field, and $C_E = E^\times\bs\Ae^\times$ if $E$ is a number field.
We consider the following form of the $L$-group of $\mb{H} = \Un(2, E/F)$:
\[
{}^L H = \GL(2, \mbb{C})\rtimes W_{E/F},
\]
where $(z, \alpha)(g) = \lp\lsm&1\\-1&\rsm\rp{}^tg^{-1}\lp\lsm&-1\\1&\rsm\rp$ for all $z \in C_E$.
The form of the $L$-group of $\rR_{E/F}\mb{H}$ which we consider is:
\[
{}^L \rR_{E/F}H = {}^L \rR_{E/F}\GL(2) = (\GL(2, \mbb{C})\times\GL(2, \mbb{C}))\rtimes W_{E/F},
\]
where $(z, \alpha)(g_1, g_2) = (g_2, g_1)$ for all $z \in C_E$.
The $L$-group homomorphisms $b_u$, $b_s$ from ${}^L H$ to ${}^L \rR_{E/F}H$ are defined as follows
(\cite[p.\ 692]{FU2}):
\[
b_s : g \rtimes (z, \tau) \mapsto (g, g)\rtimes (z, \tau),
\]
\[
b_u : g \rtimes (z, \tau) \mapsto (g\kappa(z), g\kappa(z)\delta(\tau))\rtimes(z, \tau),
\]
where $\delta(1) = 1$, $\delta(\alpha) = -1$.

\begin{lemma}\label{lemma:u2cd}
The following $L$-group diagrams are commutative:
\[
\xymatrix{
{}^L \rR_{F'E/F'}\rH & \\ 
& {}^L \rH \ar@{.>}[lu]_{b_s'} \\
& {}^L H \ar@{.>}[dl]^{b_s} \ar[u]^{b_H} \\
{}^L \rR_{E/F}H \ar@{.>}[uuu]^{b_{2}} &
}
\quad\quad\quad
\xymatrix{
{}^L \rR_{F'E/F'}\rH & \\ 
& {}^L \rH \ar@{.>}[lu]_{b_u'} \\
& {}^L H \ar@{.>}[dl]^{b_u} \ar[u]^{b_H} \\
{}^L \rR_{E/F}H \ar@{.>}[uuu]^{b_{2}} &
}
\]
\end{lemma}
\begin{proof}
This follows directly from the definitions of the $L$-group homomorphisms.
\end{proof}

\subsubsection{Unstable Packets}
Let $\{\rho\}$ be a stable (quasi-) packet of $\Un(2, E/F)(\Af)$. 
In particular, $\{\rho\}$ 
is not of the form $\rho(\theta_1, \theta_2)$ 
for any characters $\theta_1$, $\theta_2$ of $C_E^F$ (see Section \ref{sec:globalpackets}). 
To study the weak global lifting of $\{\rho\}$ to $\mb{\rG}(\Af)$, 
we apply to it equation \eqref{eq:sepevp}
and make use of the commutativity of the $L$-group diagram in Figure \ref{fig:lgroups}.
The procedure, though somewhat tedious, is not difficult.  
Hence, we try to elucidate the process as much as we can in this one case of weak lifting, 
and then skim over the details in the later cases.  
The reader may find Figure \ref{fig:dspackets} a useful depiction of the system of weak liftings involving $\{\rho\}$.
We are particularly interested in the global liftings which correspond to the $L$-homomorphisms $b_G$ and $e_H$.
\begin{figure}[htp!]
\begin{minipage}{5in}
\xymatrix{
*+{\tilde{\rho}'} \ar@{.>}[rrr]_*+++\txt{\bf{}}^{i'} 
&&&
*+{\tilde{\pi}'} \\ 
&*+ 
{\rho'} \ar@{.>}[lu]_{b_u'/b_s'}
\ar@{.>}[r]^{e'} 
&*+ 
{\{\pi'\}} \ar@{.>}[ru]^{b'}&\\
&*+ 
{\rho}\ar@{.>}[dl]^{b_u} 
\ar[u]^{b_H} 
\ar[ru]|{e_H}_*+{\txt{\bf{}}} 
\ar@{.>}[r]_e 
&*+ 
{\{\pi\}} \ar[u]_{b_G} \ar@{.>}[dr]^b &\\
*+{\tilde{\rho}} 
\ar@{.>}[uuu]^{b_{2}}_*++++++++{\txt{\bf{}}}\ar@{.>}[rrr]_i^*+++\txt{\bf{}} 
&&& 
*+{\tilde{\pi}} 
\ar@{.>}[uuu]_{b_{3}}^*++++++++\txt{\bf{}}
}
\end{minipage}\quad
\begin{tabular}{|c|c|}
\hline
Rep./Packet & Group\\
\hline
$\rho$ & $\Un(2, E/F)(\Af)$\\
\hline
$\rho'$ & $\Un(2, E'/F')(\mbb{A}_{F'})$\\
\hline
$\tilde{\rho}$ & $\GL(2, \Ae)$\\
\hline
$\tilde{\rho}'$ & $\GL(2, \mbb{A}_{E'})$\\
\hline
$\{\pi\}$ 
& $\Un(3, E/F)(\Af)$\\
\hline
$\{\pi'\}$ 
& $\Un(3, E'/F')(\mbb{A}_{F'})$\\
\hline
$\tilde{\pi}$ & $\GL(3, \Ae)$\\
\hline
$\tilde{\pi}'$ & $\GL(3, \mbb{A}_{E'})$\\
\hline
\end{tabular}
\caption{}\label{fig:dspackets}
\end{figure}

Recall that we have fixed a character $\omega$ of $C_E^F$.  
Let $\eta$ denote the character $\omega /\omega_{{\rho}}$ of $C_E^F$, where $\omega_{{\rho}}$
is the central character of ${\rho}$.
Let $\tilde{\eta}(z) = \eta(z/\alpha(z))$ for all $z \in C_E$.
The packet $\{\rho\}$ lifts via $b_s$ (resp.\ $b_u$)
to a $\sigma$-invariant cuspidal automorphic representation $\tilde{\rho}$ (resp.\ $\tilde{\rho}\otimes\kappa$) of $\GL(2, \Ae)$;
and it lifts via $e$ to an unstable packet $\{\pi\} := \pi(\{\rho\})$ of $\mb{G}(\Af)$,
which in turn $b$-lifts to the $\sigma$-invariant representation $\tilde{\pi} := I_{(2, 1)}(\kappa\tilde{\rho}, \tilde{\eta})$
of $\GL(3, \Ae)$ induced from a parabolic subgroup whose Levi component is isomorphic to $\GL(2, \Ae)\times \Ae^\times$
(see Proposition \ref{prop:u2u3gl3globalpacket}).

The representation $\tilde{\pi}$ lifts via the base change $b_3$ to the parabolically induced representation
$\tilde{\pi}' := I_{(2, 1)}(\kappa'\tilde{\rho}', \tilde{\eta}')$ 
of $\GL(3, \mbb{A}_{E'})$, where $\tilde{\rho}'$ is the base change of $\tilde{\rho}$
to $\GL(2, \mbb{A}_{E'})$, and $\tilde{\eta}' = \tilde{\eta}\circ\N_{E'/E}$.
Since $[E':E] = n \neq 2$, the representation $\tilde{\rho}'$ must be cuspidal by Theorem 4.2(a) in \cite[Chap.\ 3]{AC}.

\begin{lemma}\label{lemma:sigmainvariance2}
The representation $\tilde{\rho}'$ is the lift via $b_s'$ of a stable packet $\{\rho'\}$ on
$\rR_{F'/F}\mb{H}(\Af) = \Un(2, E'/F')(\mbb{A}_{F'})$, which coincide with the $b_H$-lift of $\{\rho\}$.
\end{lemma}
\begin{proof}
We first show that $\tilde{\rho}'$ is $\sigma'$-invariant.
At each place $v$ where $\tilde{\rho}_v$ and $E_v/F_v$ are unramified,
$\tilde{\rho}_v$ is parameterized by the conjugacy class in ${}^L \rR_{E/F}\GL(2) = \GL(2, \mbb{C})^2 \rtimes W_{E_v/F_v}$ 
of an element of the form
$t_v = (g_1, g_2)\rtimes \Fr_v$, where $\Fr_v$ denotes the Frobenius element.
Since $\tilde{\rho}$ is $\sigma$-invariant, for every such place $v$ there exists an element $h_v \in {}^L \rR_{E/F}\GL(2)$
such that $h_v^{-1}t_v h_v = (\theta(g_2), \theta(g_1))\rtimes\Fr_v$, where 
$\theta(g_i) := \lp\lsm &1\\-1&\rsm\rp \;{}^t g_i^{-1}\lp\lsm &-1\\1&\rsm\rp$.
By the definition of base change, each unramified component $\tilde{\rho}'_v$ of $\tilde{\rho}'$ 
is parameterized by the conjugacy class of $t'_v = \Delta(g_1, g_2) \rtimes \Fr_v$
in ${}^L \rR_{E'/F}\GL(2) = (\GL(2, \mbb{C})^2)^n \rtimes W_{E'_v/F_v}$, 
where $\Delta(g_1, g_2)$ is the diagonal embedding of $(g_1, g_2)$.
We have \[(\Delta h_v)^{-1} t_v' (\Delta h_v) = \Delta(\theta(g_2), \theta(g_1))\rtimes\Fr_v,\]
which implies that $\tilde{\rho}'_v$ is $\sigma'_v$-invariant, 
where $\sigma'_v$ is the local component of $\sigma'$ at $v$.
Since this is true for almost all $v$,
It follows from the rigidity, or strong multiplicity one, theorem for $\GL(2)$ that $\tilde{\rho}'$ is $\sigma'$-invariant.

By Proposition \ref{prop:u2gl2globalpacket}, 
the $\GL(2, \mbb{A}_{E'})$-module $\tilde{\rho}'$ is either a $b_s'$ or a $b_u'$-lift from $\mb{\rH}$, but not both.
Suppose the second case holds.  By Lemma \ref{lemma:u2cd} 
the packet $\{\rho\}$ must $b_u$-lift to an irreducible automorphic representation $\tilde{\rho}_u$ of $\GL(2, \Ae)$,
which base-change lifts to $\tilde{\rho}'$.  By Theorem III.3.1 of \cite{AC}, $\tilde{\rho}_u$ 
must be equal to $\ve_{E'/E}^m \,\tilde{\rho}$ for some integer $m$, 
where $\ve_{E'/E}$ is the character of $\idc{E}$ associated with the odd degree field extension $E'/E$ via global class field theory.
On the other hand, we have $\tilde{\rho}_u = \kappa\tilde{\rho}$ (see Section \ref{sec:summaryF}).
So, $\tilde{\rho} \cong \ve_{E/F}^{-m}\kappa\tilde{\rho}$.  The character $\ve_{E'/E}^{-m}\kappa$ has order different from $2$.
This is because $\ve_{E'/E}^{-m}$ has odd order, 
while the order of $\kappa$ is even or undefined, for $\kappa$ restricts to a nontrivial quadratic character on $\idc{F}$.
This implies that the equivalent irreducible $\GL(2, \Ae)$-modules $\tilde{\rho}$ and $\ve_{E/F}^m\kappa\tilde{\rho}$ 
have different central characters, a contradiction.

Since $\tilde{\rho}'$ is cuspidal, by Proposition \ref{prop:u2gl2globalpacket} the packet $\{\rho'\}$ is stable.
\end{proof}

Recall that we let $S$ be a finite set of places of $F$,
outside of which are places where 
the number field extensions $E/F$, $F'/F$ are unramified.  
The automorphic representations which we study are assumed 
to have unramified local components at all the places outside of $S$.
\begin{lemma}\label{lemma:traceglobalunstable1}
{\rm (i)} The packets $\pi(\{\rho\})$ and $\{\rho\}$ weakly lift to the packet $\pi'(\{\rho'\})$ of $\mb{\rG}(\Af)$.

{\rm (ii)}
Let $\rf, f, f_H$ be matching test functions on $\mb{\rG}(\Af), \mb{G}(\Af), \mb{H}(\Af)$, 
respectively, whose local components at the places outside of $S$ are spherical. 
The following equation holds$:$
\begin{equation}\label{eq:traceglobalunstable1}
2 \sum_{\pi' \in \pi'(\{\rho'\})}\!\!\!
m(\pi') \ep(\pi') \la \pi', \rf\times \beta\ra_S 
= \la \pi(\{\rho\}), f\ra_S +
\la \{\rho\}, f_H\ra_S,
\end{equation}
where $m(\pi')$ is the multiplicity with which $\pi'$ contributes to the discrete spectrum of $\mb{\rG}(\Af)$,
and $\ep(\pi')$ is an $n$-th root of unity.
\end{lemma}
\begin{proof}
Part (i) follows from the commutativity of the $L$-group diagram in Figure \ref{fig:lgroups}.

For part (ii):
Apply equation \eqref{eq:sepevp} to the packet $\{\pi'\} = \pi'(\{\rho'\})$.
By the commutativity of the $L$-group diagram,
we know that $\{\rho\}$ and $\pi(\{\rho\})$ contribute to the equation. We need to show that 
no other packet of $\mb{G}(\Af)$ or $\mb{H}(\Af)$ also contributes.

To determine the packets $\{\tau\}$ (resp.\ $\{\varrho\}$) of $\mb{G}(\Af)$ (resp.\ $\mb{H}(\Af)$) which contribute to the equation,
note that any such packet must $b$-lift (resp.\ $b\circ e$-lift)
to a $\sigma$-invariant automorphic representation $\tilde{\tau}$ of $\GL(3, \mbb{A}_{E})$, 
which in turn base-change lifts to $\tilde{\pi}'$ via $b_{3}$.  
Therefore, $\tilde{\tau}$ is a parabolically induced representation of the form
\[
I_{i, j} := I_{(2, 1)}(\ve^i\kappa\tilde{\rho}, \ve^j\tilde{\eta}), \quad i, j \in \{0, 1, 2,\ldots, n-1\},
\]
where $\ve$ is the character of $C_E$ associated with the number field extension $E'/E$ via global class field theory 
(\cite[Chap.\ 3 Thm.\ 3.1]{AC}).
Suppose $j \neq 0$.
Then, $I_{i, j}$ is $\sigma$-invariant if and only if 
$\ve(z \alpha (z)) = 1$ for all $z \in C_E$.
Since the Galois action of $\alpha$ commutes with that of $\beta$, we have $\ve = \ve \circ \alpha$.
So, $\ve(z \alpha (z)) = 1$ implies that $\ve^2 = 1$.
On the other hand, 
$\ve$ is the character associated with the nontrivial cyclic number field extension $E'/E$ of the odd degree $n$; thus,
$\ve^n = 1$.  These two facts together imply that $\ve$ is trivial, a contradiction.  Hence, $j = 0$.
If $\tilde{\rho} \ncong \ve^i\tilde{\rho}$, 
then by the same reasoning the $\sigma$-invariance of $\tilde{\tau}$ implies that $i = 0$.
The same type of argument applies to $\{\varrho\}$, and
we conclude that 
$\pi(\{\rho\})$ and $\{\rho\}$ are the only packets which contribute to the right-hand side of 
equation \eqref{eq:sepevp}, applied to $\pi'(\{\rho'\})$.
\end{proof}

We now consider the unstable packets of $\mb{H}(\Af) = \Un(2, E/F)(\Af)$.
Fix three characters $\theta_1$, $\theta_2$ and $\theta_3$ of $C_E^F$, 
with $\theta_1\theta_2\theta_3 = \omega$ and $\theta_1 \neq \theta_2$.
For a character $\theta$ of $C_E^F$, let $\theta'$ denote the character $\theta\circ\N_{E'/E}$ of $C_{E'}^{F'}$.
By Lemma \ref{lemma:globalnormindex} in the Appendix, 
the correspondence $\theta \mapsto \theta'$ is one-to-one.
Let $\rho(\theta_i, \theta_j)$ (resp.\ $\rho'(\theta_i', \theta_j')$)
be the packet of automorphic representations of $\mb{H}(\Af)$ (resp.\ $\mb{\rH}(\Af)$) 
associated with $(\theta_i, \theta_j)$ (resp.\ $(\theta_i', \theta_j')$), for $i, j \in \{1, 2, 3\}$.
If the characters $\theta_i$ are distinct, then all three packets in 
$\mc{Q} = \{\rho(\theta_i, \theta_j): i \neq j \in \{1, 2, 3\}\}$ are discrete spectrum packets.
Otherwise, if $\theta_1 \neq \theta_2 = \theta_3$, then $\rho(\theta_2, \theta_3)$ 
consists of the irreducible constituents of the parabolically induced representation 
$I(\tilde{\theta_2}\kappa, \omega/\tilde{\theta_2}\kappa)$, 
and the other two members of $\mc{Q}$ are discrete spectrum packets (\cite[p.\ 699-700]{FU2}).  
Here, $\tilde{\theta_2}(z) := \theta_2(z/\alpha (z))$ for all $z \in C_E$.
For all $i \neq j \in \{1, 2, 3\}$,
the packets of $\mb{G}(\Af)$ which are the $e$-lifts of $\rho(\theta_i,\theta_j)$ 
are equivalent to one another.
Let $\pi(\rho)$ (resp.\ $\pi'(\rho')$) denote the packet of $\mb{G}(\Af)$ (resp.\ $\mb{\rG}(\Af)$)
which is the lift of $\rho(\theta_1, \theta_2)$ (resp.\ $\rho'(\theta_1', \theta_2')$).

Suppose the characters $\theta_1, \theta_2, \theta_3$ are distinct. 
In particular, $\pi(\rho)$ and $\pi'(\rho')$ are discrete spectrum packets.  The following lemma holds:
\begin{lemma}\label{lemma:traceglobaluunstable}
{\rm (i)} The packets $\rho(\theta_1, \theta_2)$, $\rho(\theta_1, \theta_3)$, $\rho(\theta_2, \theta_3)$ and $\pi(\rho)$
weakly lift to the packet $\pi'(\rho')$ of $\mb{\rG}(\Af)$.

{\rm (ii)} The following equation holds for matching test functions$:$
\begin{multline}\label{eq:traceglobaluunstable}
4 \sum_{\pi' \in \pi'(\rho')}\!\!\!
m(\pi')\ep(\pi')\la \pi', \rf\times \beta\ra_S
= 
\la \pi(\rho), f \ra_S \\
+
\la \rho(\theta_1, \theta_2), f_H\ra_S 
+ \la \rho(\theta_1, \theta_3), f_H\ra_S 
+ \la \rho(\theta_2, \theta_3), f_H\ra_S,
\end{multline}
where $m(\pi')$ is the multiplicity with which $\pi'$ contributes to the discrete spectrum of $\mb{\rG}(\Af)$,
and $\ep(\pi')$ is an $n$-th root of unity.
\end{lemma}
\begin{proof}
This follows from the commutativity of the $L$-group diagram in Figure \ref{fig:lgroups},
the properties of base change for $\GL(2)$, 
the results recorded in Section \ref{sec:DORstf},
and equation \eqref{eq:sepevp}. 
\end{proof}

Suppose $\theta_1 \neq \theta_2 = \theta_3$.  In this case,
the packet $\pi(\rho)$ (resp.\ $\pi'(\rho')$)
consists of the irreducible constituents of:
\[
I\lp\tilde{\theta}_2, \omega/\tilde{\theta}_2\rp
\quad \lp \text{resp.\ } 
I\lp\tilde{\theta}'_2, \omega'/\tilde{\theta}'_2\rp \rp
\]
(see Section \cite[Part.\ 2. III.3.8]{F}).
\begin{lemma}
The following equation holds for matching functions$:$
\begin{equation}
\la \pi'(\rho'), \rf\times \beta\ra_S
= 
\Big\langle \pi(\rho(\theta_1, \theta_2)), f \Big\rangle_S\,.
\end{equation}
\end{lemma}
\begin{proof}
The representations in this case are parabolically induced,
so the equation will follow from the local character identity in 
Lemma \ref{lemma:tracelocalinduced}.
\end{proof}

Let $\mu$ be a character of $C_E^F$, and let $\mu'$ denote the character $\mu\circ\N_{E'/E}$ of $C_{E'}^{F'}$.
Then, $\mu$ (resp.\ $\mu'$) defines a one-dimensional automorphic representation of $\mb{H}(\Af)$ (resp.\ $\mb{\rH}(\Af)$)
via composition with the determinant.  
Let $\pi(\mu)$ (resp.\ $\pi'(\mu'))$ denote the quasi-packet of $\mb{G}(\Af)$ (resp.\ $\mb{\rG}(\Af)$) 
which is the lift of $\mu$ (resp.\ $\mu'$) (see \cite[Part 2.\ Sec.\ III.3.2 Cor.]{F}).
For each nonarchimedean place $v$ of $F$, and local test function $f_v$ on $G_v$, let
\[
\la \pi(\mu)_v, f_v\ra = \la\pi(\mu_v), f_v\ra := \la \pi(\mu_v)^\times, f_v\ra - \la\pi(\mu_v)^-, f_v\ra,
\]
where $\pi(\mu_v)^\times$ and $\pi(\mu_v)^-$ are the nontempered and cuspidal members, respectively, 
of the local quasi-packet $\pi(\mu_v)$.
Using the same type of argument as before, we obtain:
\begin{lemma}\label{lemma:traceglobalonedim}
{\rm (i)} The one-dimensional representation $\mu$ and quasi-packet $\pi(\mu)$ weakly lift to the quasi-packet 
$\{\pi'(\mu')\}$ of $\mb{\rG}(\Af)$.

{\rm (ii)} The following trace identity holds for matching test functions$:$
\begin{gather}
2 \sum_{\pi' \in \{\pi'(\mu')\}}
\ep(\pi')\la \pi', \rf\times\beta\ra_S
=
\la \pi(\mu), f\ra_S + \la \mu, f_H \ra_S,
\end{gather}
where $\ep(\pi')$ is an $n$-th root of unity.
\end{lemma}

\subsubsection{Stable Packets}\label{sec:traceglobalstable}
Let $\{\pi\}$ be a stable discrete spectrum (quasi-) packet of automorphic representations of $\mb{G}(\Af)$.
It lifts to a $\sigma$-invariant discrete spectrum representation $\tilde{\pi}$
of $\GL(3, \Ae)$ via the $L$-homomorphism $b$ (see Figure \ref{fig:lgroups}).  
Let $\tilde{\pi}'$ be the base change via $b_{3}$ of $\tilde{\pi}$ to $\GL(3, \mbb{A}_{E'})$.
Via the same argument used in the proof of Lemma \ref{lemma:sigmainvariance2},
it can be shown that:
\begin{lemma}\label{lemma:sigmainvariance3}
$\tilde{\pi}'$ is $\sigma'$-invariant.
\end{lemma}
We claim that there exists a packet $\{\pi'\}$ of $\mb{\rG}(\Af)$ which $b'$-lifts to $\tilde{\pi}'$.
There are two cases to consider:

\quad

\noindent
{\bf Case 1.} Suppose $n = [E':E]$ is not equal to $3$, or $n = 3$ 
and $\tilde{\pi}$ is not the cuspidal monomial representation associated with a character of $C_{E'}$.
Then, $\tilde{\pi}'$ is cuspidal by \cite[Chap. 3 Thm.\ 4.2(a)]{AC}.
Consequently, by Theorem \ref{thm:u3gl3stableglobalpacket} and Lemma \ref{lemma:sigmainvariance3}
there exists a {\it stable} packet $\{\pi'\}$ on $\mb{\rG}(\Af)$ which $b'$-lifts to $\tilde{\pi}'$.

\quad

\noindent
{\bf Case 2.} Suppose $n = 3$,
and $\tilde{\pi}$ is the cuspidal monomial representation ${\pi}(\chi)$ 
associated with a character $\chi$ of $C_{E'}$ such that $\chi\neq\beta\chi$.
The central character of ${\pi}(\chi)$ is $\chi|_{C_E}$.
Since $\tilde{\pi}$ is the $b$-lift of a representation of $\mb{G}(\Af)$, 
its central character $\omega_{\tilde{\pi}}$ is trivial on $C_F$.  
Hence, $\chi|_{C_F} = \omega_{\tilde{\pi}}|_{C_F} = 1$.

The representation $\tilde{\pi}$ base-change lifts to the representation
$\tilde{\pi}' := I(\chi, \beta\chi, \beta^2\chi)$ of $\GL(3, \mbb{A}_{E'})$ 
parabolically induced from the upper triangular Borel subgroup.  
Since $\tilde{\pi}'$ is $\sigma'$-invariant, we have 
$\{\chi, \beta\chi, \beta^2\chi\} = \{\alpha\chi^{-1},\alpha\beta\chi^{-1},\alpha\beta^2\chi^{-1}\}$.
Suppose $\chi = \alpha\beta^i\chi^{-1}$ for $i = 1$ or $2$.  Applying $\alpha\beta^i$ to both sides of the equation,
and noting that $\alpha\beta = \beta\alpha$, $\alpha^2 = 1$, we have
$
\chi^{-1} = \beta^{2i} \chi^{-1},
$
which implies, since $\beta^3 = 1$, that $\chi$ is $\beta$-invariant, a contradiction.

Hence, $\chi = \alpha\chi^{-1}$, and therefore $\chi|_{C_{F'}} = 1$ or $\ve_{E'/F'}$, 
the quadratic character of $C_{F'}$ associated with the field extension $E'/F'$ via global class field theory.
It follows from an exercise in class field theory that $\ve_{E'/F'}$ agrees with $\ve_{E/F}$ on $C_F$,
which rules out the possibility that $\chi|_{C_{F'}} = \ve_{E'/F'}$, for we have already shown that $\chi|_{C_{F}} = 1$.
Thus, $\chi|_{C_{F'}}$ is trivial, so there exists a character ${\theta}$ of $C_{E'}^{F'}$ such that
$\chi(z) = \tilde{\theta}(z) := {\theta}(z/\alpha(z))$, $z \in C_{E'}$.

Let $\kappa''$ denote the character which is the restriction of $\kappa'^{-1}$ to $C_{E'}^{F'}$.
Let $\{\pi'\}$ be the {\it unstable} packet on $\mb{\rG}(\Af) = \Un(3, E'/F')(\mbb{A}_{F'})$ which is the 
$e'$-weak-lift of the unstable packet $\rho'(\kappa''{\theta}, \kappa''\beta{\theta})$ of $\Un(2, E'/F')(\mbb{A}_{F'})$.
It follows from Proposition \ref{prop:u2u3gl3globalpacket}
and the commutativity of the $L$-group diagram in Figure \ref{fig:lgroups}
that $\{\pi'\}$ lifts to $\tilde{\pi}'$.

If $n = [E':E]$ is not $3$, then by \cite[Chap.\ 3.\ Thm.\ 3.1]{AC} the fibre of the base-change lifting $b_3$
to $\tilde{\pi}'$ consists of $\tilde{\pi}$ alone.  
If $n = 3$, then $\ve^i\tilde{\pi}$ lifts to $\tilde{\pi}'$ for each $i = 0, 1, 2$, where
$\ve$ is the character of $C_E$ associated with the number field extension $E'/E$ via global class field theory.
By the same reasoning used in the proof of Lemma \ref{lemma:traceglobalunstable1}, 
the representation $\ve^i\tilde{\pi}$ is $\sigma$-invariant if and only if $i = 0$ or $\tilde{\pi} \cong \ve\tilde{\pi}$.  
Thus, $\{\pi\}$ is the only (quasi-) packet which $b_{3}\circ b$-lifts to $\tilde{\pi}'$,
which in turn implies that it is the only (quasi-) packet which $b_G$-weakly-lifts to $\{\pi'\}$.
Since $\{\pi\}$ lifts via $b$ to a $\sigma$-invariant discrete spectrum representation of $\GL(3, \Ae)$,
no automorphic representation of $\mb{H}(\Af)$ lifts to $\{\pi\}$ (see Proposition \ref{prop:u2u3gl3globalpacket} in the Appendix). 
Consequently, in both Cases 1 and 2,
no automorphic representation of $\mb{H}(\Af)$ $e_H$-weakly-lifts to $\{\pi'\}$,
by the commutativity of the $L$-group diagram in Figure \ref{fig:lgroups}.

We write $\{\pi\} = \{\pi(\tilde{\pi})\}$ to emphasize that $\pi$ lifts (via $b$) to $\tilde{\pi}$, 
and likewise write $\{\pi'\} = \{\pi'(\tilde{\pi}')\}$.
\begin{lemma}\label{lemma:traceglobalstable}
{\rm (i)} The packet $\{\pi(\tilde{\pi})\}$ weakly lifts to $\{\pi'(\tilde{\pi}')\}$.

{\rm (ii)} For matching functions, we have$:$
\begin{equation}\label{eq:traceglobalstable}
\sum_{\pi' \in \{\pi'(\tilde{\pi}')\}} m(\pi')\ep(\pi') \la \pi', \rf\times\beta\ra_S =
\la \{\pi(\tilde{\pi})\}, f\ra_S,
\end{equation}
where $m(\pi')$ is the multiplicity with which $\pi'$ occurs in the discrete spectrum of $\mb{\rG}(\Af)$,
and $\ep(\pi')$ is an $n$-th root of unity.
\end{lemma}
\begin{proof}
Part (i) follows from the remarks preceding the lemma.  Part (ii) follows from equation \eqref{eq:sepevp}.
\end{proof}

Summarizing parts (i) of Lemmas \ref{lemma:traceglobalunstable1}, \ref{lemma:traceglobaluunstable},
\ref{lemma:traceglobalonedim} and \ref{lemma:traceglobalstable}, we have:
\begin{proposition}\label{prop:weaklift}
Every discrete spectrum $($quasi-$)$ packet of $\mb{G}(\Af)$, resp.\ $\mb{H}(\Af)$,
weakly lifts to a $($quasi-$)$ packet of $\mb{\rG}(\Af)$ via the $L$-homomorphism $b_G$, resp.\ $e_H$.
\end{proposition}
\begin{remark}
Note that the proposition is a statement on weak lifting as defined in Section \ref{sec:globaltf}.
It follows from the commutativity of the $L$-group diagram in Figure \ref{fig:lgroups}, 
and it says nothing regarding the $\beta$-invariance of the members of the (quasi-) packets of $\mb{\rG}(\Af)$
which are $b_G$ or $e_H$-weak-lifts. 
That issue shall be addressed in Proposition \ref{prop:classification2}.
\end{remark}

\section{Local Character Identities}\label{sec:localtraceidentities}
We now deduce local twisted character identities  
from the global identities which we have established in Section \ref{sec:globaltf}.
Let $k$ be a local $p$-adic field with {\it odd} residual characteristic, 
$\K$ a quadratic extension of $k$, and $k'$ a degree $n$ cyclic extension of $k$, with $n$
odd.  Let $\K'$ be the compositum field $\K k'$, and  let $\alpha$, $\beta$ be generators of $\Gal(\K'/k)$ 
such that $k' = {\K'}^\alpha$ and $\K = {\K'}^\beta$, the subfields fixed by $\alpha$ and $\beta$.

The multiplicity one theorem for $\Un(3)$ 
has been proved only for those automorphic representations each of whose dyadic local components belongs
to a local packet containing a constituent of a parabolically induced representation (\cite{F}).
Consequently, we are able to establish our main results, i.e.\ Propositions 
\ref{prop:traceglobaldiscrete1},
\ref{prop:tracelocalunstable},
\ref{prop:uustablecaseii} and
\ref{prop:n3case},
only when the residual characteristic of $k$ is odd.  This restriction may be removed once the multiplicity one theorem for 
$\Un(3)$ is available in full generality.

Let:
\begin{center}
\begin{tabular}{cll}
$\bullet$ & $\K^k = \{x \in \K^\times : \N_{\K/k}x = 1\}$, &\quad $\K'^{k'} = \{x \in \K'^\times: \N_{\K'/k'}x = 1\}$;\\
$\bullet$ & $G = \Un(3, \K/k)$, & \quad $\rG = \Un(3, \K'/k')$; \\
$\bullet$ & $H = \Un(2, \K/k) \times \K^k$, & \quad $\rH = \Un(2, \K'/k') \times \K'^{k'}$.
\end{tabular}
\end{center}
Here, for example, $\Un(3, \K/k)$ denotes a group of $k$-points rather than an algebraic group scheme, which
was the convention in Section \ref{sec:globalcase}.

In the global case, associated with the $L$-homomorphism $e: {}^L H \rightarrow {}^L G$ is a character $\kappa$ of $C_E$
such that $\kappa\alpha\kappa = 1$ (see Section \ref{sec:summaryFglobal} of the Appendix).
Likewise, associated with 
$e' : {}^L \rH \rightarrow {}^L \rG$
is the character $\kappa' := \kappa \circ \N_{E'/E}$ of $C_{E'}$.
We let $\kappa$ and $\kappa'$ denote also the corresponding local characters of $\K^\times$ and $\K'^\times$, respectively.

We fix once and for all a character $\omega'$ of $\K'^{k'}$ which  
satisfies $\omega' = \beta\omega'$.
By Lemmas \ref{lemma:localcentralchars} and \ref{lemma:normindex} in the Appendix, 
$\omega'$ uniquely determines a character $\omega$ of $\K^k$ such that $\omega' = \omega \circ \N_{\K'/\K}$.
We consider only those representations of $\rG$ (resp.\ $G$, $H$) 
whose central characters are equal to $\omega'$ (resp.\ $\omega$).
Under this condition, a representation of $H$ (resp.\ $\rH$) is uniquely determined by its $\Un(2)$-component.  Hence,
we often abuse notation and put $H := \Un(2, \K/k)$, $\rH := \Un(2, \K'/k')$.

Let $\rf, f, f_H$ denote arbitrary smooth, compactly supported mod center functions on $\rG, G, H$, 
respectively, with matching orbital integrals (\cite[p.\ 71]{KS}).
As mentioned earlier, their existence follows from the work of Waldspurger 
(\cite{W}, \cite{Wchar}, \cite{Watordue}) and Ngo (\cite{Ngo}).

Assume that the function $\rf$ transforms under the center $Z(\rG)$ of $\rG$ via $\omega'^{-1}$, 
i.e. $\rf(z g) = \omega'^{-1}(z) \rf(g)$ for all $z \in Z(\rG)$, $g \in \rG$.
Then, by the matching condition, $f$ and $f_H$ transform under the centers of the groups on which
they are defined via $\omega^{-1}$.
For an irreducible admissible representation $\pi'$ of $\rG$, 
and $A$ an intertwining operator in $\Hom_{\rG}(\pi', \beta\pi')$,
put $\la \pi', \rf\ra_A := \tr \pi'(\rf)A$.

For our purpose in this work, it suffices to consider test functions whose orbital integrals
are zero on a neighborhood of each singular element.
We say that an element $t \in \rG$ is {\bf $\beta$-regular} if the norm of $t$ in $G$ is regular,
and we let $\rG^{\beta\text{-reg}}$ denote the subset in $\rG$ of $\beta$-regular elements.
The following lemma is well-known.  
\begin{lemma}\label{lemma:nonvanishing}
Let $M = G$ or $H$.
Given any test function $f_M$ on $M$ whose orbital integral is zero on a neighborhood of each singular element,
there exists a test function $f$ on $\rG$ such that $f$ and $f_M$ are matching functions.
\end{lemma}
\begin{proof}
Defined in \cite[Sec.\ 2]{Alcr} is an adjoint transfer factor $\Delta_{\rG, M}$ on $\rG^{\beta\text{-}{\rm reg}} \times M^\reg$.
(Or rather, we define $\Delta_{\rG, M}$ in terms of the transfer factor $\Delta_{M, \rG}$ 
on $M^\reg \times \rG^{\beta\text{-}{\rm reg}}$ in virtually the same way as in \cite[Eq.\ (2.3)]{Alcr}.)
For each $t \in \rG^{\beta\text{-}{\rm reg}}$, let:
\[
F(t) =
 \sum_{t_M \in M^\reg} \Delta_{G, M}(t, t_M) SO_{f_M}(t_M), 
\]
where the sum is over representatives of the regular conjugacy classes of $M$, 
and $SO_{f_M}(t_M)$ denotes the stable orbital integral of $f_M$ at $t_M$.
We claim that $F(t)$ may be realized as a $\beta$-twisted orbital integral; namely, there exists a smooth compactly
supported modulo center function $f$ on $\rG$ such that:
\[
F(t) = O_f(t\beta) := \int_{Z(\rG)\rG_{t\beta}\bs\rG} f(g^{-1}t\beta(g))dg
\]
for all $t \in \rG^{\beta\text{-}{\rm reg}}$, 
where $\rG_{t\beta} := \{g \in \rG : g^{-1} t\beta(g) = t\}$, the $\beta$-twisted centralizer of $t$ in $\rG$.
Once that is shown, it follows from the ``inversion formula'' (2.5) in \cite{Alcr}, extended to the twisted case, 
that $f$ and $f_M$ are in fact matching functions (see \cite[p.\ 528-529]{Alcr}).  That is, the stable orbital integral of $f_M$ 
at each element $t_M \in M^\reg$ is equal to the so-called ``$\kappa$-orbital integral'' of $f$ 
computed at the stable $\beta$-conjugacy class of an element $t \in G^{\beta\text{-}{\rm reg}}$ whose norm is $t_M$,
where $\kappa$ is a character of the finite group parameterizing the $\beta$-conjugacy classes within the stable 
$\beta$-conjugacy class of $t$ (not to be confused with the $\kappa$ associated with the lifting from 
$\Un(2, E/F)$ to $\rR_{E/F}\GL(2)$ introduced earlier).  
In this work we contend ourselves with claiming that it is reasonable to assume that the inversion formula extends to the twisted case,
without proving it.

The factor $\Delta_{G, M}(t, t_M)$ is zero unless $t_M$ is a norm of $t$.  
Hence, it suffices to consider only those elements in $\rG^{\beta\text{-}{\rm reg}}$ whose norms in $M$ are regular.
Let $t \in \rG$ be such an element.
In our context of base change for $\Un(3)$, $t$ has a norm in $G$ which is regular, even when $M = H$.
By \cite[Lemma 3.2.A]{KS}, the $\beta$-conjugacy class of $t$ meets a $\beta$-invariant maximal torus in $\rG$.
So, without loss of generality we may assume that $t$ belongs to a $\beta$-invariant maximal torus $T'$ of $\rG$.
In particular, $T'$ is the centralizer in $\rG$ of the norm $\N t$ of $t$ in $G \subset \rG$, and
$\rG_{t\beta}$ is equal to the centralizer of $\N t$ in $G$.  We also have $\rG_{t\beta} = T'^\beta$,
the subgroup of $\beta$-fixed elements in $T'$.

Let $\{T'_i\}$ be a set of representatives for the $\beta$-conjugacy classes of the $\beta$-invariant maximal tori of $\rG$.
For each $T'_i$, put
\[
N_{\rG}(T'_i\beta) := \{g \in \rG : g^{-1}T'_i\beta(g) = T'_i\},
\]
the $\beta$-normalizer of $T'_i$ in $\rG$.
Fix a smooth function $a_{T'_i}$ on $\rG$, compactly supported modulo $Z(\rG)N_{\rG}(T'_i\beta)$,
such that $a_{T'_i}({hg}) = a_{T'_i}({g})$ for all $g \in \rG$, $h \in Z(\rG)N_{\rG}(T'_i\beta)$, and
$\int_{Z(\rG)T'^\beta_i\bs \rG} a_{T'_i}(g) dg = 1$.  
For each $h \in \rG$, let:
\[
f(h) =
\begin{cases}
F(t)a_{T'_i}(g) & \text{ if } h = g^{-1} t \beta(g) \text{ for some }t\in T'_i,\, g  \in \rG, \\
0 & \text{otherwise.}
\end{cases}
\]
The function $f$ is well-defined, 
since we have made the heavy assumption that the orbital integral of $f_M$ is zero on a neighborhood of the singular set,
so that we can separate nonconjugate tori and need not study asymptotic expansions.
For each $t \in \rG^{\beta\text{-}{\rm reg}}$ with a regular norm in $G$, we have:
\[O_f(t\beta) = \int_{Z(\rG)\rG_{t\beta}\bs \rG} f(g^{-1}t \beta(g))dg
=
\int_{Z(\rG)T'^\beta_t\bs \rG} F(t)a_{T'_t}(g) dg = F(t)\] 
after a suitable change of variable; $T'_t$ is $T'_i$ for a suitable $i$ depending on $t$. 
\end{proof}

\noindent

Unless otherwise noted, all representations studied in this section are irreducible and smooth admissible.

\subsection{Classification of Local Packets for $\Un(3)$}
We now give a summary of the classification of the local (quasi-) packets of admissible representations of $\Un(3)$.  
All results recorded in this section are due to \cite{FU2}, \cite{F}. 
The local packets are defined in terms of the local character identities which they satisfy.
For an element $g$ in $\GL(m, \K)$ ($m = 2, 3$), let $\sigma (g) = J\alpha({}^t g^{-1}) J^{-1}$,
where $J$ is $\lp\lsm &&1\\&-1&\\1&&\rsm\rp$ if $m = 3$, $\lp\lsm &1\\-1&\rsm\rp$ if $m = 2$.
We say that a representation $(\pi, V)$ of $\GL(m, \K)$ is 
$\sigma$-invariant if $(\pi, V) \cong (\sigma \pi, V)$, where $(\sigma \pi, V)$ is the
$\GL(m, \K)$-module defined by
$
\sigma \pi : g \mapsto \pi(\sigma (g))$, $\forall \, g \in \GL(m, \K).
$

Each $\sigma$-invariant square-integrable representation $\tilde{\pi}$ of $\GL(3, \K)$ 
is the lift of a local packet consisting of a single square-integrable representation $\pi(\tilde{\pi})$ of $G$.
Each pair $(\theta_1, \theta_2)$ of characters of $\K^k$ lifts to a local packet
$\rho(\theta_1, \theta_2)$ which consists of two irreducible representations $\pi^+$, $\pi^-$ of $H$.  They are
cuspidal if $\theta_1 \neq \theta_2$.  If an irreducible representation $\rho$ of $H$ does not lie in a local packet of
the form $\rho(\theta_1, \theta_2)$, then $\rho$ belongs to a local packet consisting of itself alone.
A local packet $\{\rho\}$ of $H$ lifts to a local packet $\pi(\{\rho\})$ of $G$ which has cardinality
$\abs{\pi(\{\rho\})} = 2\abs{\{\rho\}}$.

A more detailed summary is given in Section \ref{sec:summaryF} of the Appendix.

\subsection{Parabolically Induced Representations}\label{sec:tracelocalinduced}
We now state the twisted character identities for the parabolically induced representations of $\rG$.
Let $\mu$ be a character of $\K^\times$, and $\eta$ a character of $\K^k$. 
Let $\mu\otimes\eta$ denote the following representation of the upper triangular parabolic subgroup $P$ of $G$
:
\[
\mu\otimes\eta : 
p = \lp\lsm a & * & * \\ & b &* \\ && \alpha (a)^{-1}\rsm\rp \mapsto \mu(a)\eta(b), 
\quad p \in P.
\]
Let $I_G(\mu, \eta)$ denote the representation of $G$ parabolically induced from $\mu\otimes\eta$, with normalization.
We extend this notation to $\rG$.
The character $\mu$ defines a representation of the upper triangular parabolic subgroup of $H = \Un(2, \K/k)$ via
\[
\mu\lp\lp\lsm a & *\\&\alpha (a)^{-1}\rsm\rp\rp = \mu(a), \quad a \in \K^\times.
\]  
Let $I_H(\mu)$ denote the representation
of $H$ parabolically induced from $\mu$, with normalization.

Let $\mu'$ denote the character $\mu\circ\N_{\K'/\K}$ of $\K'^\times$ , 
$\eta'$ the character $\eta\circ\N_{\K'/\K}$ of $\K'^{k'}$. 
For a smooth function $\phi$ in the space of $I' :=  I_{\rG}(\mu', \eta')$, let
\[
(A\phi)(g) = \phi(\beta (g)),\quad g \in \rG.
\]
Then, $A$ is an intertwining operator in $\Hom_{G'}(I', \beta I')$.
\begin{lemma}\label{lemma:tracelocalinduced}
The following character identities hold for all matching functions$:$
\begin{gather}
\la I', \rf\ra_A
= \la I_G(\mu, \eta), f\ra 
= \la I_H(\mu \kappa^{-1}), f_H\ra.
\end{gather}
\end{lemma}
\begin{proof}
This follows from a standard computation of the trace of an induced representation, 
using the Weyl integration formula and its twisted form.
\end{proof}

\subsection{Cuspidal Representations}\label{sec:localcuspidal}
In this section, we establish local character relations for cuspidal representations.
Many of the proofs in this section rely on the construction of global data whose local components
coincide with the given local data.  We describe once and for all our notation here:
For a given positive integer $\n$, we construct a system of totally imaginary number fields $F$, $F'$, $E$, $E'$ such that:
$[E:F] = 2$, $F'/F$ is cyclic of degree $n$, $E' = F'E$; there is a place $w$ of $F$, which stays prime in $E$ and $F'$, 
such that $F_w = k$, $E_w = \K$, $F'_w = k'$, $E'_w = \K'$;
and there are $\n$ places $w_1,\ldots,w_\n$ of $F$ not equal to $w$, 
which stay prime in $F'$, such that $F_{w_1} = F_{w_2} = \ldots = F_{w_\n}$.
By abuse of notation, we let $\alpha$ and $\beta$ denote also generators of $\Gal(E'/F)$,
with $F' = E'^\alpha$ and $E = E'^\beta$.
Recall that we put 
\[
C_F := F^\times\bs\Af^\times, \quad C_{E}^F := \Un(1, E/F)(F)\,\bs\,\Un(1, E/F)(\Af),
\]
and define $C_E, C_{F'}, C_{E'}, C_{E'}^{F'}$ similarly.

\subsubsection{Global Construction}
Let $\pi_0$ be an irreducible square-integrable representation of a connected, reductive $p$-adic group.  
By definition, $\pi_0$ embeds in a Hilbert space (\cite[p. 5]{HC}). 
We let $(\cdot, \cdot)$ denote the inner product on the space of $\pi_0$.
Given any unit vector $v$ in the space of $\pi_0$, we define the {\bf matrix coefficient} function
$f_{\pi_0}(g) := d(\pi_0)\ol{(\pi_0(g)v, v)}$ on the group, where $d(\pi_0)$ is the formal degree of $\pi_0$ (see \cite{HC}).
The function $f_{\pi_0}$ is smooth and square-integrable modulo the center of the group.
If $\pi_0$ is cuspidal, $f_{\pi_0}$ is compactly supported mod center,
and the trace $\tr\! \pi(f_{\pi_0})$ is well-defined for any admissible representation $\pi$ 
of the group with the same central character as $\pi_0$. 
For irreducible admissible representations $\pi_0$, $\pi$ of the group with the same central character, such that $\pi_0$ is cuspidal,
Harish-Chandra proved in \cite{HC} that the trace $\tr\! \pi(f_{\pi_0})$ is nonzero if and only if $\pi\cong \pi_0$, 
and $\tr\! \pi_0(f_{\pi_0}) = 1$.
If $\pi_0$ is square-integrable but not cuspidal, then $\tr\! \pi(f_{\pi_0})$ is not defined, 
since $f_{\pi_0}$ is not compactly supported mod center.

In \cite{K1}, D. Kazhdan proved the existence of {\bf pseudo-coefficients} for square-integrable representations.  
These are smooth, compactly supported modulo center functions which behave like the matrix coefficients of cuspidal representations.  
More precisely, 
a pseudo-coefficient $f_{\pi_0}$ of an irreducible square-integrable representation $\pi_0$ 
has the property that: For an irreducible admissible representation $\pi$ with the same central character as $\pi_0$,
we have $\tr\! \pi(f_{\pi_0}) \neq 0$ if and only if $\pi$ is an elliptic constituent of the parabolically induced representation $I$ 
of which $\pi_0$ is a subquotient.  
Here, we say that a representation is {\bf elliptic} if its character is nonzero on the elliptic regular set.
If $\pi$ is tempered, then $\tr\!\pi(f_{\pi_0})$ is equal to $1$ if $\pi_0\cong \pi$, and $0$ otherwise (\cite[Thm.\ K]{K1}).
Note that any matrix coefficient of a cuspidal representation is a pseudo-coefficient.

The existence of pseudo-coefficients also holds for nonconnected groups, see for example \cite[Chap.\ 1]{AC}, where the case of $\rR_{F'/F}\GL(n, F)\rtimes \Gal(F'/F)$ ($F'/F$ cyclic field extension) is discussed, or \cite[p.\ 197]{Frigid}.
Hence, for each $\beta$-invariant, irreducible, square-integrable representation $\pi'_0$ of $\rG$,
there exists a smooth, compactly supported modulo center function $f_{\pi'_0}$ on $\rG$ with the property that: 
For any $\beta$-invariant, irreducible, admissible representation $\pi'$ of $\rG$ with the same central character as $\pi'_0$, 
and nonzero intertwining operator $A' \in \Hom_{\rG}(\pi', \beta\pi')$, 
the twisted trace $\langle \pi', f_{\pi'_0}\rangle_{A'} := \tr\!\pi'(f_{\pi'_0})A'$ is nonzero
if and only if $\pi'$ is an elliptic constituent of the parabolically induced representation
of which $\pi'_0$ is a subquotient.
If $\pi'$ is tempered, then $\langle\pi', f_{\pi'_0}\rangle_{A'}$ is nonzero if and only if $\pi_0'\cong \pi'$.
We call $f_{\pi'_0}$ a {\bf $\beta$-twisted pseudo-coefficient} of $\pi'_0$.
If $\pi'_0$ is cuspidal, we may take its twisted pseudo-coefficient to be any of its matrix coefficients, 
defined as in the nontwisted case.  For simplicity, we normalize the twisted pseudo-coefficient 
$f_{\pi'}$ for each $\beta$-invariant square-integrable $\rG$-module $\pi'$, such that $\la \pi', f_{\pi'}\ra_{A(\pi')} = 1$,
where $A(\pi')$ is a nonzero intertwining operator in $\Hom_{\rG}(\pi', \beta\pi')$ fixed once and for all,
with $A(\pi')^n = 1$.
Note that twisted pseudo-coefficients in fact exist for a larger class of representations called $\beta$-discrete representations
(see \cite[Sec.\ 1.2.3]{AC}), which are not necessarily square-integrable,
but for simplicity we restrict our attention to only the $\beta$-invariant square-integrable representations of $\Un(3, \K'/k')$.

The following lemma enables us to construct $\beta$-invariant automorphic representations of $\mb{\rG}(\Af)$
with given $\beta$-invariant cuspidal and square-integrable local components.  
A priori, if $[k':k] > 1$, it is possible that no cuspidal representation of $\rG$ is $\beta$-invariant,
in which case the lemma is vacuous.
As we shall see, at least in the case where the residual characteristic of $k$ is odd,
$\beta$-invariant cuspidal representations of $\rG$ do exist by Proposition \ref{prop:traceglobaldiscrete1},
which is proved in part by applying the lemma to the case where the field extension $k'/k$ is trivial.
As shall be shown by the comments made before Proposition \ref{prop:tracelocalsteinberg}, 
there always exist $\beta$-invariant square-integrable representations of $\rG$, 
regardless of the residual characteristic of $k$.
\newcommand{\m}{m}
\begin{lemma}\label{lemma:globalconstruction}
Let $\m$ be a positive integer.
Let $\pi'_i$, $0 \leq i \leq \m$, be $\Gal(\K'/\K)$-invariant, irreducible, square-integrable $\rG$-modules.
Let $\pi'_0$ be moreover cuspidal.

There exist totally imaginary number fields $F, F', E, E'=F'E$, with $[E:F] = 2$, $F'/F$ cyclic of degree $n$, and a $\Gal(E'/E)$-invariant, irreducible, cuspidal, automorphic representation
$\pi'$ of \[\mb{\rG}(\Af) = (\rR_{F'/F}\Un(3, E/F))(\Af) = \Un(3, E'/F')(\mbb{A}_{F'}),\] 
such that$:$

\noindent $(i)$
There are $\m + 1$ nonarchimedean places $\{w_i\}_{i = 0}^\m$ of $F$ which stay prime in $F'$ and $E$,
such that$:$ For $0 \leq i \leq \m$, we have $F_{w_i} = k$, $F'_{w_i} = k'$, $E_{w_i} = \K$, $E'_{w_i} = \K'$.

\noindent $(ii)$ 
$\pi'_{w_i}$ $(0 \leq i \leq \m)$ is equivalent to $\pi'_{i}$$;$

\noindent $(iii)$ 
$\pi_v$ is unramified for each nonarchimedean place $v \neq {w_0}, w_1, \ldots, w_\m$
of $F$ at which $E/F$ is unramified.

\noindent $(iv)$ 
For each place $u$ outside of $\{w_i\}_{i = 0}^\m$ which is either an archimedean place or
a nonarchimedean place at which $E/F$ is ramified, $\pi'_u$ is an irreducible principal series representation.

\end{lemma}
\begin{proof}
The lemma follows from techniques used in \cite[p.\ 173.\ Prop III.3]{Frigid} and, for example, \cite{Fsym}.
We summarize the procedure as follows:  

For any finite prime $p$ and finite field extension $\mbb{Q}_p(\gamma)$ of $\mbb{Q}_p$,
Krasner's Lemma implies that if an element $c$ in $\mbb{\ol{Q}}\cap\mbb{Q}_p(\gamma)$ 
is sufficiently close to $\gamma$ in $\mbb{Q}_p(\gamma)$, then $\mbb{Q}(c)\mbb{Q}_p = \mbb{Q}_p(\gamma)$.
Thus, by the Chinese Remainder Theorem, for any finite collection of algebraic $p$-adic field extensions $\{\mbb{Q}_{p_i}(\gamma_i)\}$,
with $p_i$ distinct, there exists $c \in \mbb{\ol{Q}}$ such that $\mbb{Q}(c)\mbb{Q}_{p_i} = \mbb{Q}_{p_i}(\gamma_i)$ for all $i$.
Hence, there are totally imaginary number fields $F, F', E, E' = F'E$, with $[E:F] = 2$ and $F'/F$ cyclic of degree $n$,
which satisfy condition $(i)$ of the lemma.  Note that this construction is not unique.  
Let $\beta$ denote also the generator of $\Gal(E'/E)$,
which is the restriction to $E'$ of the fixed generator $\beta$ of $\Gal(\K'/\K)$.

We consider the $\beta$-twisted ``simple trace formula'' of $\Un(3, E'/F')$, or equivalently the simple trace formula
of the nonconnected group $\rR_{F'/F}\Un(3, E/F)\rtimes\la\beta\ra$ 
(see \cite{DKV}, \cite[Chap.\ IV.\ Sec.\ 3]{Frigid}, \cite[Chap.\ 1.\ Lemma 2.5]{AC}).
To apply it we must first fix a 
$\beta$-invariant character $\omega'$ of $\Un(1, E'/F')(\mbb{A}_{F'})$, trivial on $\Un(1, E'/F')(F')$,
such that the test function $\rf$ on $\mb{\rG}(\Af)$ in the formula transforms under the center 
$Z(\mb{\rG})(\Af) \cong \Un(1, E'/F')(\mbb{A}_{F'})$ of $\mb{\rG}(\Af)$ as follows:
\[
\rf(zg) = \omega'^{-1}(z)\rf(g), \quad g \in \mb{\rG}(\Af),\, z \in Z(\mb{\rG})(\Af).
\]
Moreover, we would like $\omega'_{w_i}$ to coincide with the central character of $\pi'_i$, for $i = 0, 1, \ldots, \m$.
In what follows, we identify $\Un(1, E'/F')$ with an $F$-group via restriction of scalars.

By assumption, the central characters $\omega'_i$ ($0 \leq i \leq \m$) of $\pi'_i$
are $\beta_{w_i}$-invariant.  Hence, by Lemma \ref{lemma:localcentralchars},
there are characters $\omega_i$ of the group $\Un(1, \K/k)$ such that $\omega'_i = \omega_i \circ\N_{\K/k}$.
We construct a global character $\omega$ of $C_E^F$ such that $\omega_{w_i} = \omega_i$ for $i = 0, 1,\ldots, \m$,
and $\omega_v$ is unramified for each nonarchimedean place $v \notin \{w_0, w_1,\ldots, w_\m\}$ at which $E/F$ is unramified.
This construction is possible because $C_{E}^{F}$ and $E_{w_i}^{F_{w_i}}$(elements in $E_{w_i}$ with norm $1$ in $F_{w_i}$), 
$0 \leq i \leq \m$, are compact in their respective topologies,
which allows us to construct $\omega$ using the Poisson Summation Formula for $\Un(1)$.

Let $\omega' = \omega \circ \N_{E'/E}$.  Then, $\omega'$ is a $\beta$-invariant character of $C_{E'}^{F'}$
with the property that $\omega'_{w_i} = \omega'_i$ for $i = 0, 1, \ldots, \m$,
and $\omega'_v$ is unramified for all nonarchimedean places $v \notin \{w_i\}_{i = 0}^{\m}$ of $F$ at which $E/F$ is unramified.

For simplicity, we now view $\mb{\rG}(\Af)$ as the group of $\mbb{A}_{F'}$-points of the $F'$-group $\Un(3, E'/F')$,
so that the local components of an irreducible automorphic representation of $\mb{\rG}(\Af)$ are indexed by the places of $F'$.
Let $V_{\rm ram}$ denote the set of nonarchimedean places of $F'$ outside of $\{w_i\}_{i = 0}^\m$ 
at which $E'/F'$ is ramified.  
Here, we identify the places $w_i$ of $F$ with places of $F'$, as we may since they are by assumption prime in $F'$.

We let the local component $\rf_{w_0}$ of the test function $\rf$ be a matrix coefficient of $\pi'_{w_0}$.
For $1 \leq i \leq \m$, we let $\rf_{w_i}$ be a twisted (with respect to $\beta_{w_i}$)
pseudo-coefficient of $\pi'_{i}$.
For each nonarchimedean place $v$ of $F'$ outside of $V_{\rm ram} \cup \{w_i\}_{i = 0}^\m$,
we let $\rf_v$ be a spherical function.
Note that the condition $\m \geq 1$ is needed to ensure that the simple trace formula holds for such $\rf$.

For each place $u \in V_{\rm ram}$,
fix a ramified  character $\mu_u$ of $F'^\times_u$ of odd order.
Let $\mu_u'$ denote the character $\mu_u\circ\N_{E'_u/F'_u}$ of $E'^\times_u$.
For all $x \in {F'_u}^\times$, we have $\mu_u'(x) = \mu_u^2(x)$.
Since the order of $\mu_u$ is odd and the ramified character $\mu_u$ is nontrivial on $\mc{O}_{F'_u}^\times$,
the restriction of $\mu_u'$ to $\mc{O}_{F'_u}^\times$ is nontrivial.
We choose $\rf_u$ such that its normalized $\beta_u$-twisted orbital integral $\Delta \Phi_\beta(\rf_u)$ 
is supported on the $\beta_u$-twisted conjugacy classes of elements of the form 
$z\, \diag(r\vp_u^j,\, x({r, j}),\, \alpha(r\vp_u^j)^{-1})$, 
where $z$ is an element in the center of $\rG_u := \Un(3, E'_u/F'_u)$, 
$r$ is a unit in the ring of integers $\mc{O}_{E'_u}$ of $E'_u$, 
$\vp_u$ is a fixed uniformizer of $\mc{O}_{E'_u}$, $j$ is a fixed positive integer, 
and $x({r, j}) := (\alpha(r)/r)(\alpha(\vp_u)/\vp_u)^j$.
Moreover, we require that $\Delta\Phi_\beta(\rf_u)$ takes the value $\omega'_u(z)(\mu_u'(r) + \mu_u'(\alpha(r)^{-1}))$ 
on the twisted conjugacy class of $z\, \diag(r\vp_u^j,\, x({r, j}),\, \alpha(r\vp_u^j)^{-1})$. 

Let $\mb{Z} \cong \Un(1, E'/F')$ be the center of $\Un(3, E'/F')$.  
Let ${\rm SU}(3, E'/F')$ be the subgroup of elements in $\Un(3, E'/F')$ with determinant $1$.  
Then, \[Z{\rm SU}(3, E'/F')_{w_i} := \mb{Z}(F'_{w_i}){\rm SU}(3, E'/F')(F'_{w_i})\]
is a subgroup of index $3$ in $\Un(3, E'/F')(F'_{w_i})$, for $i = 0, 1, \ldots, \m$.
The restriction of $\pi'_{w_i}$ ($0\leq i \leq \m$) to $Z{\rm SU}(3, E'/F')_{w_i}$ 
can be shown to be nonzero on developing endoscopy for ${\rm SU}(3)$ in analogy with the theory for ${\rm SL}(2)$.
Hence, the twisted orbital integrals of the matrix/pseudo-coefficients $f'_{w_i}$ do not vanish identically on 
$Z{\rm SU}(3, E'/F')_{w_i}$, nor on ${\rm SU}(3, E'/F')_{w_i}$ since they transform under $Z_{w_i}$
via multiplication with the fixed central character $\omega'_i$

By the weak approximation property of the simply connected algebraic $F'$-group ${\rm SU}(3, E'/F')$ (\cite[Prop.\ 7.9]{PR}),
there is a $g_0 \in \mb{Z}{\rm SU}(3, E'/F')(F')$ 
which lies simultaneously in the support of the twisted orbital integrals of all $f_{w_i}'$ ($0 \leq i \leq \m$),
and is arbitrarily close to an element of the form $z\, \diag(r\vp_u^j,\, x({r, j}),\, \alpha(r\vp_u^j)^{-1})$ (see above)
in $Z{\rm SU}(3, E'/F')_u$ for each $u \in V_{\rm ram}$.
Adjusting the support of the archimedean components of $\rf$ if necessary to ensure that the twisted orbital integral of $\rf$
is nonzero at only one elliptic $\beta$-regular element in $\Un(3, E'/F')(F')$,
the geometric side of the $\beta$-twisted simple trace formula of $\Un(3, E'/F')(\mbb{A}_{F'})$ is nonzero for this choice of $\rf$.
Hence, there must be a $\beta$-invariant, irreducible, cuspidal, automorphic representation $\pi'$ 
of $\Un(3, E'/F')(\mbb{A}_{F'})$ such that
$\la \pi', \rf\times\beta\ra \neq 0$.  Since $\rf_{w_0}$ is the matrix coefficient of the cuspidal representation $\pi'_0$,
we must have $\pi'_{w_0} \cong \pi'_0$.
This proves part $(ii)$ of the lemma for $i = 0$.
Part $(iii)$ of the lemma also follows because $\pi'_v$ must be unramified if 
$\rf_v$ is spherical and $\la\pi'_v, \rf_v\times\beta_v\ra \neq 0$.

We now consider part $(iv)$ of the lemma.  Since the number fields are totally imaginary, 
the archimedean local components of $\pi'$ must be irreducible  principal series representations of $\GL(3, \mbb{C})$.
Let $u$ be a place in $V_{\rm ram}$.
Given our choice of $\rf_u$,
by the same argument as in \cite[Sec.\ 1.5]{Fsym} which uses the Deligne-Casselman Theorem \cite{C}, 
$\la \pi'_u, \rf_u \times \beta_u\ra \neq 0$ implies that the representation $\pi'_u$ 
must be a constituent of a parabolically induced representation $I_{\rG_u}(\chi', \eta')$, 
such that $\chi'$ agrees with $\mu_u'$ on $\mc{O}_{E'_u}^\times$, and $\eta'$ is trivial on $\mc{O}_{E'_u}^\times \cap E'^{F'_u}_u$
(i.e. $\eta'$ is unramified).

To show that $I_{\rG_u}(\chi', \eta')$ is irreducible, 
we first recall the criteria for the reducibility of the parabolically induced representations of $\rG_u$, 
as recorded in \cite[Part 2.\ I.4.3]{F}: 
$I_{\rG_u}(\chi', \eta')$ is reducible if and only if one of the following conditions holds:
\newcounter{Rcount}
\\(1)
$\chi'$ is trivial on ${F'^\times_u}$, 
and $\chi' \neq \tilde{\eta}'$, where $\tilde{\eta}'$ is the character of $E'^\times_u$ defined by 
$\tilde{\eta}'(z) = \eta'(z/\alpha(z))$ for all $z \in E'^\times_u$.
In this case, $I_{\rG_u}(\chi', \eta')$ is the direct sum of two tempered representations.
\\(2)
$\chi' = \xi'\kappa'\nu'^{1/2}$, for some characters $\xi', \kappa'$ of $E'^\times_u$,
such that $\xi'|_{F'^\times_u} = 1$, and $\kappa'|_{F'^\times_u} = \ve_{E'_u/F'_u} = \ve_{E_u/F_u} \circ \N_{F_u'/F_u}$, 
where $\ve_{E_u/F_u}$ is the quadratic character associated with the local field extension $E_u/F_u$ 
via local class field theory.  Here, $\nu'$ is the normalized absolute value character of $E'^\times_u$.
The composition series of $I_{\rG_u}(\xi'\kappa'\nu'^{1/2}, \eta')$ has length two
and consists of a discrete series component and a nontempered component.
\\(3)
$\chi' = \tilde{\eta}'\nu'$.  The composition series of the representation $I_{\rG_u}(\chi', \eta')$ 
consists of a square-integrable Steinberg subrepresentation and a nontempered one-dimensional quotient.

Since $\mu_u$ is ramified of odd order, 
the restriction $\mu'_u|_{\mc{O}_{F'_u}^\times} = \mu_u^2|_{\mc{O}_{F'_u}^\times}$ is nontrivial with odd order.
So, the restriction of $\chi'$ to $\mc{O}_{F'_u}^\times$ is also ramified with odd order.  
In particular, $\chi'$ is nontrivial on $F'^\times_u$;
hence, $I_{\rG_u}(\chi', \eta')$ does not belong to case (1)
of the classification above. 
Since the order of $\chi'|_{\mc{O}_{F'_u}^\times}$ is odd, $\chi'$ cannot be of the form $\xi'\kappa'\nu'^{1/2}$ 
as in case (2).
Lastly, $I_{\rG_u}(\chi', \eta')$ does not belong to case (3) of the classification above, 
since the character $\tilde{\eta}'\nu'$ is unramified, but $\chi'$ is ramified.
Thus, we conclude that $\pi'_u$ is an irreducible principal series representation.

We now prove part $(ii)$ of the lemma for $0 < i \leq \m$.
Suppose the cuspidal automorphic representation $\pi'$ of $\Un(3, E'/F')(\mbb{A}_{F'})$
has a nontempered local component $\pi'_v$ for some nonarchimedean place $v \in \{w_i\}_{i = 1}^\m$. 
By the classification of the admissible representations of $\Un(3, E'_v/F'_v)$ recorded in \cite[p.\ 214, 215]{F}, 
$\pi'_v$ must be either (a) a one-dimensional representation, or (b)
the nontempered quotient $\pi'^\times$ of an induced representation belonging to case (2) of the classification above.
Suppose case (a) holds.  Since \[\Un(3, E/F)(\Af) = {\rm SU}(3, E/F)(\Af)\cdot\{\diag(x, 1, \alpha(x)^{-1}):x \in \idc{E}\},\]
the strong approximation property of $\mbb{G}_m$ (\cite[Thm.\ 1.5]{PR})
and of the simply connected group ${\rm SU}(3, E/F)$ (\cite[Thm.\ 7.12]{PR}) implies that 
$\pi'$ is one-dimensional, a contradiction. 
In case (b), by \cite[Part 2.\ III.\ Prop.\ 5.2.3]{F} the representation $\pi'$ 
must belong to the quasi-packet which is the $e'$-lift of a one-dimensional representation of $\Un(2, E'/F')(\mbb{A}_{F'})$.
This contradicts the fact that $\pi'$ has, by construction, irreducible principal series local components.
Hence, neither case (a) nor (b) holds, and we conclude that $\pi'_v$ is tempered for all $v \in \{w_i\}_{i = 1}^\m$.
By construction, the local test functions $f_{w_i}$ ($1 \leq i \leq \m$) 
are twisted pseudo-coefficients of the square-integrable representations $\pi'_i$ of our choice, 
so $\la \pi'_{w_i}, \rf_{w_i}\times\beta_{w_i}\ra \neq 0$ implies that $\pi'_{w_i} \cong \pi'_i$.
This proves part $(ii)$ of the lemma.
\end{proof}

\begin{lemma}\label{lemma:globalconstruction1}
Let $\pi_0$ be a $\beta$-invariant, cuspidal, irreducible $\Un(3, \K'/k')$-module.
Let $F$, $F'$, $E$, $E'=F'E$ be totally imaginary number fields, 
where $[E:F] = 2$ and $F'$ is a degree $n$ cyclic extension of $F$, such that
there is a place $w$ of $F$, which remains prime in $E$ and $F'$, such that $F_w = k$, $E_w = \K$, $F'_w = k'$ and $E'_w = \K'$.
Let $u$ be a place of $F$ which stays prime in $F'$ and splits in $E$.  
There exists a $\beta$-invariant, irreducible, cuspidal, automorphic representation $\pi$ of $\Un(3, E'/F')(\mbb{A}_{F'})$ such that$:$
$(i)$ $\pi_w \cong \pi_0$$;$
$(ii)$ The $\beta$-twisted character of $\pi_u$ is not identically zero on the set of 
elliptic $\beta$-regular elements in $\Un(3, E'/F')(F'_u)$$;$
$(iii)$ $\pi_v$ is an irreducible principal series representation for all places $v$ of $F'$ at which $E'/F'$ is ramified$;$
and $(iv)$ $\pi_v$ is unramified for each place $v \notin \{w, u\}$ of $F'$ at which $E'/F'$ is unramified.
\end{lemma}
\begin{proof}
The proof is very similar to that of Lemma \ref{lemma:globalconstruction}.
In this case the situation is even simpler, for we put no restriction on $\pi_u$ other than it be elliptic.
So, in the simple trace formula (see proof of Lemma \ref{lemma:globalconstruction}), it suffices to let $\rf_u$ 
be any test function whose twisted orbital integral is supported only on elliptic 
$\beta$-regular elements.
\end{proof}

\subsubsection{Local Character Relations}
If a local packet $\{\pi\}$ of $G$ lifts to a representation $\tilde{\pi}$ of $\GL(3, E)$,
we put $\{\pi(\tilde{\pi})\} := \{\pi\}$.
If $\tilde{\pi}$ is cuspidal, then $\{\pi(\tilde{\pi})\}$ is a singleton,
by Proposition \ref{prop:u3singletonlocalpacket} in the Appendix.
\begin{proposition}\label{prop:traceglobaldiscrete1}
{\rm (i)}
Let $\pi$ be a cuspidal $G$-module belonging to a local packet $\{\pi(\tilde{\pi})\}$
which lifts to a cuspidal representation $\tilde{\pi}$ of $\GL(3, \K)$.
In particular, $\tilde{\pi}$ is $\sigma$-invariant and the local packet $\{\pi(\tilde{\pi})\}$ consists of $\pi$ alone.
Assume, in the case where $n = [\K':\K] = 3$, 
that $\tilde{\pi}$ is not the monomial representation associated with any character of $\K'^\times$. 
Then, there exists a $\beta$-invariant cuspidal representation $\pi'$ of $\rG$, 
and a nonzero intertwining operator $A \in \Hom_{\rG}(\pi', \beta\pi')$, such that
the following character identity holds for matching functions$:$
\begin{equation}\label{eq:tracelocalstableprop}
\la \pi', \rf\ra_A = \la \pi, f\ra.
\end{equation}
Moreover, $\pi'$ belongs to the singleton local packet $\{\pi'(\tilde{\pi}')\}$ which lifts to
the $\beta$- and $\sigma'$-invariant cuspidal representation $\tilde{\pi}'$ of $\GL(3, \K')$, 
where $\tilde{\pi}'$ is the base-change lift of $\tilde{\pi}$ with respect to $\K'/\K$.

{\rm (ii)}
Conversely, given a $\beta$-invariant representation $\pi'$ of $\rG$ 
which lifts to a cuspidal representation $\tilde{\pi}'$ of $\GL(3, \K')$,
there exists a cuspidal $G$-module $\pi$,
and a nonzero intertwining operator $A \in \Hom_{\rG}(\pi', \beta\pi')$, 
such that equation \eqref{eq:tracelocalstableprop} holds for matching functions.
\end{proposition}
\begin{proof}[Proof of part $($i$)$ of Proposition \ref{prop:traceglobaldiscrete1}]
Let $F, F', E, E'=F'E$ be totally imaginary number fields, where $[E:F] = 2$, $F'/F$ is cyclic of degree $n$,
and there is a place $w$ of $F$ which stays prime in $E$ and $F'$, such that $F_w = k$, $E_w = \K$, $F'_w = k'$ and $E'_w = \K'$.
Let $S$ be the set of the archimedean places of $F$, the dyadic places,
and the nonarchimedean places different from $w$ where at least one of the number field extensions $E/F$, $F'/F$ is ramified.
The cardinality of $S$ is finite.
Let $u \notin S$ be a nonarchimedean place of $F$ which splits in $E$ but stays prime in $F'$.

By Lemma \ref{lemma:globalconstruction1}, there exists a cuspidal automorphic representation of $\mb{G}(\Af)$, 
which we denote again by $\pi$, such that:
\renewcommand{\labelenumi}{(\arabic{enumi})}
\begin{enumerate}
\item
$\pi_{w}$ is equivalent to the local representation of $\Un(3, \K/k)$ previously denoted by $\pi$;
\item
$\pi_u$ is an elliptic representation of $G_u = \GL(3, F_u)$;
\item
$\pi_v$ is unramified for each place $v \notin S \cup \{w, u\}$; 
and $\pi_v$ is an irreducible principal series representation for each place $v \in S$.
\end{enumerate}

By construction, the global packet $\{\pi\}$ containing $\pi$ is stable because $\{\pi_w\}$ is a singleton local packet,
so $\{\pi\}$ $b$-lifts to a \mbox{$\sigma$-invariant}, 
cuspidal, automorphic representation $\tilde{\pi}$ of $\GL(3, \mbb{A}_{E})$. 
The packet $\tilde{\pi}$ lifts via base change to a $\beta$-invariant automorphic representation $\tilde{\pi}'$ 
of $\GL(3, \mbb{A}_{E'})$, which is $\sigma'$-invariant by Lemma \ref{lemma:sigmainvariance3}.
Its local component $\tilde{\pi}'_{w}$, being the local base-change lift of the cuspidal $\GL(3, E_w)$-module $\tilde{\pi}_{w}$,
is non-cuspidal only if $n = 3$ and $\tilde{\pi}_w$ is the monomial representation associated with a character of $\K'^\times$.
Since that is not the case by hypothesis, $\tilde{\pi}'_{w}$ is cuspidal, 
which implies that $\tilde{\pi}'$ is a discrete spectrum automorphic representation.
Hence, $\tilde{\pi}'$ is the lift of a global packet $\{\pi'\} = \{\pi'(\tilde{\pi}')\}$ on $\mb{\rG}(\Af)$.
Moreover, the local packet $\{\pi'_w\} =\{\pi'(\tilde{\pi}'_{w})\}$ consists of a single cuspidal representation $\pi'_w$. 

Let $\rf$ and $f$ be matching functions on $\mb{\rG}(\Af)$, $\mb{G}(\Af)$, respectively, 
such that their local components at each place outside of $S\cup\{w, u\}$ are spherical.
By Lemma \ref{lemma:traceglobalstable}, we have the following global character identity:
\begin{equation}\label{eq:tracelocalstable}
\sum_{\tau' \in \{\pi'\}}\ep(\tau')m(\tau')\la \tau', \rf\times\beta\ra_{S\cup\{w, u\}}
= 
\la \{\pi\}, f\ra_{S\cup\{w, u\}}.
\end{equation}
Here, recall from Section \ref{sec:globaltf} that
\[
\la \tau', \rf\times\beta\ra_{S\cup\{w, u\}} := \prod_{v \in S\cup\{w, u\}}\la \tau'_v, \rf_v\times\beta_v\ra,\]
where $\la \tau'_v, \rf_v\times\beta_v\ra$ is the twisted character associated with a fixed intertwining operator
$A(\tau'_v) \in \Hom_{\rG_v}(\tau'_v, \beta_v\tau'_v)$, that is
$\la \tau'_v, \rf_v\times\beta_v\ra = \la \tau'_v, \rf_v\ra_{A(\tau'_v)}$.

For each dyadic place $v$ of $F'$, the local component $\tau'_v$ of $\tau'$ belongs to the same local packet as $\pi'_v$.
By construction, $\pi'_v$ is an irreducible principal series representation,
which implies that the multiplicity one property holds for $\tau'$ (see Theorem \ref{thm:multoneu3}).
Hence, $m(\tau') = 1$ for all $\tau' \in \{\pi'\}$.

The coefficient $\ep(\tau')$ is an $n$-th root of unity, defined as follows:
If $\tau'$ is $\beta$-invariant, then the operator induced by $T_\beta$ on $\otimes_v\tau'_v$ is nonzero, 
and $\ep(\tau')$ is the $n$-th root of unity by which the global intertwining operator $\otimes_v A(\tau'_v)$ differs from 
$T_\beta$.  If $\tau'$ is not $\beta$-invariant,
then both operators are zero, and we let $\ep(\tau') = 1$.

Since $\{\pi'\}$ is a restricted tensor product $\otimes'_v\{\pi'_v\}$ of local packets, 
and $\{\pi'_w\}$ consists of a single cuspidal representation $\pi'_{w}$,
the local component $\tau'_{w}$ of every representation $\tau' \in \{\pi'\}$ is equivalent to $\pi'_{w}$.
Moreover, by hypothesis the local packet $\{\pi_w\}$ consists of $\pi_w$ alone.
So, equation \eqref{eq:tracelocalstable} is equivalent to:
\begin{multline}\label{eq:traceglobaldiscrete11}
\sum_{\tau' \in \{\pi'\}} \ep(\tau') \lp \prod_{v \in S} 
\la \tau'_v, \rf_v\times\beta_v\ra\rp
 \la \tau'_{u}, \rf_{u}\times\beta_{u}\ra \la \pi'_{w}, \rf_{w}\times\beta_{w}\ra
\\=
\lp \prod_{v \in S} \la \{\pi_v\}, f_v\ra \rp
\la \{\pi_u\} , f_{u}\ra \la \pi_w , f_{w}\ra.
\end{multline}

Since the number fields are totally imaginary, 
the archimedean components of the automorphic representations are all irreducible principal series representations.  
Hence, by the construction of $\pi$, 
for each place $v \in S$ the local packet $\{\pi_v\}$ consists of a single irreducible principal series representation.
By the linear independence of characters, 
the right-hand side of equation \eqref{eq:traceglobaldiscrete11}
is not identically zero as a distribution on the space of test functions on $\mb{G}(\Af)$.
So, by Lemma \ref{lemma:nonvanishing} the left-hand side of the equation is nonzero.
Hence, by the linear independence of characters and Lemma \ref{lemma:tracelocalinduced}, 
the local components at $v \in S$ of the automorphic representations in equation \eqref{eq:traceglobaldiscrete11}
cancel one another.

Lastly, at the place $u$, the group $G_u$ (resp.\ $\rG_u$) is $\GL(3, F_u)$ (resp.\ $\GL(3, F'_u)$), 
and the $\GL(3, F'_u)$-module $\pi'_u$ is the local base-change lift of $\pi_u$, 
by the commutativity of the $L$-group diagram (Figure \ref{fig:lgroups}).
Since local base-change lifting for general linear groups has been established in \cite[Chap.\ 1.\ Thm.\ 6.2]{AC}, 
we may cancel the local components at $u$ of the representations in equation \eqref{eq:traceglobaldiscrete11}.
We thus obtain the identity:
\[
\xi_n \la \pi'_{w} , \rf_{w}\ra_{A(\pi'_w)} = \la \pi_{w}, f_{w}\ra
\]
for some $n$-th root of unity $\xi_n$. 
Letting $A = \xi_n A(\pi'_{w})$, part (i) of the proposition follows.
\end{proof}
\begin{proof}[Proof of part $($ii$)$ of Proposition \ref{prop:traceglobaldiscrete1}]
Let $F, F', E, E'$ again be the system of totally imaginary number fields 
introduced in the proof of part (i) of the proposition.
We let $S$ denote the same set of ``bad'' places as before.
Using Lemma \ref{lemma:globalconstruction1},
we construct a $\beta$-invariant, irreducible, cuspidal, automorphic representation of $\mb{\rG}(\Af)$, which we denote again by $\pi'$,
such that:
\renewcommand{\labelenumi}{(\arabic{enumi})}
\begin{enumerate}
\item
$\pi'_{w}$ is equivalent to the local representation of $\Un(3, \K'/k')$ previously denoted by $\pi'$;
\item
$\pi'_u$ is an elliptic representation of $\rG_u$.
\item
$\pi'_v$ is unramified for each place $v \notin S \cup \{w, u\}$, 
and $\pi'_v$ is an irreducible principal series representation for each place $v \in S$.
\end{enumerate}

By construction, the packet $\{\pi'\}$ containing $\pi'$ is stable because $\{\pi'_w\}$ is a singleton local packet,
which implies that $\{\pi'\}$ $b'$-lifts to a \mbox{$\sigma'$-invariant}, 
cuspidal, automorphic representation $\tilde{\pi}'$ of $\GL(3, \mbb{A}_{E'})$. 
By Proposition \ref{prop:classification1}, $\{\pi'\}$ is the weak lift of packet(s) from $\mb{G}$
and/or $\mb{H}$, via the $L$-homomorphisms $b_G$ and/or $e_H$.
Suppose it is the $e_H$-weak-lift of a packet of $\mb{H}(\Af)$.  
By the commutativity of the $L$-group diagram in Figure \ref{fig:lgroups}, 
the packet $\tilde{\pi}'$ must be the $b_{3}$-lift of a parabolically induced representation of $\GL(3, \Ae)$, 
which contradicts the cuspidality of $\tilde{\pi}'$.
Hence, the packet $\{\pi'\}$ must be the $b_G$-weak-lift of a stable packet $\{\pi\} = \{\pi(\tilde{\pi})\}$ of $\mb{G}(\Af)$, 
which $b$-lifts to a $\sigma$-invariant, cuspidal, automorphic representation $\tilde{\pi}$ of $\GL(3, \Ae)$.
By the commutativity of the $L$-group diagram, the representation $\tilde{\pi}$ base-change lifts to $\tilde{\pi}'$ via $b_{3}$.

By the theory of base change for general linear groups,
the automorphic representation $\tilde{\pi}$ is cuspidal at $w$, elliptic at $u$, 
and irreducible principal series everywhere else.
Hence, by Proposition \ref{prop:u3gl3localpacket}, the representation $\pi$
satisfies the conditions (1), (2) and (3) in the proof of part (i) of the proposition.  
In particular, the cuspidal representation $\pi_w$ belongs to the singleton local packet $\{\pi(\tilde{\pi}_w)\}$.

The twisted local character identity now follows via the same argument as in the proof of part (i) of the proposition.
\end{proof}

Let $\rho$ be a cuspidal representation of $H = \Un(2, \K/k)$ which does not belong to a local packet of the form
$\rho(\theta_1, \theta_2)$ for any pair of characters $(\theta_1, \theta_2)$ of $\K^k$.
Let $\pi(\rho)$ denote the local packet of $G$ which is the lift of $\rho$.
By Lemma \ref{lemma:sigmainvariance2} and the same reasoning used in the proof of Proposition \ref{prop:traceglobaldiscrete1},
there is a base-change lifting of $\rho$ 
to a cuspidal representation $\rho'$ of $\rH = \Un(2, \K'/k')$, 
corresponding to the $L$-homomorphism $b_H$ (see Figure \ref{fig:lgroups}).
Since $\rho$ is not a lift from $\Un(1, \K/k)\times\Un(1, \K/k)$, it lifts to a cuspidal representation $\tilde{\rho}$ of $\GL(2, \K)$.
The representation $\rho'$ lifts to the $\GL(2, \K')$-module $\tilde{\rho}'$ which is the base-change lift of $\tilde{\rho}$.
Since $\K'/\K$ is an odd degree extension, $\tilde{\rho}'$ is cuspidal, which implies that $\rho'$ is not a lift from
$\Un(1, \K'/k')\times\Un(1, \K'/k')$.
Let $\pi'(\rho')$ denote the local packet of $\rG$ which is the lift of $\rho'$ via the $L$-homomorphism $e'$.
It consists of two cuspidal representations $\pi'^+$ and $\pi'^-$ (see Proposition \ref{prop:u2u3localpacket}).
\begin{proposition}\label{prop:unstablebetainvariance}
The cuspidal representations $\pi'^+$ and $\pi'^-$ of $\rG$ are $\beta$-invariant.
\end{proposition}
\begin{proof}
Construct number fields $F, F', E, E'$
such that there are $2$ nonarchimedean places $w_1, w_2$ of $F$ which stay prime in $E$ and $F'$, 
and $F_w =k$, $F'_w = k'$, $E_w = \K$, $E'_w = \K'$.
Using the same technique as in the proof of Lemma \ref{lemma:globalconstruction},
we construct an automorphic representation of $\mb{H}(\Af) = \Un(2, E/F)(\Af)$ such that
its local components at $w_1$, $w_2$ are equivalent to $\rho$.

From henceforth, by abuse of notation we denote this automorphic representation of $\mb{H}(\Af)$ by $\rho$. 
Hence, $\rho_{w_i}$ ($i = 1, 2$) is equivalent to the local representation denoted by $\rho$ in the preamble to the lemma.
Let $\rho'$ be the weak base-change lift of
$\rho$ to $\mb{\rH}(\Af) = \Un(2, E'/F')(\mbb{A}_{F'})$ corresponding to the $L$-homomorphism $b_H$
(see Lemma \ref{lemma:sigmainvariance2}).
Put $\{\pi'\} := \pi'(\rho')$, the unstable packet of $\mb{\rG}(\Af)$ which is the $e'$-lift of
$\rho'$. 

It follows from the definitions of the $L$-homomorphisms $b_H$ and $e'$ (see \cite[p.\ 207]{F})
that almost all unramified local components of each member of $\{\pi'\}$ are $\beta$-invariant. 
Hence, by the rigidity theorem for the packets of $\Un(3)$ (Theorem \ref{thm:rigidity}), the global 
intertwining operator $T_\beta$ associated with $\beta$ must map each discrete spectrum automorphic representation $\pi'$ in 
$\{\pi'\}$ to an automorphic representation in the same global packet which is equivalent to $\beta\pi'$.
This implies that the local packet $\{\pi'^+, \pi'^-\}$ is equivalent to $\{\beta\pi'^+, \beta\pi'^-\}$.
On the other hand, the action of $\beta$ is of odd order, while $\{\pi'^+, \pi'^-\}$ consists of
two inequivalent irreducible representations.  Hence, $\pi'^+ \cong \beta\pi'^+$, and $\pi'^- \cong \beta\pi'^-$.
\end{proof}

\begin{proposition}\label{prop:tracelocalunstable}
There exist nonzero operators 
$A^\dagger \in \Hom_{\rG}(\pi'^\dagger, \beta\pi'^\dagger)$ $(\dagger$ is $+$ or $-)$ such that
the following system of character identities holds for matching functions$:$
\begin{equation}
\label{eq:tracelocalunstableprop1}
\begin{split}
2 \la \pi'^+, \rf\ra_{A^+} &= \la \pi(\rho), f\ra + \la \rho, f_H\ra,\\
2 \la \pi'^-, \rf\ra_{A^-} &= \la \pi(\rho), f\ra - \la \rho, f_H\ra.
\end{split}
\end{equation}
\end{proposition}

\begin{proof}
Construct totally imaginary number fields $F, F', E, E' = F'E$, such that there are $m \geq 2$ places $w_1, w_2,\ldots, w_\m$ of $F$
which stay prime in $E$ and $F'$, and $F_{w_i} = k$, $E_{w_i} = \K$, $F'_{w_i} = k'$, $E'_{w_i} = \K'$, $1 \leq i \leq \m$. 
Using Lemma \ref{lemma:globalconstruction}, we construct $\rho$, $\rho'$, $\{\pi\}$, $\{\pi'\}$, where
$\rho$ is an automorphic representation of $\mb{H}(\Af)$ whose local component at each place in $\{w_i\}_{i = 1}^\n$
is equivalent to the local representation of $H$ previously denoted by $\rho$; $\rho'$ is the automorphic representation of 
$\mb{\rH}(\Af)$ which is the $b_H$-lift of $\rho$; $\{\pi\} := \pi(\rho)$ is the $e$-lift of $\rho$ to $\mb{G}(\Af)$;
and $\{\pi'\} := \pi'(\rho')$ is the $e'$-lift of $\rho'$ to $\mb{\rG}(\Af)$.
Let $\{\rho\}$ be the global packet of $\mb{H}(\Af)$ to which $\rho$ belongs.
By assumption, $\rho$ does not belong to a packet of the form $\rho(\theta_1, \theta_2)$ 
for any pair of characters $\theta_1, \theta_2$ of $C_E^F$.  Hence, $\{\rho\}$ is stable.

As in the case of $\Un(3)$,
we may (and do) construct $\rho$ such that,
at a place $v \notin \{w_i\}_{i = 1}^\n$, 
the local packet $\{\rho_v\}$ contains an unramified representation 
or consists of a single irreducible principal series representation.
Consequently, 
each member of $\pi'(\rho')$ occurs in the discrete spectrum of $\mb{\rG}(\Af)$ with multiplicity at most one
(Theorem \ref{thm:multoneu3}), and it is $\beta$-invariant by Proposition \ref{prop:unstablebetainvariance}.
By Lemma \ref{lemma:traceglobalunstable1}, we have:
\[
2 \sum_{\pi' \in \pi'(\rho')}\!\!\!
m(\pi') \ep(\pi') \la \pi', \rf\times \beta\ra_S 
= \la \pi(\rho), f\ra_S +
\la \rho, f_H\ra_S.
\]
Here, $m(\pi')$ is the multiplicity with which $\pi'$ occurs in the discrete spectrum of $\mb{\rG}(\Af)$.
It is one if an even number (including zero) of local components of $\pi'$ are equivalent to $\pi'^-$, and zero otherwise.
The coefficient $\ep(\pi')$ is an $n$-th root of unity by which the global intertwining operator $T_\beta$,
restricted to $\pi'$, differs from the tensor product of the local intertwining operators $A(\pi'_v)$, 
see the proof of part (i) of Proposition \ref{prop:traceglobaldiscrete1}.

By Lemma \ref{lemma:tracelocalinduced} and the linear independence of characters, we may
cancel the unramified and parabolically induced local components of the automorphic representations, 
and obtain the following equation for matching functions:
\begin{multline}\label{eq:tracelocalunstable1}
2 \sum_{\pi' \in \{\pi'\}} m(\pi') \ep(\pi')\prod_{i = 1}^\n \la \pi'_{w_i}, \rf_{w_i}\times\beta_{w_i}\ra
\\= 
\prod_{i = 1}^\n \la \{\pi_{w_i}\}, f_{w_i}\ra + \prod_{i = 1}^\n \la \{\rho_{w_i}\}, f_{H, w_i}\ra.
\end{multline}

For $* \in \{+, -\}$,
put $A^* := A(\pi'^*)$, a fixed nonzero intertwining operator in $\Hom_{\rG}(\pi'^*, \beta\pi'^*)$
(with $\la \pi'^*, \rf_{w_i}\times\beta_{w_i}\ra = \la \pi'^*, \rf_{w_i}\ra_{A(\pi'^*)}$).
Let $\rf^*$ be a matrix coefficient of the cuspidal representation $\pi'^*$.
Hence, for a $\beta$-invariant, irreducible, admissible representation $\pi''$ of $\rG$ with nonzero intertwining operator 
$A'' \in \Hom_{\rG}(\pi'', \beta\pi'')$, 
the twisted trace $\la \pi'', \rf^*\ra_{A''}$ is nonzero if and only if $\pi'' \cong \pi'^*$.  
For simplicity, we normalize $\rf^*$ with respect to $A^*$, so that $\la \pi'^*, \rf^*\ra_{A^*} = 1$.
Let $f^*$ (resp.\ $f_H^*$)
be a fixed local test function on $G$ (resp.\ $H$) which matches $\rf^*$.
Note that, for $1 \leq i \leq \n$, the local packets
$\{\pi_{w_i}\} = \pi(\rho_{w_i})$ (resp.\ representations $\rho_{w_i}$) are equivalent to one another.

Put
\[
D_\pi^* := \la \pi(\rho_{w_i}), f^* \ra, \quad D_\rho^* := \la \rho_{w_i}, f_H^*\ra.
\]
For $1 \leq i \leq \n$, let $\rf_{w_i} = \rf^+$,
and $f_{w_i} = f^+$, $f_{H, w_i} = f_H^+$.
For such test functions, 
it follows from equation \eqref{eq:tracelocalunstable1} and the multiplicity formula for $\{\pi'\}$ that, 
for each $\n \geq 2$, there exists an $n$-th root of unity $\xi_n^+(\n)$ such that:
\begin{equation}\label{eq:trick+}
2 \cdot \xi_n^+(\n) =  (D_\pi^+)^\n + (D_\rho^+)^\n.
\end{equation}
By a similar argument, 
we also have, for each {\it even} integer $\n \geq 2$,
an $n$-th root of unity $\xi_n^-(\n)$ such that:
\begin{equation}\label{eq:trick-}
2 \cdot \xi_n^-(\n) = \lp D_\pi^- \rp^{\n} + \lp D_\rho^- \rp^{\n}.
\end{equation}

Now let $\rf_{w_i} = \rf^+$,  $f_{w_i} = f^+$ and $f_{H, w_i} = f_H^+$ for $1 \leq i \leq \n - 1$.
Let $\rf_{w_\n} = \rf^-$, $f_{w_\n} = f^-$ and $f_{H, w_\n} = f_H^-$.
For such test functions, it follows from equation \eqref{eq:tracelocalunstable1} and the multiplicity formula
for $\{\pi'\}$ that
\begin{equation}\label{eq:trick}
0 = \lp D_\pi^+ \rp^{\n-1} D_\pi^- + \lp D_\rho^+ \rp^{\n-1} D_\rho^-,\quad\forall\,\n\geq 2.
\end{equation}

Let $z$ be a complex variable.
Multiply both sides of the above equation by $z^{m - 1}$ and take the sum of each side over 
$2 \leq \n < \infty$. For $z$ sufficiently close to $0$, we get:
\begin{equation}\label{eq:meromcontunstable}
0 = \frac{D_\pi^+ D_\pi^-}{1 - D_\pi^+ z} + \frac{D_\rho^+ D_\rho^-}{1 - D_\rho^+ z}\;.
\end{equation}
Each side of the above equation meromorphically continues to the whole complex plane. 
Since the left-hand sides of equations \eqref{eq:trick+} and \eqref{eq:trick-} have absolute values $2$ 
for infinitely many integers $\n$, 
none of $D_\pi^+$, $D_\pi^-$, $D_\rho^+$ and $D_\rho^-$ can be zero.
By the absence of poles on the left-hand side (i.e.\ the zero function) 
of equation \eqref{eq:meromcontunstable}, the two nonzero terms on the right-hand side must cancel each other, which implies
that $D_\pi^+ = D_\rho^+$ and $D_\pi^- = - D_\rho^-$.
By equation \eqref{eq:trick+}, we have
\[
\xi_n^+(2) = (D_\pi^+)^2, \quad \xi_n^+(3) = (D_\pi^+)^3.
\]
Dividing one equation by the other, we have $D_\pi^+ = D_\rho^+ = \xi_n := \xi_n^+(3)/\xi_n^+(2)$, 
which is an $n$-th root of unity.
Similarly, equation \eqref{eq:trick-} implies that $(D_\pi^-)^2 = \xi_n^-(4)/\xi_n^-(2)$,
so $D_\pi^- =  -D_\rho^- = \xi_{2n}$ for some $2n$-th root of unity.

We are now at the final stage in the proof of 
the proposition.  Let $\n = 2$.  
Let $\rf_{w_2} = \rf^+$, $f_{w_2} = f^+$, and $f_{H, w_2} = f_H^+$.  
Let $\rf_{w_1}, f_{w_1}, f_{H, w_1}$ be arbitrary matching local test functions.
By equation \eqref{eq:tracelocalunstable1} and the multiplicity formula for $\{\pi'\}$, 
we have
\[
2\, \xi_n' \la \pi'^+, \rf_{w_1}\ra_{A^+} = D_\pi^+ \la \pi(\rho_{w_1}), f_{w_1}\ra + D_\rho^+ \la \rho_{w_1}, f_{H, w_1}\ra
\]
for some $n$-th root of unity $\xi_n'$.
Hence, by the values of $D_\pi^+$ and $D_\rho^+$ just computed,
we have
\[
2\, \xi_n'' \la \pi'^+, \rf_{w_1}\ra_{A^+} = 
\la \pi(\rho_{w_1}), f_{w_1}\ra + \la \rho_{w_1}, f_{H, w_1}\ra
\]
for some $n$-th root of unity $\xi_n''$.
Now let $\rf_{w_2} = \rf^-$, $f_{w_2} = f^-$, $f_{H, w_2} = f_H^-$, 
and let $\rf_{w_1}, f_{w_1}, f_{H, w_1}$ be arbitrary matching local test functions.
By the same argument as above, we have
\[
2\, \xi_{2n}'' \la \pi'^-, \rf_{w_1}\ra_{A^-} = 
\la \pi(\rho_{w_1}), f_{w_1}\ra - \la \rho_{w_1}, f_{H, w_1}\ra
\]
for some $2n$-th root of unity $\xi_{2n}''$.
Multiplying the intertwining operators $A^+$, $A^-$ by $\xi_n''$, $\xi_{2n}''$, respectively,
the system of equations \eqref{eq:tracelocalunstableprop1} follows.  
\end{proof}

We now consider the local packets of $H$ which are lifts from $\K^k \times \K^k$. 
Let $\theta_1, \theta_2$ be characters of $\K^k$ such that $\theta_1, \theta_2$ and $\theta_3 := \omega/\theta_1\theta_2$
are distinct (recall that $\omega$ is our fixed central character of $G$).
For $i = 1, 2, 3$, 
let $\theta_i'$ be the character $\theta_i\circ\N_{\K'/\K}$ of $\K'^{k'}$.
In particular, $\theta_1', \theta_2'$ and $\theta'_3$ are distinct by Lemma \ref{lemma:normindex} in the Appendix.
Let $\rho' = \rho'(\theta_1', \theta_2')$ denote the packet of $\rH = \Un(2, \K'/k')$ 
associated with the pair $(\theta_1', \theta_2')$.
Let $\{\pi'(\theta_1', \theta_2')\} = \{\pi'_a, \pi'_b, \pi'_c, \pi'_d\}$ 
denote the packet, consisting of four inequivalent cuspidal representations of $\rG$, which is the lift of $\rho'$.
We let $\pi'_a$ be the unique generic representation in $\{\pi'(\theta_1', \theta_2')\}$ (see \cite[p.\ 214]{F}).

\begin{proposition}\label{prop:uustablebetainvariance}
All four members of $\{\pi'(\theta_1', \theta_2')\}$ are $\beta$-invariant.
\end{proposition}
\begin{proof}
The equivalence classes of the inner forms of $\rH$ are indexed by the two elements of $k'^\times/\N_{\K'/k'}\K'^\times$.
We write $\rH_1 := \rH$, and let $\rH_d$ denote a (unique up to equivalence) non-quasisplit inner form of $\rH$ associated with
the nontrivial element $d \in k'^\times/\N_{\K'/k'}\K'^\times$.

Fix a $\beta$-invariant additive character $\psi'$ of $k'$.  For $* = 1, d$, 
the pair of characters $(\psi', \kappa'^{-1})$ determines a Howe lifting 
$\rho' \mapsto \mc{H}^*_{\psi'}(\rho', \kappa'^{-1})$,
of irreducible admissible representations $\rho'$ of $\rH_*$, to $\rG$ (see \cite{WHowe}).
Let $\rho_d' := \rho_d'(\theta_1', \theta_2')$ denote the packet on $\rH_d$ which consists of the transfer to inner form of the members
of $\rho' := \rho'(\theta_1', \theta_2')$.
The packets $\rho'$ and $\rho_d'$ each consists of two representations. 
By Theorem 4.3 in \cite{GRS}, 
each of these four representations lifts via Howe correspondence to one of the members of
$\{\pi'(\theta_1', \theta_2')\}$.  
This lifting is one-to-one, hence its image exhausts all four members of the packet.

Let $V$ (resp.\ $W$) be the Hermitian space (resp.\ skew-Hermitian space) over $\K'$ associated with $\rG$ (resp.\ $\rH$).
So, as $\K'$-vector spaces, $V = \K'^3$ and $W = \K'^2$. 
The Hermitian structures of these spaces endow the tensor product $V\otimes W$ with a symplectic structure
(see, for example, \cite[p.\ 454]{GR}).
Let $\Sp = \Sp(V\otimes W)$ be the associated symplectic group.
Let $\mf{H}$ be the Heisenberg group associated with $V\otimes W$ (see \cite[Sec.\ 1]{Howe}).
Let $\mc{S} = \mc{S}(\mf{H})_{\psi'}$ be the space of complex valued, smooth, compactly supported modulo center functions on $\mf{H}$
which transform under the center (identified with $\K'$) of $\mf{H}$ via $\psi'$ (see \cite[p.\ 28]{MWV}).
Let $(\rho_{\psi'}, \mc{S})$ be the Schr\"odinger representation of $\mf{H}$, 
where $\mf{H}$ acts on the functions in $\mc{S}$
via right translation.  
Let $\Mp_1$ denote the metaplectic group ${\rm Mp}(V\otimes W)$.  
It is the group of elements $(g, M)$ in $\Sp \times \GL(\mc{S})$ which satisfy:
\[
M\rho_{\psi'}(h)M^{-1} = \rho_{\psi'}(gh),\quad \forall\, h \in \Sp.
\]
The oscillator representation $\Omega = (\Omega_{\psi'}, \mathcal{S})$ is the representation of $\Mp_1$ given by
$\Omega(g, M) = M$ (see \cite[Sec.\ 2]{Howe}).

A section $g \mapsto (g, M[g])$ from $\Sp$ to $\Mp_1$ is defined in \cite[Chap.\ 2, Sec.\ II.2]{MWV}.
It is associated with a two-cocycle $c(g, g')$ on $\Sp$ (see \cite[Chap.\ 3, Sec.\ I]{MWV}) such that:
\[M[g]\cdot M[g'] = c(g, g')M[g g'],\quad \forall\; g, g' \in \Sp.\]
Let $\Mp_2$ be the group $\Sp \times \mbb{C}^\times$ with multiplication law:
\[(g_1, b_1)(g_2, b_2) = (g_1 g_2, c(g_1, g_2)b_1 b_2).\] 
We have an isomorphism $\varphi : \Mp_2 \rightarrow \Mp_1$ defined by:
\[
\varphi(g, b) = (g,\, b\cdot M[g]), \quad g \in \Sp,\; b \in \mbb{C}^\times.
\]

By Proposition 3.1.1 in \cite{GR}, there is a splitting:
\[
s = s_{\psi', \kappa'^{-1}}:(g, h) \mapsto (m(g, h), b(g, h))
\]
from $\rG \times \rH$ to $\Mp_2$, associated with the pair of characters $(\psi', \kappa'^{-1})$.
Note that for simplicity  we have suppressed from the notation the dependence of $b$ on $(\psi', \kappa'^{-1})$.
Composing $\varphi$ with $s$, we have a splitting $\varphi\circ s : \rG\times \rH \rightarrow \Mp_1$.
Let $\Omega'$ denote the representation $\Omega\circ\varphi\circ s$ of $\rG\times\rH$.
For each irreducible admissible representation $\rho'$ of $\rH$, the representation
$\pi' = \mc{H}^1_{\psi'}(\rho', \kappa'^{-1})$ is by definition the unique irreducible $\rG$-module
such that $\pi'\otimes\rho'$ occurs as a direct summand of $\Omega'$ (see \cite{WHowe}).
We claim that $\Omega'$ is $\beta$-invariant.

Let $\omega(\beta)$ be the vector space endomorphism of $\mc{S}$
defined by:
\[
(\omega(\beta)\phi)(x) = \phi(\beta^{-1}(x)), \quad \phi \in \mc{S}, x \in \mf{H}.
\]
We define an isomorphism $\tilde{\beta}_1 : {\rm Mp}_1 \rightarrow \Mp_1$ 
by: \[\tilde{\beta}_1(g, M) = (\beta(g),\, \omega(\beta)\,M\,\omega(\beta)^{-1}),\] 
and a map $\tilde{\beta}_2 : \Mp_2 \rightarrow \Mp_2$ by: $\tilde{\beta}_2(g, b) = (\beta(g), b).$
We claim that $\tilde{\beta}_2$ is a group isomorphism.
The two-cocycle $c(g, g')$ is defined in terms of the Leray invariant $q(g, g')$,
which is uniquely determined by the triplet $(X, g^{-1}X, g'X)$,
where $X$ is a fixed Lagrangian subspace of the symplectic space $V\otimes W$ 
(see \cite[p. 55]{MWV}, \cite{Kudla}, \cite{Rao}).
Since the $\K'$-vector spaces $\beta(g)^{-1}X$ and $g^{-1}X$ are equal to each other, and likewise $\beta(g')X = g'X$,
the Leray invariant $q(g, g')$ is $\beta$-invariant, i.e.\ $q(\beta(g), \beta(g')) = q(g, g')$.
Consequently, the two-cocycle $c(g, g')$ is $\beta$-invariant, and $\tilde{\beta}_2$ is a group isomorphism.
 
Observe that the isomorphisms $\tilde{\beta}_2$ and $\varphi^{-1}\circ\tilde{\beta}_1\circ\varphi$ of $\Mp_2$ 
are equal to each other.
Moreover, the map $b : \rG\times\rH \rightarrow \mbb{C}^\times$, defined in terms of the Hilbert symbol, is $\beta$-invariant.
It now follows from the definitions of $\tilde{\beta}_1$ and $\tilde{\beta}_2$ that, for $(g, h) \in \rG\times\rH$,
\[\begin{split}
\Omega'(\beta(g), \beta(h)) &=  b(\beta(g), \beta(h)) M[m(\beta(g), \beta(h))] \\
&= b(g, h) \omega(\beta)M[m(g, h)]\omega(\beta)^{-1} = \omega(\beta)\Omega'(g, h)\omega(\beta)^{-1}.
\end{split}\]
This shows that $\Omega'$ is a $\beta$-invariant representation of $\rG\times\rH$.

Consequently, if a representation $\pi'$ of $\rG$ is the Howe lift of a representation $\rho'$ of $\rH$, 
then $\beta\pi'$ is the Howe lift of $\beta\rho'$. 

Consider the packet $\rho'(\theta_1', \theta_2') = \{\rho'^+, \rho'^-\}$ on $\rH$.
Since the order of the Galois group element $\beta$ has odd order $n$, 
it follows from the same argument used in the proof of Proposition \ref{prop:unstablebetainvariance} that
$\rho'^+$ and $\rho'^-$ are $\beta$-invariant.  By Theorem 4.3 in \cite{GRS}, 
two of the representations in the packet $\{\pi'(\theta_1', \theta_2')\}$ are the Howe lifts
$\pi'_1 := \mc{H}^1_{\psi'}(\rho'^+, \kappa'^{-1})\otimes\chi$ and $\pi'_2 := \mc{H}^1_{\psi'}(\rho'^-, \kappa'^{-1})\otimes\chi$,
where $\chi =\kappa'|_{\K'^{k'}}\cdot\theta_3'$.
For $* = +, -$ and $i = 1, 2$, 
both $\rho'^*\otimes \pi'_i$ and $\beta\rho'^*\otimes\beta\pi'_i$ occur in $\Omega'$, and $\rho'^*$ is $\beta$-invariant.
It follows from the uniqueness of Howe correspondence that $\pi'_i$ is $\beta$-invariant.

Since the $\Gal(\K'/\K)$-orbit of each of the remaining two members of $\{\pi'(\theta_1', \theta_2')\}$ 
must have cardinality dividing the odd number $n$, we conclude that these representations must also be $\beta$-invariant.
\end{proof}
\noindent {\sc Remark 1.}
We can give an alternative, more explicit proof of Proposition \ref{prop:uustablebetainvariance}.
Indeed, one can realize $\Omega'$ in the Schr\"odinger model $\mc{S}(V)$, which is the space
of complex valued, smooth compactly supported functions on $V$ (see \cite[p. 432]{GRS}).
To prove the $\beta$-invariance of  $(\Omega', \mc{S}(V))$, 
one can resort to an explicit description of the action of $\rG\times\rH$ on $(\Omega', \mc{S}(V))$:
For each function $\phi \in \mc{S}(V)$,
\begin{align*}\label{eq:oscillatorconditions}
& \lp\Omega'(g, 1)\phi\rp(x) &=&\; \kappa'^{-1}(\det g)\phi(x g),& g \in \rG,\\
& \lp\Omega'(1, \diag(c, \alpha(c)^{-1}))\phi\rp(x) &=&\; \kappa'^{-1}(c^3)\abs{c}^{3/2}\phi(cx),& c \in \K'^\times,\\
& \lp\Omega'\lp 1, \lp\lsm 1 & t\\0&1\rsm\rp\rp\phi\rp(x) &=&\; \psi'(t\la x, x\ra)\phi(x),& t \in \K',\\
& \lp\Omega'(1, \antidiag(1, -1))\phi\rp(x) &=&\; \gamma\hat{\phi}(x). &
\end{align*}
Here, $\hat{\phi}$ is the Fourier transform of $\phi$ associated with $\psi'$, and $\gamma$ 
is the Weil factor associated with the function $x \mapsto \psi'\lp\la x, x\ra\rp$ on $V$,
where $\la \cdot, \cdot\ra$ is the symplectic form on $V$.

Let $\omega(\beta)$ be the vector space endomorphism of $\mc{S}(V)$ 
defined by $(\omega(\beta)\phi)(x) = \phi(\beta^{-1}(x))$.
It follows from the $\beta$-invariance of the characters $\psi'$ and $\kappa'^{-1}$, 
and the equations defining $\Omega'$, that
$\omega(\beta)$ is an intertwining operator in $\Hom_{\rG\times\rH}(\Omega', \beta\Omega')$.
This completes the alternative proof.

\quad\\
\noindent{\sc Remark 2.}
Note that the assertion in \cite{GRS} that the multiplicity one theorem for $\Un(3)$ has been proved is incorrect
(see \cite[p.\ 394-395]{F} for explanation).
However, in the proof of Proposition \ref{prop:uustablebetainvariance},
we do not rely on any statement which uses multiplicity one for $\Un(3)$.
\quad\\

Let $\theta_i, \theta'_i$, $1 \leq i \leq 3$, be as in the paragraph preceding Proposition \ref{prop:uustablebetainvariance}.
For $i \neq j \in \{1, 2, 3\}$, let $\rho_{i, j}$ denote the local packet $\rho(\theta_i, \theta_j)$ of $H$ 
which is the lift of the pair of characters $(\theta_i, \theta_j)$.
Let $\{\pi\} := \{\pi(\theta_1, \theta_2)\}$ denote the local packet of $G$ which is the lift of $\rho(\theta_1, \theta_2)$.
Let $\{\pi'\} := \{\pi'(\theta_1', \theta_2')\}$ 
be the packet of $\rG$ defined similarly, where $\theta_i' := \theta_i\circ\N_{\K'/\K}$ ($i = 1, 2$).

\begin{proposition}\label{prop:uustablecaseii}
There exist nonzero intertwining operators 
$A_{*} \in \Hom_{\rG}(\pi'_*, \beta\pi'_*)$ $(* = a, b, c, d)$, 
and a way to index the set $\{\rho_{i, j} : {i \neq j \in \{1, 2, 3\}}\}$ as $\{\rho_1, \rho_2, \rho_3\}$,
such that the following system of local character identities holds for matching functions$:$
\vspace{-.02in}
\begin{equation}
\begin{split}
\label{eq:tracelocaluunstable00}
4 \la \pi'_a, \rf\ra_{A_a} &= \la \{\pi\}, f\ra
+ \la \rho_1, f_H\ra + \la \rho_2, f_H\ra + \la \rho_3, f_H\ra,\\
4 \la \pi'_b, \rf\ra_{A_b} &= \la \{\pi\}, f \ra
- \la \rho_1, f_H\ra - \la \rho_2, f_H\ra + \la \rho_3, f_H\ra,\\
4 \la \pi'_c, \rf\ra_{A_c} &= \la \{\pi\}, f\ra
- \la \rho_1, f_H\ra + \la \rho_2, f_H\ra  - \la \rho_3, f_H\ra,\\
4 \la \pi'_d, \rf\ra_{A_d} &= \la \{\pi\}, f\ra
+ \la \rho_1, f_H\ra - \la \rho_2, f_H\ra - \la \rho_3, f_H\ra.
\end{split}
\end{equation}
\end{proposition}

\begin{proof}
Let $F, F', E, E' = F'E$ be totally imaginary number fields, with $\m \geq 2$ places $w_1 =\ldots = w_\m$ of $F$,
which stay prime in $E$ and $F'$, such that $F_{w_i} = k$, $E_{w_i} = \K$, $F'_{w_i} = k'$ and $E'_{w_i} = \K'$, $1 \leq i \leq \m$.
Let $S$ be the set of the archimedean places of $F$ 
and the nonarchimedean places where at least one of the number field extensions $E/F$, $F'/F$ is ramified.

Construct characters of $C_E^F$, which we denote again by $\theta_i$ ($i = 1, 2, 3$), such that:
\begin{itemize}
\item[(i)]
For $1 \leq s \leq \n$, the local component $\theta_{i, w_s}$
is equal to the character of $\K^k$ previously denoted by $\theta_i$;
\item[(ii)]
At each place $v \notin S\cup\{w_1,\ldots, w_\n\}$, the character $\theta_{i, v}$ is unramified.
\item[(iii)]
For each $v \in S$ which does not split in $E$, 
the characters $\theta_{1, v}, \theta_{2, v}$ and $\theta_{3, v}$ are equal to one another.
\end{itemize}
This construction is possible due to the Poisson summation formula for $\Un(1)$ and
the fact that $\Un(1, E_v/F_v)$ is compact for the places $v$ referred to in (i), (iii).

For $i \neq j \in \{1, 2, 3\}$, we now let $\rho_{i, j}$ denote the {\it global} unstable packet
$\rho(\theta_i, \theta_j)$ of $\mb{H}(\Af)$ associated with the pair of global characters $(\theta_i, \theta_j)$.
By the construction of $\theta_i$ ($i = 1, 2, 3$), for each place $v \in S$ which does not split in $E$,
the local component $\rho_{i, j, v}$ of $\rho_{i,j}$ 
consists of the irreducible constituent(s) of a parabolically induced representation (see \cite[p.\ 699]{FU2}).
If $v$ splits in $E$, then $\rH_v$ is $\GL(2, F_v)$, and $\rho_{i, j, v}$ is an irreducible principal series representation.
Lastly, for $v \notin S\cup\{w_1, \ldots, w_\n\}$, 
the representation $\rho_{i, j, v}$ is unramified, hence parabolically induced.

For $i = 1, 2$, let $\theta_i'$ denote the character $\theta_i\circ\N_{E'/E}$ of $C_{E'}^{F'}$.
Thus, $\theta'_{i, w_s}$ ($s = 1, 2, \ldots, \n$) is equal to the local character previously denoted by $\theta'_i$.
Let $\{\pi\}$ be the unstable packet of $\mb{G}(\Af)$ which is the lift of $\rho_{i, j}$ 
corresponding to the $L$-homomorphism $e$ (see Figure \ref{fig:lgroups}).
Let $\rf$, $f$ and $f_H$ be matching test functions on $\mb{\rG}(\Af)$, $\mb{G}(\Af)$ and $\mb{H}(\Af)$, respectively.
Applying Lemmas \ref{lemma:traceglobaluunstable}, \ref{lemma:tracelocalinduced}
and the linear independence of characters
to cancel the contribution from the parabolically induced local components of the automorphic representations,
we obtain the following character identity:
\begin{multline}\label{eq:tracelocaluunstablen}
4 \!\!\! \sum_{\pi' \in \{\pi'(\theta_1', \theta_2')\}}
m(\pi') \ep(\pi')\prod_{s = 1}^\n\la \pi'_{w_s}, \rf_{w_s} \times \beta_{w_s}\ra\\
= \prod_{s = 1}^\n \la \{\pi_{w_s}\}, f_{w_s}\ra
+ \sum_{i \neq j \in \{1, 2,3\}} \prod_{s = 1}^\n \la \rho_{i, j, w_s}, f_{H, w_s}\ra,
\end{multline}
where $m(\pi')$ is the multiplicity with which $\pi'$ occurs in the discrete spectrum of $\mb{\rG}(\Af)$,
and $\ep(\pi')$ is the $n$-th root of unity by which the global intertwining operator $T_\beta|_{\pi'}$ differs from 
the tensor product $\otimes_v A(\pi'_v)$ of fixed local intertwining operators.

By construction, the dyadic local components of $\{\pi'\}$ contain constituents of parabolically induced representations,
so the following formula for $m(\pi')$, derived in \cite[Part 2.\ Sec.\ III.5.1]{F}, holds:
For each $s = 1, \ldots, \n$, the representation $\pi'_{w_s}$ belongs
to the local packet $\{\pi_a', \pi_b', \pi_c', \pi_d'\}$ of four cuspidal representations of $\rG_{w_s} = \mb{\rG}(k)$, 
where the representations are indexed such that $\pi_a'$ is the unique generic representation in the packet.
There exist distinct maps 
\[
\ve_{i} : \{\pi_a', \pi_b', \pi_c', \pi_d'\} \rightarrow  \{\pm 1\}, \quad i \in \{1, 2, 3\},
\]
such that:
\begin{itemize}
\item
$\ve_i(\pi_a') = 1$ for all $i$;
\item
For each $* = b, c, d$, there are exactly two $i$'s such that $\ve_i(\pi_*') = -1$;
\item
Putting $\ve_i(\pi') := \prod_{s = 1}^\n \ve_i(\pi'_{w_s})$, we have
\begin{equation}\label{eq:multformulauunstable}
m(\pi') = \frac{1}{4}\lp 1 + \sum_{i = 1}^3 \ve_i(\pi')\rp.
\end{equation}
It is equal to 0 or 1.
\end{itemize}
 
For $* = a, b, c, d$, by Proposition \ref{prop:uustablebetainvariance} there exist nonzero intertwining operators $A_*$
in $\Hom_{\rG}(\pi_*', \beta\pi_*')$.  We normalize $A_*$ such that $A_*^n = 1$.
Let $\rf^*$ be a matrix coefficient of $\pi_*'$ which is normalized 
such that, for any tempered representation $\pi''$ of $\rG$, $\la \pi'', \rf^*\ra_{A_*}  = \delta_{\pi_*', \pi''}$.  
For $1 \leq s \leq \n$, the local packets $\{\pi_{w_s}\}$ (resp.\ $\rho_{i,j,w_s}$) are equivalent to one another.
Put
\[
D_\pi^* := \la \{\pi_{w_s}\}, f^*\ra,\; D_{\rho_{i, j}}^* := \la \rho_{i, j, w_s}, f_{H}^*\ra,
\quad * \in \{a, b, c, d\},\;i \neq j \in \{1, 2, 3\},
\]
where $f^*$ (resp.\ $f_H^*$) is a fixed test function on $G$ (resp.\ $H$) matching $\rf^*$.
Let
$\rf_{w_s} = \rf^*$, $f_{w_s} = f^*$ and $f_{H, w_s} = f_H^*$ for $1 \leq s \leq \n$.
By equation \eqref{eq:tracelocaluunstablen} and the multiplicity formula 
\eqref{eq:multformulauunstable}, we have:
\begin{equation}\label{eq:trickuu+}
\begin{split}
4 \cdot \xi^a(\n) &= (D_\pi^a)^\n \;\;+ \sum_{i \neq j \in \{1, 2, 3\}} \lp D_{\rho_{i, j}}^a\rp^\n, \quad \forall\, \n \geq 2;\\
4 \cdot \xi^*(\n) &= (D_\pi^*)^\n \;\;+ \sum_{i \neq j \in \{1, 2, 3\}} \lp D_{\rho_{i, j}}^*\rp^\n, 
\quad * \in \{b, c, d\}, \forall\, \n \text{ even};
\end{split}
\end{equation}
where $\xi^*(\n)$ ($* = a, b, c, d$) is an $n$-th root of unity dependent on $\n$.

Now let $\rf_{w_s} = \rf^a$, $f_{w_s} = f^a$ and $f_{H, w_s} = f_H^a$ for $1 \leq s \leq \n - 1$.
Let $\rf_{w_\n} = \rf^*$, $f_{w_\n} = f^*$ and $f_{H, w_\n} = f_H^*$, for $* = b, c, d$.  
Then, again by equation \eqref{eq:tracelocaluunstablen} and the multiplicity formula \eqref{eq:multformulauunstable}, 
we have
\[
0 =\lp D_\pi^a\rp^{\n - 1}D_\pi^*\;\; + \sum_{i \neq j \in \{1, 2, 3\}}\lp D_{\rho_{i, j}}^a\rp^{\n - 1}D_{\rho_{i, j}}^* \, ,
\quad * \in \{b, c, d\}.
\]
Multiplying each side of the above equation by the $(\n - 1)$-st power of a complex variable $z$,
and summing over $2 \leq m < \infty$,
we obtain the following equality of meromorphic functions:
\begin{equation}\label{eq:trickuu}
0 =  \frac{D_\pi^a D_\pi^*}{1 - D_\pi^a z}\;\; + \sum_{i \neq j \in \{1, 2, 3\}} 
\frac{D_{\rho_{i, j}}^a D_{\rho_{i, j}}^*}{1 - D_{\rho_{i, j}}^a z},
\quad * \in \{b, c, d\}.
\end{equation}

The values of $D_\pi^*, D_{\rho_{i,j}}^*$ may be computed using the same technique as in the proof of Proposition 
\ref{prop:tracelocalunstable}.
We leave the following as an exercise:
It follows from the equations \eqref{eq:trickuu+}
and the absence of poles on the left-hand side of \eqref{eq:trickuu}, that
we may index $\rho_{i,j, w_1}$ as $\rho_s$ ($s = 1, 2, 3)$ such that:
\begin{gather*}
\xi_n = D_\pi^a = D_{\rho_1}^a = \pm D_{\rho_2}^a = \pm D_{\rho_3}^a,\\
\xi_{2n}^* = D_\pi^* = - D_{\rho_1}^* = \pm D_{\rho_2}^* = \mp D_{\rho_3}^*
\end{gather*}
for some $n$-th root of unity $\xi_n$ and $2n$-th root of unity $\xi_{2n}^*$, for all $* \in \{b, c, d\}$.

Arguing as we did at the end of the proof of Proposition \ref{prop:tracelocalunstable}, 
multiplying each intertwining operator $A_*$ by a root of unity if necessary, and noting that the
twisted characters of inequivalent $\beta$-invariant representations are linearly independent, 
the proposition follows.
\end{proof}

The local packet $\{\pi'\}$ of $\rG$
which we considered in Proposition \ref{prop:uustablecaseii} is the lift of the nonsingleton packet 
$\{\pi\} = \{\pi_a, \pi_b, \pi_c, \pi_d\}$ of $G$ which is the lift of a local packet of $H$.
The next proposition shows that, in the case where $n = [\K':\K] = 3$,
there is a packet of $\rG$  of the form $\{\pi'_a, \pi'_b, \pi'_c, \pi'_d\}$ 
which is the lift of a {\it singleton} packet of $G$.

\begin{proposition}\label{prop:n3case}
Suppose $n = [\K':\K] = 3$.
Let $\theta$ be a character of $\K'^{k'}$ such that $\theta \neq \beta\theta$.  
Moreover, suppose $\omega' = \theta\circ\N_{E'/E}$.
Let $\rho' = \rho'(\theta, \beta\theta)$ be the local packet $($of cardinality $2$$)$ of $\rH$ 
which is the lift of  $\theta\otimes\beta\theta$ from $\K'^{k'} \times \K'^{k'}$.
Let $\{\pi'\} = \{\pi'_a, \pi'_b, \pi'_c, \pi'_d\}$ be the packet of $\rG$ which is the lift of $\rho'$,
where $\pi'_a$ is the unique generic representation in the packet.
Then, $\pi'_a$ is the \textbf{only} $\beta$-invariant representation in the packet, 
and there exists a cuspidal representation $\pi$ of $G$
and a nonzero intertwining operator $A \in \Hom_{\rG}(\pi'_a, \beta\pi'_a)$, such that$:$
\[
\la \pi'_a, \rf\ra_A = \la \pi, f\ra
\]
for all matching functions.
More precisely, $\pi$ is the sole member of the singleton packet of $G$ 
which lifts to the cuspidal monomial $\GL(3, \K)$-module $\pi(\kappa'^2\tilde{\theta}) = \kappa^2\pi(\tilde{\theta})$ 
associated with $\kappa'^2\tilde{\theta}$, where $\tilde{\theta}(z) := \theta(z/\alpha(z))$, $z \in \K'^\times$.  
In particular, $\kappa^2\pi(\tilde{\theta})$ is $\sigma$-invariant.
\end{proposition}
\begin{proof}
As before, construct totally imaginary number fields $F, F', E, E'$
such that $F_{w_i} = k$, $E_{w_i} = \K$, $F'_{w_i} = k'$ and $E'_{w_i} = \K'$, for $m \geq 2$ places $w_1,\ldots, w_\m$ of $F$ 
which stay prime in $E'$ and $F'$.
Let $S$ be the set of archimedean places of $F$
and those nonarchimedean places where at least one of $E/F, F'/F$ is ramified.
Construct a global character $\theta$ of $C_{E'}^{F'}$ such that:
(i) For $1 \leq i \leq \n$, $\theta_{w_i}$ is equal to the local character of $\K'^{k'}$
previously denoted by $\theta$ in the proposition; 
(ii) 
for each place $u \in S$ which does not split in $E$, $\theta_u = (\beta\theta)_u$;
(iii) $\theta_v$ is unramified for each place $v \notin S\cup\{w_1, \ldots, w_\n\}$.
(See proof of Proposition \ref{prop:uustablecaseii}.)
Let $\tilde{\theta}$ be the character of $C_{E'}$ defined by $\tilde{\theta}(z) := \theta(z/\alpha(z))$, $z \in C_{E'}$.
In particular, $\tilde{\theta} \neq \beta\tilde{\theta}$.

Let $\tilde{\pi} := \pi(\kappa'^2\tilde{\theta}) = \kappa^2\pi(\tilde{\theta})
$ be the cuspidal, monomial, automorphic representation of $\GL(3, \Ae)$ 
associated with the character $\kappa'^2\tilde{\theta}$ of $C_{E'}$.  
Then, $\tilde{\pi}_{w_i}$ ($1 \leq i \leq \n$) is equivalent to the local representation of $\GL(3, \K)$
previously denoted by $\pi(\kappa'^2\theta)$ in the proposition.  
The automorphic representation $\tilde{\pi}$ base-change lifts to the representation
$\tilde{\pi}' := \kappa'^2 I(\tilde{\theta},\beta\tilde{\theta},\beta^2\tilde{\theta})$ 
of $\GL(3, \mbb{A}_{E'})$ induced from the upper triangular Borel subgroup.  
Since $\kappa\alpha\kappa = \tilde{\theta}\alpha\tilde{\theta} = 1$, 
it is clear that $\tilde{\pi}'$ is $\sigma'$-invariant.  Hence, $\tilde{\pi}'$ is the
$b'$-lift of a packet $\{\pi'\}$ of $\mb{\rG}(\Af)$, by Theorem \ref{thm:u3gl3stableglobalpacket}.  

We want to show that $\{\pi'\}$ is the $b_G$-weak-lift of a stable packet of $\mb{G}(\Af)$.
Towards this end, we first show that the cuspidal representation $\tilde{\pi} = \kappa^2\pi(\tilde{\theta})$
of $\GL(3, \Ae)$ is $\sigma$-invariant.
Since, $\kappa\alpha\kappa = 1$, it suffices to show that $\pi(\tilde{\theta})$ is equivalent to $\sigma(\pi(\tilde{\theta}))$.
At almost every place $v$ of $F$, the number field extensions $E/F$, $F'/F$ and the character $\theta_v$ are unramified.
If such a place $v$ stays prime in $F'$ and $E$, 
then there exists a character $\mu_v$ of $E_v^\times$ such that $\tilde{\theta}_v = \mu_v \circ \N_{E'/E}$, 
and $\pi(\tilde{\theta}_v)$ is equivalent to the induced representation 
$\mu_v I(1, \ve_{E'/E, v}, \ve_{E'/E, v}^2)$, where $\ve_{E'/E}$ is the character of $C_E$ associated with
the number field extension $E'/E$ via global class field theory.
Since $\alpha$ commutes with $\beta$, we have $\alpha\ve_{E'/E, v} = \ve_{E'/E, v}$.
Since $\alpha\tilde{\theta}^{-1}_v = \tilde{\theta}_v$,
we have $\alpha\mu^{-1}_v = \ve_{E'/E, v}^i\mu_v$ for some $i = 0, 1$ or $2$.
Hence, $\sigma(\pi(\tilde{\theta}_v)) = \ve_{E'/E, v}^i\mu_v I(1, \ve^{-1}_{E'/E, v}, \ve^{-2}_{E'/E, v})$, 
which is equivalent to $\pi(\tilde{\theta}_v)$ because $\ve_{E'/E}$ is cyclic of order $3$.
By a similar argument, we also have $(\sigma\pi(\tilde{\theta}))_v \cong \pi(\tilde{\theta})_v$ when $v$ splits in $F'$ and/or $E$.
By the strong multiplicity one theorem for $\GL(3)$,
we conclude that the cuspidal automorphic representation $\tilde{\pi} = \kappa^2\pi(\tilde{\theta})$ is $\sigma$-invariant.
Hence, $\tilde{\pi}$ is the lift of a stable packet $\{\pi\}$ of $\mb{G}(\Af)$
by Theorem \ref{thm:u3gl3stableglobalpacket}.

The representation $\tilde{\pi}'$ is the $i'$-lift (i.e. parabolic induction) of $\sigma$-invariant representations
$\tilde{\rho}'$ of $\GL(2, \mbb{A}_{E'})$ of the form $\kappa'^2 I(\tilde{\theta}_1,\tilde{\theta}_2)$, 
where $\theta_1, \theta_2$ are distinct characters in $\{\theta, \beta\theta, \beta^2\theta\}$.
Moreover, no other representation of $\GL(2, \mbb{A}_{E'})$ lifts to $\tilde{\pi}'$.  
The representations $\tilde{\rho}'$ are $b_u'$-lifts of the unstable packets 
$\{\rho'\} := \rho'(\theta_1, \theta_2)$ of $\Un(2 , E'/F')(\mbb{A}_{F'})$,
which are the lifts of the pairs $(\theta_1, \theta_2)$ (see \cite[Cor.\ 7.2]{FU2}).  
These packets $\{\rho'\}$ all lift via the $L$-group homomorphism $e'$ to the same unstable packet $\pi'(\theta, \beta\theta)$ 
of $\mb{\rG}(\Af)$, which must be $\{\pi'\}$ by the commutativity of the $L$-group diagram in Figure \ref{fig:lgroups}.
It is clear that $\{\pi'\}$ must be the $b_G$-weak-lift of $\{\pi\}$, again by the commutativity of the $L$-group diagram.
We have therefore the commutative diagram of liftings in Figure \ref{fig:n3case}.

\begin{figure}[htp!]
\begin{gather*}
\xymatrix@1{
*+{\substack{{\kappa'}^2 I(\tilde{\theta}, \beta\tilde{\theta}) \\ \GL(2, \mbb{A}_{E'})}}
\ar@{.>}[rrr]_*+++\txt{\bf{}}^{i'} 
&&&
*+{
\substack{\kappa'^2 I(\tilde{\theta}, \beta\tilde{\theta}, \beta^2\tilde{\theta}) \\ \GL(3, \mbb{A}_{E'})}
}
\\ &
*+{\substack{ \rho'(\theta, \beta\theta) \\ \Un(2, E'/F')(\mbb{A}_{F'})} 
}
\ar@{.>}[lu]_{b_u'}\ar@{.>}[r]^{e'}
&
*+{\substack{ \{\pi'\} \\ \Un(3, E'/F')(\mbb{A}_{F'})}} \ar@{.>}[ru]^{b'}
&
\\ &&
*+{\substack{\{\pi\} \\ \Un(3, E/F)(\Af)}} \ar[u]_{b_G} \ar@{.>}[dr]^b &
\\ &&&
*+{\substack{\kappa^2\pi(\tilde{\theta}) \\ \GL(3, \Ae)}}
\ar@{.>}[uuu]_{b_{3}}^*++++++++\txt{\bf{}}
}
\end{gather*}
\caption{}\label{fig:n3case}
\end{figure}

Applying Lemma \ref{lemma:traceglobalstable} to $\{\pi'\}$, we obtain:
\begin{equation}\label{eq:n3eq}
\sum_{\pi' \in \{\pi'\}}m(\pi')\ep(\pi')\la \pi', \rf\times\beta\ra_{S \cup \{w_1, \ldots, w_\n\}}
= \la \{\pi\}, f\ra_{S \cup \{w_1,\ldots, w_\n\}}.
\end{equation}
At each place $u \in S$ which does not split in $E$, by the construction of $\theta$ 
the local packet $\{\pi'_u\} = \pi'(\theta_u, (\beta\theta)_u) = \pi'(\theta_u, \theta_u)$ 
consists of a single irreducible principal series representation
(\cite[Part 2.\ I.4.3]{F}).
For each $u \in S$ which splits in $E$, we have $\rG_u = \GL(3, F'_u)$ 
and $\{\pi'_u\} = I(\theta_u, (\beta\theta)_u, (\beta^2\theta)_u)$.

For $1 \leq i \leq \n$, the local component $\{\pi'\}_{w_i}$ of the unstable global packet $\{\pi'\}$ 
is the local packet  $\{\pi'_0\} := \{\pi'_a, \pi'_b, \pi'_c, \pi'_d\}$ in the proposition.
Namely, $\{\pi'_0\}$ is the lift of the local packet $\rho'(\theta_{w_i}, (\beta\theta)_{w_i})$ of $\rH_{w_i} = \Un(2, \K'/k')$.
(Note that $\theta_{w_1} = \theta_{w_2} = \ldots = \theta_{w_\n}$.)
Via the same argument as in the proof of Proposition \ref{prop:unstablebetainvariance}, 
but using the definition of the $L$-homomorphism $b_G$ instead of $e'$, the rigidity theorem for the global packets of $\Un(3)$
implies that:
$
\{\pi'_a, \pi'_b, \pi'_c, \pi'_d\} = \{\beta\pi'_a, \beta\pi'_b, \beta\pi'_c, \beta\pi'_d\}.
$
Since $n = 3$, and $\pi'_a$ is the unique generic representation in the packet, either all four members of $\{\pi'_0\}$
are $\beta$-invariant, or only $\pi'_a$ is.  In either case, the left-hand side of equation \eqref{eq:n3eq} is not identically zero.
Using Lemma \ref{lemma:tracelocalinduced}
and the linear independence of characters to cancel the principal series local components in \eqref{eq:n3eq}, we obtain:
\begin{equation}\label{eq:traceglobaldiscrete3}
\sum_{\pi' \in \{\pi'\}} m(\pi')\ep(\pi')\prod_{i = 1}^\n \la \pi'_{w_i}, \rf_{w_i}\times\beta_{w_i}\ra 
= \prod_{i = 1}^\n \la \{\pi_{w_i}\}, f_{w_i}\ra.
\end{equation}

We want to show that not all four members of $\{\pi'_0\}$ are $\beta$-invariant.
Suppose they are, then in particular $\pi'_b$ is $\beta$-invariant.
Let $A_b$ be a nonzero intertwining operator in $\Hom_{\rG}(\pi'_b, \beta\pi'_b)$.
Let $\rf^b$ be a matrix coefficient of $\pi'_b$, normalized so that $\la \pi'_b, \rf^b\ra_{A_b} = 1$.
For $1 \leq i \leq \n$, the local packets $\{\pi\}_{w_i}$ of $G_{w_i} = G$ are all equivalent to the same local packet, 
which we denote by $\{\pi_0\}$.
We choose the matching global test functions such that $\rf_{w_i} = \rf^b$ for $1 \leq i \leq \n$,
and $f_{w_i}$ is a fixed local test function $f^b$ on $G$ matching $\rf^b$.
For such global test functions, equation \eqref{eq:traceglobaldiscrete3} and the multiplicity formula \eqref{eq:multformulauunstable}
for $\{\pi'\}$ imply that:
\[
\langle\{\pi_0\}, f^b\rangle^\n = \begin{cases}
\xi_n(\n) & \text{ if $\n$ is even},\\
0 & \text{ if $\n$ is odd},
\end{cases}
\]
where $\xi_n(\n)$ is an $n$-th root of unity.  Clearly, no such number $\la \{\pi_0\}, f^b\ra$ exists.  
Hence, $\pi'_a$ is the only $\beta$-invariant representation in the local packet $\{\pi'_0\}$.
Let $A_a$ be a nonzero intertwining operator in $\Hom_{\rG}(\pi'_a, \beta\pi'_a)$.
By equation \eqref{eq:traceglobaldiscrete3} and the multiplicity formula for $\{\pi'\}$, we have: 
\begin{equation}\label{eq:traceglobaldiscrete31}
\xi_n(\n)\prod_{i = 1}^\n \la \pi'_a, \rf_{w_i}\ra_{A_a}
= \prod_{i = 1}^\n \la \{\pi_0\}, f_{w_i}\ra,
\end{equation}
where $\xi_n(\n)$ is an $n$-th root of unity dependent on $\n$.  

Let $\rf^a$ be a matrix coefficient of $\pi'_a$, normalized so that $\la \pi'_a, \rf^a\ra_{A_a} = 1$.
Let $f^a$ be a local test function on $G$ matching $\rf^a$.
Similar to what we did before, 
we choose the matching global test functions such that $\rf_{w_i} = \rf^a$, $f_{w_i} = f^a$ for $1 \leq i \leq \n$.
Then, equation \eqref{eq:traceglobaldiscrete31} 
implies that:
\[
\xi_n(\n)\la \pi'_a, \rf^a\ra_{A_a}^\n =  \xi_n(\n) = \la \{\pi_0\}, f^a\ra^\n,\quad \forall\, \n \geq 2.
\]
Hence, $\la \{\pi_0\}, f^a\ra$ is an $n$-th root of unity $\xi_0$.

Now, let $\n = 2$.  Choose the global test functions such that $\rf_{w_2} = \rf^a$, $f_{w_2} = f^a$,
and $\rf_{w_1}, f_{w_1}$ are arbitrary matching local test functions.
Using equation \eqref{eq:traceglobaldiscrete31} again, we have:
\[
\begin{split}
\xi_n(2) \la \pi'_a, \rf_{w_1}\ra_{A_a} &= \xi_0 \la \{\pi_0\}, f_{w_1}\ra.
\end{split}
\]
Multiplying $A_a$ by an $n$-th root of unity if necessary, 
we have
\[
\la \pi'_a, \rf_{w_1}\ra_{A_a} = \la \{\pi_0\}, f_{w_1}\ra.
\]

Recall that the global packet $\{\pi\}$ of $\mb{G}(\Af)$ lifts to the cuspidal, monomial, automorphic representation 
$\tilde{\pi} = \kappa^2\pi(\tilde{\theta})$ of $\GL(3, \mbb{A}_{E})$, and that $\tilde{\pi}$ is $\sigma$-invariant.
Hence, $\{\pi_0\} = \{\pi\}_{w_i}$ ($1 \leq i \leq \n$) 
lifts to the $\sigma$-invariant, cuspidal, monomial representation 
$\tilde{\pi}_{w_i} = \kappa_{w_i}^2\pi(\tilde{\theta}_{w_i})$ of $\GL(3, E_{w_i}) = \GL(3, \K)$.
By Proposition \ref{prop:u3singletonlocalpacket},
the local packet $\{\pi_0\}$ is a singleton.  
The proof of the proposition is now complete.
\end{proof}
\begin{remark}
The packet $\{\pi'\}$ is the lift of the packet $\rho'(\theta, \beta\theta)$ on $\rH$, 
which is not $\beta$-invariant.
Hence, we cannot apply the Howe correspondence argument in the proof of Proposition \ref{prop:uustablebetainvariance} to this case
and show that each member of $\{\pi'\}$ is $\beta$-invariant.
\end{remark}

\subsection{Other Representations}
Let $\nu$ denote the normalized absolute value character of $\K^\times$.
Let $\mu$ be a character of $\K^k$.  Let $\mu'$ denote the character $\mu\circ\N_{\K'/\K}$ of $\K'^{k'}$.
Then, $\mu$ (resp.\ $\mu'$) defines via composition with the determinant a one-dimensional representation of $G$ (resp.\ $\rG$).
As a one-dimensional representation of $\rG$, $\mu'$ is clearly $\beta$-invariant.
Let $\tilde{\mu}$ be the character of $\K^\times$ defined by $\tilde{\mu}(z) = \mu(z/\alpha(z))$ for all $z \in \K^\times$.
The one-dimensional representation $\mu$ of $G$ 
is the unique nontempered quotient of the length two parabolically induced representation $I := I_G(\tilde{\mu}\nu, \mu)$
of $G$. 
The other constituent of $I$ is a Steinberg square-integrable subrepresentation,
which we denote by $\st(\mu)$.  
Likewise, the one-dimensional representation $\mu'$ of $\rG$ is the unique nontempered constituent of the length two
parabolically induced representation $I' := I_{\rG}(\tilde{\mu}'\nu', \mu')$ of $\rG$, where $\nu' = \nu\circ\N_{\K'/\K}$ is the absolute value character of $\K'^\times$.
We denote the unique square-integrable subrepresentation of $I'$ by $\st(\mu')$.

Let $A$ be the intertwining operator in $\Hom_{\rG}(I', \beta I')$ defined in the paragraph preceding Lemma
\ref{lemma:tracelocalinduced}.  Since $\st(\mu')$ is the unique square-integrable subrepresentation of $I'$, 
the restriction of $A$ to $\st(\mu')$ defines a nonzero intertwining operator, which we denote also by $A$,
in $\Hom_{\rG}(\st(\mu'), \beta\st(\mu'))$.  
Since $A^n = 1$ is a scalar on $I'$ and $A\st(\mu') \subseteq \st(\mu')$,
we have $A\varphi \notin \st(\mu')$ for all $\varphi \in I'$ which lies outside of $\st(\mu')$. 
Consequently, $A$ induces a nonzero operator $A(\mu')$ in $\Hom_{\rG}(\mu', \beta\mu')$.
Since $\mu'$ is one-dimensional, we may, and do, normalize $A$ such that $A(\mu') = 1$.

\begin{proposition}\label{prop:tracelocalsteinberg}
The following character identity holds for all matching functions $\rf$ on $\rG$ and $f$ on $G$$:$
\[
\la \st(\mu'), \rf\ra_A = \la \st(\mu), f\ra.
\]
\end{proposition}
\begin{proof}
Let $F, F', E, E' = F'E$ be number fields, with a place $w_1$ of $F$ which stays prime in $E$ and $F'$, such that
$F_{w_1} = k$, $E_{w_1} = \K$, $F'_{w_1} = k'$ and $E'_{w_1} = \K'$.
We construct a character of $C_{E}^{F}$ whose local component at $w_1$ is the local character $\mu$ in the proposition.
We denote this global character also by $\mu$, 
and denote within this proof what was previously $\mu$ in the proposition by $\mu_{w_1}$.
We let $\mu'$ denote the character $\mu\circ\N_{E'/E}$ of $C_{E'}^{F'}$.
The characters $\mu$ and $\mu'$ define one-dimensional representations of $\mb{G}(\Af)$ and $\mb{\rG}(\Af)$, 
respectively, via composition with the determinant.

Let $\tilde{\mu}'$ denote the character of $C_{E'}$ defined by $\tilde{\mu}'(z) = \mu'(z/\alpha(z))$
for all $z$ in $C_{E'}$.
The one-dimensional representation $\mu'$ of $\mb{\rG}(\Af)$ lifts via the $L$-group homomorphism $b'$ 
(see Figure \ref{fig:lgroups})
to the one-dimensional $\sigma$-invariant representation $\tilde{\mu}' := \tilde{\mu}'\circ\det$ of $\GL(3, \Ae)$.
The representation $\mu'$ of $\mb{\rG}(\Af)$ belongs to a singleton stable packet.

By equation \eqref{eq:traceglobalstable}, we have:
\begin{equation}\label{eq:tracelocal1dim}
\la \mu', \rf \times \beta\ra_S 
= \la \mu, f\ra_S,
\end{equation}
where $S$ is a finite set of places containing $w_1$, 
such that at the places outside of $S$ the character $\mu$ and the number field extensions $E/F$, $F'/F$ are unramified.

Fix the local components of the test functions at each place $v \in S - \{w_1\}$, 
and treat $\la \mu'_v, \rf_v \times \beta_v\ra, \la \mu_v, f_v\ra$ as constants.  
We may, and do, choose the test functions such that these constants are nonzero, 
since both sides of equation \eqref{eq:tracelocal1dim} define nonzero distributions.
Thus, we have:
\[
C \la \mu'_{w_1}, \rf_{w_1} \times \beta_{w_1}\ra = \la \mu_{w_1}, f_{w_1} \ra
\]
for some nonzero constant $C$.

By Lemma \ref{lemma:tracelocalinduced}, we have:
\[
\la \st(\mu'_{w_1}), \rf_{w_1}\ra_A + \la \mu'_{w_1}, \rf_{w_1}\ra_A 
= \la \st(\mu_{w_1}), f_{w_1}\ra + 
\underbrace{\la \mu_{w_1}, f_{w_1}\ra}_{=\; C \la \mu'_{w_1}, \rf_{w_1} \times \beta_{w_1}\ra}.
\]
By Schur's Lemma, 
the twisted characters $\la \mu'_{w_1}, \rf_{w_1}\ra_A$ and $\la \mu'_{w_1}, \rf_{w_1} \times \beta_{w_1}\ra$
are scalar multiples of each other.
The representations $\st(\mu'_{w_1})$ and $\st(\mu_{w_1})$ are square-integrable, so their central exponents decay.
The representations $\mu'_{w_1}$ and $\mu_{w_1}$ are nontempered, so their central exponents do not decay.
By a well-known theorem of Harish-Chandra's, 
the character of an admissible representation $\pi$ is represented by a locally integrable function $\chi_{\pi}$.  
By the Deligne-Casselman theorem \cite{C}, the central exponents of $\pi$ may be computed from $\chi_{\pi}$.
Lastly,
by a well-known argument employing the (twisted) Weyl integration formula and the linear independence of central exponents 
(\cite[Sec.\ 21]{FK}), we conclude that:
\[\la \mu'_{w_1}, \rf_{w_1}\ra_A = C \la \mu'_{w_1}, \rf_{w_1} \times \beta_{w_1}\ra,\]
\[
\la \st(\mu'_{w_1}), \rf_{w_1}\ra_A  = \la \st(\mu_{w_1}), f_{w_1}\ra.
\]
\end{proof}

The parabolically induced representation $I = I_H(\tilde{\mu}\nu^{1/2})$ of $H = \Un(2, \K/k)$
has a unique square-integrable subrepresentation $s(\mu)$ 
and a unique nontempered quotient which is the one-dimensional representation $\mu := \mu\circ\det$ of $H$.
The representation $I_{\rH}(\tilde{\mu}'\nu'^{1/2})$ of $\rH = \Un(2, \K'/k')$
likewise has a unique square-integrable subrepresentation $s(\mu')$ and a unique nontempered quotient which is the
one-dimensional representation $\mu' := \mu'\circ\det = \mu\circ\N_{\K'/\K}\circ\det$ of $\rH$.

Let $\pi(s(\mu)) = \{\pi^+, \pi^-\}$ (resp.\ $\pi'(s(\mu')) = \{\pi'^+, \pi'^-\}$) be the local packet of $G$ 
(resp.\ $\rG$) which is the lift of $s(\mu)$ from $H$ (resp.\ $s(\mu')$ from $\rH$).
The representations $\pi'^+$ and $\pi^+$ are square-integrable, while $\pi'^-$ and $\pi^-$ are cuspidal 
(\cite[p.\ 214, 215]{F}).
\begin{lemma}\label{lemma:tracelocalsqint}
There exist nonzero intertwining operators $A^\dagger \in \Hom_{\rG}(\pi'^\dagger, \beta\pi'^\dagger)$
$(\dagger$ is $+$ or $-)$  
such that the following system of local character identities holds for matching functions$:$
\[\begin{split}
2 \la \pi'^+, \rf\ra_{A^+} =& \la\pi(s(\mu)), f\ra + \la s(\mu), f_H\ra,\\
2 \la \pi'^-, \rf\ra_{A^-} =& \la\pi(s(\mu)), f\ra - \la s(\mu), f_H\ra.
\end{split}\]
\end{lemma}
\begin{proof}
Let $F, F', E, E'$ be the same number fields as in the proof of Proposition \ref{prop:tracelocalsteinberg}.
Let $w$ be a nondyadic nonarchimedean place of $F$, different from $w_1$, which does not split in $E$ or $F'$.
Fix a cuspidal representation $\rho_0$ of $H_w$ which lies in a singleton packet.  
Construct a cuspidal automorphic representation $\rho$ of $\mb{H}(\Af)$ such that $\rho_{w_1} = s(\mu)$,
$\rho_{w} = \rho_0$, and $\rho_{v}$ is an irreducible principal series representation for each place $v \neq w_1, w$.

Let $\rho_0'$ be the base-change lift of $\rho_0$ to $\rH_w = \Un(2, E'_w/F'_w)$.
Let $\pi(\rho_0)$ (resp.\ $\pi'(\rho_0') = \{\pi_0'^+, \pi_0'^-\}$) be the packet of $G_w$ (resp.\ $\rG_w$)
which is the lift of $\rho_0$ (resp.\ $\rho_0'$).
Let $\rho'$ be the lift of $\rho$ to $\mb{\rH}(\Af)$ via the $L$-homomorphism $b_H$.
Let $\pi'(\rho')$ be the packet of $\mb{\rG}(\Af)$ which is the $e'$-lift of $\rho'$.
By the construction of $\rho$,
each dyadic local component of $\pi'(\rho')$ is a local packet containing a constituent of a
parabolically induced representation.  Hence, the members of $\pi'(\rho')$ occur in the discrete spectrum
of $\mb{\rG}$ with multiplicity at most $1$.

We now apply the same argument which we have already used repeatedly:
By equation \eqref{eq:traceglobalunstable1}, the multiplicity formula for the packet $\pi'(\rho')$
(see proof of Proposition \ref{prop:tracelocalunstable}), 
and the cancellation of the principal series local components using the linear independence of characters,
we have:
\begin{multline}\label{eq:tracelocalsqint}
2 \ep^+\la \pi'^+, \rf_{w_1}\times\beta_{w_1}\ra \la \pi_0'^+, \rf_{w}\times\beta_{w}\ra +
2 \ep^-\la \pi'^-, \rf_{w_1}\times\beta_{w_1}\ra \la \pi_0'^-, \rf_{w}\times\beta_{w}\ra \\
=
\la \pi(s(\mu)), f_{w_1}\ra \la \pi(\rho_0), f_{w}\ra +
\la s(\mu), f_{H, w_1}\ra \la \rho_0, f_{H, w}\ra.
\end{multline}
The coefficients $\ep^\pm$ are $n$-th roots of unity, which we may assume to be $1$, since we can absorb them into the intertwining operators $A(\pi'^\pm)$ associated with the twisted characters $\la \pi'^\pm, \rf_w\times\beta_w\ra$.
By Proposition \ref{prop:unstablebetainvariance}, 
the representations $\pi_0'^+$ and $\pi_0'^-$ are $\beta_w$-invariant.
Choose $\rf_{w}$ to be a matrix coefficient $\rf^+$ of the cuspidal representation $\pi_0'^+$, 
normalized such that $\la \pi_0'^+, \rf^+\times\beta_{w}\ra = 1$.
The system of equations in Proposition \ref{prop:tracelocalunstable} implies that:
\[
1 = \la \pi(\rho_0), f_{w}\ra = \la \rho_0, f_{H, w}\ra,
\]
where $f_{w}$ (resp.\ $f_{H, w}$) is any test function on $G_w$ (resp.\ $H_w$) which matches $\rf^+$. 
For global test functions with such local components at the place $w$, equation \eqref{eq:tracelocalsqint} gives:
\[
2 \la \pi'^+, \rf_{w_1} \times \beta_{w_1}\ra = \la \pi(s(\mu)), f_{w_1}\ra + \la s(\mu), f_{H, w_1}\ra.
\]
Likewise, by choosing $\rf_{w}$ to be a matrix coefficient of $\pi_0'^-$, we obtain:
\[
2 \la \pi'^-, \rf_{w_1} \times \beta_{w_1}\ra = \la \pi(s(\mu)), f_{w_1}\ra - \la s(\mu), f_{H, w_1}\ra.
\]
\end{proof}

Let $\pi(\mu) = \{\pi^\times, \pi^-\}$ be the local quasi-packet of $G$ which is the lift of 
the one-dimensional representation $\mu$ 
of $H$ (see \cite[p.\ 214, 215]{F}).
Let $\pi'(\mu') = \{\pi'^\times, \pi'^-\}$ be the local quasi-packet of $\rG$ 
which is the lift of the one-dimensional representation $\mu'$ of $\rH$.
\begin{proposition}\label{prop:tracelocalonedim}
There exist nonzero operators $A^\dagger$ in $\Hom_{\rG}(\pi'^\dagger, \beta\pi'^\dagger)$ 
$(\dagger$ is $+$ or $-)$  
such that the following system of character identities holds for matching functions$:$
\[\begin{split}
\la\pi'^\times, \rf\ra_{A^\times} + \la\pi'^-, \rf\ra_{A^-} &= \la \mu, f_H\ra, \\
\la\pi'^\times, \rf\ra_{A^\times} - \la\pi'^-, \rf\ra_{A^-} &= \la \pi^\times, f\ra - \la \pi^-, f\ra.
\end{split}\]
\end{proposition}
\begin{proof}
The representations $\pi^+$ and $\pi^\times$ are the unique square-integrable subrepresentation and nontempered quotient, 
respectively, of the parabolically induced representation $I := I_G(\kappa\tilde{\mu}\nu^{1/2}, \eta)$ of $G$,
where $\eta$ is the character of $\K^k$ such that the central character of $I$ is our fixed $\omega$.
Likewise, the representations $\pi'^+$ and $\pi'^\times$ are the unique square-integrable and nontempered constituents of
the representation $I' := I_{\rG}(\kappa'\tilde{\mu}'\nu'^{1/2},\eta')$ of $\rG$, 
where the prime on a character denotes composition with the norm map $\N_{\K'/\K}$.

The intertwining operator $A$ on $I'$, as defined in the paragraph preceding Lemma \ref{lemma:tracelocalinduced},
must send subrepresentation to subrepresentation, and quotient to quotient.  
Thus, it induces intertwining operators $A^+$ and $A^\times$ on $\pi'^+$ and $\pi'^\times$, respectively.
Since the representations $\mu$ and $s(\mu)$ are the constituents of $I_H(\tilde{\mu}\nu^{1/2})$, 
by Lemma \ref{lemma:tracelocalinduced} we have:
\begin{equation}\label{eq:tracelocalnontempered1}
\la \pi'^+, \rf\ra_{A^+} + \la \pi'^\times, \rf\ra_{A^\times} = \la \mu, f_H\ra + \la s(\mu), f_H\ra.
\end{equation}
By Lemma \ref{lemma:tracelocalsqint}, we have
\begin{equation}\label{eq:tracelocalnontempered2}
\xi_n \la \pi'^+, \rf\ra_{A^+} - \la \pi'^-, \rf\ra_{A^-} = \la s(\mu), f_H\ra
\end{equation}
for some $n$-th root of unity $\xi_n$.
The presence of $\xi_n$ is due to the fact that the intertwining operator $A^+$ 
obtained in Lemma \ref{lemma:tracelocalsqint} may differ by an $n$-th root of unity from the $A^+$ defined by $A$.
Subtracting equation \eqref{eq:tracelocalnontempered2} from \eqref{eq:tracelocalnontempered1}, we obtain:
\[
(1 - \xi_n)\la \pi'^+, \rf\ra_{A^+} + \la \pi'^\times, \rf\ra_{A^\times} + \la \pi'^-, \rf\ra_{A^-}
= \la \mu, f_H\ra.
\]
The central exponents of the square-integrable representation $\pi'^+$ decay, 
while those of the nontempered representations $\pi'^\times$ and $\mu$ do not.
Since $\pi'^-$ is cuspidal, its central exponents are zero.
As in the proof of Proposition \ref{prop:tracelocalsteinberg}, 
it follows from the linear independence of central exponents that $1 - \xi_n = 0$.
Hence,
\[
\la \pi'^\times, \rf\ra_{A^\times} + \la \pi'^-, \rf\ra_{A^-}
= \la \mu, f_H\ra.
\]

By Lemma \ref{lemma:tracelocalsqint}, and the fact that $\pi^+$ and $\pi^-$ are the members of the local packet $\pi(s(\mu))$,
 the following holds for matching functions:
\begin{equation}\label{eq:tracelocalnontempered3}
\la \pi'^+, \rf\ra_{A^+} + \la \pi'^-, \rf\ra_{A^-} = \la \pi(s(\mu)), f\ra = \la \pi^+, f\ra + \la \pi^-, f\ra.
\end{equation}
On the other hand, by Lemma \ref{lemma:tracelocalinduced} we have:
\begin{equation}\label{eq:tracelocalnontempered4}
\la \pi'^+, \rf\ra_{A^+} + \la \pi'^\times, \rf\ra_{A^\times} = \la \pi^+, f\ra + \la \pi^\times, f\ra.
\end{equation}
Subtracting equation \eqref{eq:tracelocalnontempered3} from \eqref{eq:tracelocalnontempered4}, we obtain:
\[
\la \pi'^\times, \rf\ra_{A^\times} - \la \pi'^-, \rf\ra_{A^-}
=  \la \pi^\times, f\ra - \la \pi^-, f\ra.
\]
\end{proof}

\section{Global Classification}
Let $F, F', E ,E'$ be the number fields described in the Introduction.

\begin{proposition}\label{prop:classification2}
Every discrete spectrum automorphic representation of $\mb{G}(\Af)$ or $\mb{H}(\Af)$ 
lifts to a $($quasi-$)$ packet of $\mb{\rG}(\Af)$ which contains a $\beta$-invariant,
discrete spectrum, automorphic representation.
\end{proposition}
\begin{proof}
By Proposition \ref{prop:weaklift}, each discrete spectrum automorphic representation of
$\mb{G}(\Af)$ or $\mb{H}(\Af)$ weakly lifts to a (quasi-) packet of $\mb{\rG}(\Af)$ via the $L$-homomorphism
$b_G$ or $e_H$, respectively.
What remains to be shown is that every such packet of $\mb{\rG}(\Af)$
contains a $\beta$-invariant, discrete spectrum, automorphic representation.

Recall that a global (quasi-) packet $\{\pi'\}$ is a restricted tensor product of local packets $\{\pi'_v\}$.
Suppose $\{\pi'\}$ is a (quasi-) packet of $\mb{\rG}(\Af)$ 
which is a weak lift from either or both of $\mb{G}(\Af)$ and $\mb{H}(\Af)$.  
Then, by the classification of the images of the weak liftings $b_G$ and $e_H$ in Section \ref{sec:discretespectrumpackets}, 
the local components $\{\pi'_v\}$ of $\{\pi'\}$ 
are precisely the local (quasi-) packets which are examined in Section \ref{sec:localtraceidentities}. 
In particular, by Lemma \ref{lemma:tracelocalinduced} and Propositions \ref{prop:traceglobaldiscrete1},
\ref{prop:unstablebetainvariance}, \ref{prop:uustablebetainvariance}, \ref{prop:n3case}, \ref{prop:tracelocalonedim},
for each place $v$ the local packet  $\{\pi'_v\}$ contains at least one $\beta$-invariant representation $\pi'_{0, v}$.
If $\{\pi'_v\}$ is of the form $\{\pi'^+_v, \pi'^-_v\}$, we take $\pi'_{0, v}$ to be $\pi'^+_v$.
If $\{\pi'_v\}$ is of the form $\{\pi'_{v, a}, \pi'_{v, b}, \pi'_{v, c}, \pi'_{v, d}\}$, 
we let $\pi'_{0, v}$ be the unique generic member $\pi'_{v, a}$.
If $\{\pi'\}$ is the global quasi-packet which is the $e'$-lift of a 
$\beta$-invariant one-dimensional representation $\mu'$ of $\mb{\rH}(\Af)$,
then $\{\pi'_v\}$ is of the form $\{\pi'^\times_v, \pi'^-_v\}$.  In this case, we let $\pi'_{0, v} = \pi'^\times_v$ 
for all $v$ if the factor  $\ve(\tilde{\mu}', \kappa')$ (see \cite[p.\ 387]{F}) 
is equal to $1$, and we let $\pi'_{0, v}$ be $\pi'^\times_v$
for all but one $v$ if $\ve(\tilde{\mu}', \kappa') = -1$. 
(Note: It is conjectured by \cite{A1} and \cite{Ha} that $\ve(\tilde{\mu}', \kappa')$ 
is equal to the epsilon factor $\ve(\frac{1}{2}, \tilde{\mu}'\kappa')$.)
By the multiplicity formulas of the global packets which are $b_G$- or $e_H$-weak-lifts,
the tensor product $\otimes_v \pi'_{v, 0} \in \{\pi'\}$ 
is equivalent to a discrete spectrum automorphic representation of $\mb{\rG}(\Af)$, 
and it is $\beta$-invariant by construction.
\end{proof}

Theorem \ref{thm:classification} now follows from Propositions \ref{prop:classification1} and \ref{prop:classification2},
and all the local character identities which we have derived, 
since the condition that the residual characteristics of the local fields be odd may be removed 
if we assume that the multiplicity one theorem holds for $\Un(3)$ in general.

\appendix
\section{Discretely Occurring Representations}\label{appendix:DOR}
In this section, 
we compute the representations which occur discretely in the spectral side of the $\beta$-twisted trace formula of $\mb{\rG}(\Af)$.
The computation follows from an explicit recipe established in \cite{CLL}, which we now describe.

Let $F$ be a number field, with ring of ad\`eles $\Af$.  
For any finite place $v$ of $F$, let $\mc{O}_v$ denote the ring of integers of $F_v$.
Let $\mb{H}$ be a reductive $F$-group, with center $\mb{Z}$.
Let $\sigma$ be an automorphism of $\mb{H}(\Af)$.

Fix a character $\omega$ of $\mb{Z}(F)\bs\mb{Z}(\Af)$.
For any place $v$ of $F$, put $H_v := \mb{H}(F_v)$.
Let $C(H_v, \omega_v)$ be the space of smooth functions $f_v$ on $H_v$
such that $f_v$ is compactly supported
modulo $Z_v := \mb{Z}(F_v)$ and $f_v(zh) = \omega_v^{-1}(z)f(h)$ for all $z \in Z_{v}, h \in H_v$.

For each place $v$ of $F$, fix a maximal compact subgroup $K_v$ of $H_v$ such that
$K_v$ is the hyperspecial subgroup $\mb{H}(\mathcal{O}_v)$ if $\mb{H}$ is unramified as an $F_v$-group.
Let $\mathcal{H}(H_v, \omega_v)$ denote the Hecke algebra of $K_v$-biinvariant (resp.\ $K_v$-finite) functions
in $C(H_v, \omega_v)$ if $v$ is nonarchimedean (resp.\ archimedean).

Let $C(\HH, \omega)$ denote the span of the smooth, compactly supported functions on $\HH$ which are
of the form $\otimes_v f_v$, where $f_v \in C(H_v, \omega_v)$
for all $v$ and $f_v$ is a unit in the Hecke algebra $\mc{H}(H_v, \omega_v)$ for almost all nonarchimedean $v$.
In this work, whenever we mention a function $f \in C(\HH, \omega)$, we assume that
$f$ is a tensor product of local components $f_v$.  

Fix a minimal parabolic subgroup $\mb{P}_0$ of $\mb{H}$.  Let $\mb{A}_0$ be the maximal $F$-split component
of the Levi subgroup of $\mb{P}_0$.
Let $\mb{M}$  be a Levi subgroup of $\mb{H}$.
Let $\mb{A}_M$ denote the
split component of the center of $\mb{M}$.
Let $X_*(\mb{A}_M) := \Hom(\mbb{G}_m, \mb{A}_M)$.
Let $\mf{a}_M$ denote $X_*(\mb{A}_M)\otimes_{\mbb{Z}}\mbb{R}$, and $\mf{a}_M^*$ its dual. 
Let $\mc{P}(M)$ denote the set of parabolic subgroups of $\mb{H}$ with Levi component $\mb{M}$.

For any parabolic subgroup $\mb{P} \in \mc{P}(M)$, representation $\tau$ of $\mb{M}(\Af)$, 
and element $\zeta \in (\mf{a}_M/\mf{a}_H)^*$,
let $I_{P, \tau}(\zeta)$ denote the $\HH$-module normalizedly induced from the $\mb{P}(\Af)$-module
which sends an element $p$ in $\mb{P}(\Af)$ to $\chi_{\zeta}(m)\tau(m)$,
where $m \in \mb{M}(\Af)$ is the Levi component of $p$. 
Here, \[\chi_\zeta(m) := e^{\la \zeta, \ol{H_M({m})}\ra},\] 
where $H_M(m) \in \mf{a}_M$ is the standard notation for the ``logarithm'' of $m$,
and $\ol{H_M(m)}$ is its image in $\mf{a}_M/\mf{a}_H$.

In other words, $I_{P, \tau}(\zeta)$ is the right-regular representation on the space
of smooth functions $\varphi$ on $\HH$ which satisfy:
\[
\varphi(p g) = (\delta^{1/2}\chi_\zeta \otimes \tau)(m)\varphi(g), \quad  
\forall\, p \in \mb{P}(\Af), \; g \in \HH.  
\]
Here, $\delta(m) := \abs{\det\lp{\rm Ad}\;m|_{\mf{n}}\rp}$,
where $\mf{n}$ denotes the Lie algebra of the unipotent component of $\mb{P}$.
For any function $f$ in $C(\mb{H}(\Af), \omega)$, let
$I_{P, \tau}(\zeta, f)$ denote the convolution operator
\[
\int_{\mb{Z}(\Af)\bs\mb{H}(\Af)}\lp I_{P, \tau}(\zeta)\rp(h)f(h)\;dh,
\]
where $dh$ is a fixed Tamagawa measure on $\mb{H}(\Af)$.

Let $W(A_0, H)$ denote the Weyl group of $\mb{A}_0$ in $\mb{H}$.
Suppose $s$ is an element in $W(A_0, H)$ which satisfies $s\sigma\mb{M} = \mb{M}$.
For each $\mb{M}(\Af)$-module $\tau$,
there exists a family of intertwining operators of $\mb{H}(\Af)$-modules:
\[ 
M_{P|\sigma P}(s, \zeta) : I_{\sigma P, s^{-1}\tau}(s^{-1}\zeta) \longrightarrow I_{P, \tau}(\zeta),
\quad \zeta \in (\mf{a}_M/\mf{a}_H)^*,
\]
which is meromorphic in $\zeta$.  Here, $s^{-1}\tau$ is the $\sigma \mb{M}(\Af)$-module: 
\[
s^{-1}\tau : m \mapsto  \tau(sm), \quad \forall\, m \in \sigma\mb{M}(\Af).
\]

Let $I_{P, \tau}(\sigma)$ be the operator on the space of
$I_{P, \tau}$ which sends $\varphi(g)$ to $\varphi(\sigma g)$.  
Put 
\[
I_{P, \tau}(\zeta, f\times\sigma) := I_{P, \tau}(\zeta, f)I_{P, \tau}(\sigma).
\]

The spectral side of Arthur's trace formula is called the {\it fine $\chi$-expansion}.
For $f$ in $C(\mb{H}(\Af), \omega)$, the twisted (with respect to $\sigma$) fine $\chi$-expansion
is a sum over the set of quadruples
$\{\chi\} = (\mb{M}, \mb{L}, \tau, s)$, consisting of Levi subgroups $\mb{M}, \mb{L}$ of $\mb{H}$, 
an element $s \in W(A_0, M)$, 
and a discrete spectrum automorphic representation $\tau$ of $\mb{M}(\Af)$, such that:
\begin{itemize}
\item
  $\mb{M} \subseteq \mb{L}$;
\item
  $\mf{a}_L$ is the subspace of $\mf{a}_M$ fixed pointwise by $s$;
\item
For all $x \in \mf{a}_M$ which lies outside of $\mf{a}_H$, $\tau(\exp x)$;
\item
  the restriction of $\tau$ to $\mb{Z}(\Af)$ is equal to $\omega$;
\item
  $\tau$ is equivalent to $s \sigma \tau$, 
where $s \sigma \tau(m) := \tau (\sigma s^{-1} m)$ for all $m \in \mb{M}(\Af)$.
\end{itemize}

We say that a parabolic subgroup is \emph{standard} if it contains the fixed minimal parabolic subgroup $\mb{P}_0$.
Let $\mb{P} \in \mc{P}(M)$ be the standard parabolic subgroup with Levi component $\mb{M}$.
The term associated with $(\mb{M}, \mb{L}, \tau, s)$ in the twisted fine $\chi$-expansion is the product of
\begin{equation}\label{eq:finechiexp1}
\frac{\abs{W(A_0, M)}}{\abs{W(A_0, H)}}
\abs{ \det (1 - s)|_{\mf{a}_M/\mf{a}_L}}^{-1}
\end{equation}
and
\begin{equation}\label{eq:finechiexp2}
(2\pi)^{-\dim \lp\mf{a}_L / \mf{a}_H\rp}\int_{i (\mf{a}_L/\mf{a}_H)^*} 
\tr 
\mc{M}^T_L(P, \zeta) 
I_{P, \tau}(\zeta, f) M_{P|\sigma P}(s, \sigma(\zeta)) I_{P, \tau}(\sigma) \, d\zeta.
\end{equation}
Here, $\mc{M}^T_L(P, \zeta)$ is the logarithmic derivative of an intertwining operator.  We will not
reproduce here the definition of $\mc{M}^T_L(P, \zeta)$.  We only use the fact that $\mc{M}^T_L(P, \zeta) = 1$
when $\mb{L}$ is equal to $\mb{H}$.
In particular,
the term in the expansion associated with a quadruple of the form $(\mb{M}, \mb{H}, \tau, s)$ is
\begin{equation}\label{eq:finechidiscrete}
\frac{\abs{W(A_0, M)}}{\abs{W(A_0, H)}}
\abs{ \det (1 - s)|_{\mf{a}_M/\mf{a}_H}}^{-1}
\tr I_{P, \tau}(\zeta, f\times\sigma) M_{P|\sigma P}(s, 0).
\end{equation}
We call the sum of all the terms of the form \eqref{eq:finechidiscrete} the {\bf discrete part} of the fine $\chi$-expansion
of the $\sigma$-twisted trace formula.
We denote it by $I^d(H, f\times\sigma)$. 
We call the sum of the rest of the terms the {\bf continuous part}.  
Observe that the $\sigma$-twisted character of a discrete spectrum representation $\pi$ of $\HH$ corresponds
to the quadruple $(\mb{H}, \mb{H}, \pi, 1)$.  
For any quadruple of the form $(\mb{M}, \mb{H}, \tau, s)$ and parabolic subgroup $\mb{P}$ in $\mc{P}(M)$,
we say that the (normalizedly) parabolically induced representation 
$I_{P, \tau}$ {\bf occurs $\sigma$-discretely} in the spectrum of $\mb{H}(\Af)$.

In this work, we only study the discrete parts of the spectral expansions of the groups.  
This section is devoted to the computation of the terms in 
$I^d(\rG, \rf\times\beta)$,
where $\beta$ denotes also the automorphism of $\mb{\rG}(\Af)$ defined by:
\[
\beta(g_{ij}) := (\beta (g_{ij})),\quad \forall \, (g_{ij}) \in \mb{\rG}(\Af). 
\]

\subsection{$\beta$-Discretely Occurring Representations of $\mb{\rG}(\Af)$}\label{sec:DORrG}
If a quadruple has the form $(\mb{\rG}, \mb{\rG}, \tau, s)$, then $s$ is necessarily trivial, and
$\tau$ can be any $\beta$-invariant, discrete spectrum, automorphic representation of $\mb{\rG}(\Af)$ with
central character equal to the fixed $\omega$.
The contribution to the fine $\chi$-expansion corresponding to
quadruples of the form $(\mb{\rG}, \mb{\rG}, \tau, 1)$ is therefore
\[
\sum_{\tau} \tr \tau(\rf\times\beta),
\]
where the sum is over all $\tau$ as described just now, and $\tau(\rf \times \beta)$ denotes $\tau(\rf)T_\beta$
(see the beginning of Section \ref{sec:DORstf} for the definition of $T_\beta$).

The group $\mb{\rG}$ has only one proper Levi subgroup, namely the 
maximal diagonal torus, which we denote by $\mb{\rM}$.  Its group of $\Af$-points is
\[
\mb{\rM}(\Af) = 
\{\diag(x, y, \alpha(x)^{-1}) : x, y \in \Aep^\times,\; y\alpha(y) = 1\}.
\]
A quadruple of the form $(\mb{\rG}, \mb{\rM}, \tau, s)$ satisfies the conditions prescribed earlier only if
$s$ is the generator $(13)$ of the Weyl group $W(\mb{\rM}, \mb{\rG})$, where
$(13)$ swaps the first and third entries of a diagonal element in $\mb{\rM}(\Af)$ and leaves the
second entry unchanged.
The $\mb{\rM}(\Af)$-module $\tau$ is of the form:
\[
\tau = \mu' \otimes \eta' : \diag(x, y, \alpha (x)^{-1}) \mapsto \mu'(x)\eta'(y),\;\;
\forall \diag(x, y, \alpha(x)^{-1}) \in \mb{\rM}(\Af),
\]
where $\mu'$, $\eta'$ are characters of $C_{E'}$, $C_{E'}^{F'}$, respectively.
By the conditions on the quadruple, we must have
\begin{gather*}
\mu'\otimes\eta' = \tau \cong (13)\beta\tau = \alpha\beta \mu'^{-1} \otimes \beta\eta',\\
\text{i.e.} \quad \mu'\cdot \alpha\beta\mu' = 1,\quad \eta' = \beta\eta'.
\end{gather*}

The constant \eqref{eq:finechiexp1} associated with $(\mb{\rM}, \mb{\rG}, \tau, (13))$
is $1/4$.  Since $\alpha^2 = 1$, the condition $\mu'\cdot \alpha\beta\mu' = 1$ implies that $\beta^2 \mu' = \mu'$,
which in turn implies that $\beta \mu' = \mu'$ because the order $n = [E':E]$ of $\beta$ is odd.
Hence, $\mu' = \alpha\mu^{-1}$.

Consequently, $\mu'\otimes\eta'$ is invariant under the action of $W(\mb{\rM},\mb{\rG})$, and
$I(\mu', \eta')$ is equivalent to $I(\mu'', \eta'')$ if and only if $\mu' = \mu''$, $\eta' = \eta''$.
In conclusion:

\begin{lemma}
The discrete part of the spectral side of the $\beta$-twisted trace formula of $\mb{\rG}$ is the sum$:$
\renewcommand{\labelenumi}{(R\arabic{enumi})\quad}
\[
\sum_{\substack{\tau' = \beta\tau'}} \tr \tau'(\rf \times \beta)
\quad + \quad
\frac{1}{4}\sum_{\substack{\mu'\cdot\alpha\mu' = 1\\\mu'=\beta\mu'\\\eta'=\beta\eta'}}
\!\!\tr I(\mu', \eta')(\rf\times\beta)M_{P|\beta P}((13), 0),
\]
where $\mb{P}$ is the upper triangular Borel subgroup of $\Un(3, E'/F')$.
\end{lemma}

\section{Surjectivity of the Norm}\label{appendix:normindex}
For a number field or a nonarchimedean local field $L$ with quadratic extension $K$,
put 
\[
\begin{split}
C_L &:= \begin{cases}
L^\times\bs\mbb{A}_L^\times & \text{ if } L \text{ is a number field,}\\
L^\times & \text{ if } L \text{ is a local field;}
\end{cases}\\
C_K^L &:= \begin{cases}
\Un(1, K/L)(L) \bs \Un(1, K/L)(\mbb{A}_L) & \text{ if } L \text{ is a number field,}\\
\{z \in K^\times : \N_{K/L} z = 1\} & \text{ if } L \text{ is a local field.}
\end{cases}
\end{split}
\]

Let $F$ be a number field or a local nonarchimedean field.  Let $E$ be a quadratic extension of $F$,
and $F'$ a cyclic extension of $F$ of odd degree $n$.  Let $E'$ be the compositum field $EF'$, with
$\Gal(E'/E) = \la \beta \ra$ and $\Gal(E'/F') = \la \alpha \ra$.

\begin{lemma}\label{lemma:normindex}\label{lemma:globalnormindex}
For both global and local $F$, the norm map$:$
\[
\N_{E'/E} : C_{E'}^{F'} \rightarrow C_E^F
\]
defined by $\N_{E'/E} x = \prod_{i = 0}^{n - 1} \beta^i x$ is surjective.
\end{lemma}
\begin{proof}
By Hilbert's Theorem 90, the group $C_{E'}^{F'}$ is equal to $\{z/\alpha z : z \in C_{E'}\}$, and likewise
$C_E^F = \{z/\alpha z : z \in C_E\}$.  
Hence, to prove the lemma it suffices to show that $C_E = C_F \cdot \N_{E'/E}C_{E'}$.

Consider the group homomorphism:
\[
\iota : C_F/\N_{F'/F}C_{F'} \rightarrow C_E/\N_{E'/E}C_{E'}
\]
induced from the natural embedding of $F$ into $E$.  
It is well-known in class field theory that both quotient groups are cyclic of order $n$.
We claim that $\iota$ is injective:
Let $x \in C_F$ be a representative of a class $\bar{x}$ in $C_F/\N_{F'/F}C_{F'}$
such that $\iota(\bar{x}) = 1$.  By definition, there exists $y \in C_{E'}$ such that $\N_{E'/E}\,y = x$.
Moreover, $\N_{E/F}\N_{E'/E}\, y = \N_{F'/F}\N_{E'/F'}\, y = x^2$, 
which implies that $x^2 \in \N_{F'/F}F'^\times$, i.e.\ $\bar{x}^2 = 1$.  
Since $C_{F'}/\N_{F'/F}C_{F'}$ is of odd order, $\bar{x}^2 = 1$ implies that $\bar{x} = 1$.  

Since $\iota$ is an injective homomorphism between two finite groups of equal order, it is surjective.
Hence, $C_E = C_F\cdot\N_{E'/E}C_{E'}$, which completes the proof.
\end{proof}

\section{$\beta$-Invariant Characters}
\begin{lemma}\label{lemma:localcentralchars}\label{lemma:globalcentralchars}
Let the fields $F, F', E, E'$ be as in the previous section.
In both the global and local cases,
a character $\omega'$ of $C_{E'}^{F'}$ is $\beta$-invariant if and only if it is equal to $\omega\circ\N_{E'/E}$
for some character $\omega$ of $C_E^F$.
\end{lemma}
\begin{proof}
That $\omega\circ\N_{E'/E}$ is $\beta$-invariant is obvious.  We now prove the converse.
Let $\omega'$ be a $\beta$-invariant character of $C_{E'}^{F'}$.
Let $\tilde{\omega}'$ be the character of $C_{E'}$ defined by:
\[
\tilde{\omega}'(z) = \omega'(z/\alpha z),\quad \forall\, z \in C_{E'}.
\]
Since $\omega' = \beta\omega'$, we have $\tilde{\omega}' = \beta\tilde{\omega}'$, 
and it follows from Hilbert's Theorem 90 that $\tilde{\omega}' = \tilde{\omega}\circ\N_{E'/E}$ 
for some character $\tilde{\omega}$ of $C_E$.

Since the restriction of the character $\tilde{\omega}'$ to $C_{F'}$ is trivial, 
so is the restriction of $\tilde{\omega}$ to $\N_{E'/E}C_{F'} = \N_{F'/F}C_{F'}$.
By class field theory, the group $C_F/\N_{F'/F}C_{F'}$ is isomorphic to the cyclic group 
$\Gal(F'/F) \cong \mbb{Z}/n\mbb{Z}$.  Let $c$ be an element of $C_F$ 
such that the coset $c\, \N_{F'/F}C_{F'}$ generates $C_F/\N_{F'/F}C_{F'}$.
We have:
\[
\tilde{\omega}(c)^n = \tilde{\omega}\circ\N_{E'/E}(c) = \tilde{\omega}'(c) = 1.
\]
Hence, $\tilde{\omega}(c) = \ve_{F'/F}^i(c)$ for some integer $i$,
where $\ve_{F'/F}$ is the character of $C_F$ associated with the field extension $F'/F$ via class field theory.

Since $n$ is odd, there exists an integer $m$ such that $2m \equiv 1\! \mod n$.  
Let $\ve_{E'}$ denote the character $\ve_{F'/F}^m\circ\N_{E'/F'}$ of $C_{E'}$, which is trivial on $\N_{F'/F}C_{F'}$.
Noting that $\ve_{F'/F}^n = 1$, we have $\ve_E(c) = \ve_{F'/F}^{2m}(c) = \ve_{F'/F}(c)$.
Thus, $\ve_{E'}$ restricts to $\ve_{F'/F}$ on $C_{F}$, and 
the character $\tilde{\omega}_0 := \ve_E^{-i}\tilde{\omega}$ is trivial on $C_F$.
Hence, by Hilbert's Theorem 90 there exists a character $\omega$ of $C_E^F$ 
such that $\tilde{\omega}_0(z)$ is equal to $\omega(z/\alpha z)$ 
for all $z \in C_E$.  
Since $\tilde{\omega}' = \tilde{\omega}_0\circ\N_{E'/E}$, and $\alpha$, $\beta \in \Gal(E'/F)$ commute with each other, 
we have:
\[
\omega'\lp\frac{z}{\alpha z}\rp 
= \omega\lp\frac{\N_{E'/E} z}{\alpha(\N_{E'/E} z)}\rp 
= \omega\lp\N_{E'/E}\lp\frac{z}{\alpha z}\rp\rp,
\quad \forall\, z \in C_{E'}.
\]
So, invoking Hilbert's Theorem 90 yet again, we conclude that the character 
$\omega'$ of $C_{E'}^{F'}$ is equal to $\omega\circ\N_{E'/E}$.
\end{proof}

\section{Selected results of \cite{F} and \cite{FU2}}\label{sec:summaryF}
Let $F$ be a number field or a $p$-adic field, $E$ a quadratic extension of $F$.
For an element $g$ in $\GL(m, E)$ ($m = 2, 3$), let $\sigma (g) = J\alpha({}^t g^{-1}) J^{-1}$,
where $J$ is $\lp\lsm &&1\\&-1&\\1&&\rsm\rp$ if $m = 3$, $\lp\lsm &1\\-1&\rsm\rp$ if $m = 2$.
We say that a representation $(\pi, V)$ of $\GL(m, E)$ is 
$\sigma$-invariant if $(\pi, V) \cong (\sigma \pi, V)$, where $(\sigma \pi, V)$ is the
$\GL(m, E)$-module defined by
$
\sigma \pi : g \mapsto \pi(\sigma (g))$, $\forall \, g \in \GL(m, E).
$

\subsection{Local results}\label{sec:summaryFlocal}
We first consider the case where $F$ is $p$-adic.
In \cite{FU2} (resp.\ \cite{F}), 
the local packets of $\Un(2, E/F)$ (resp.\ $\Un(3, E/F)$) 
are defined in terms of the local character identities which they satisfy.
The sum of the characters of the representations in a local packet is shown to be invariant under stable conjugation.

All representations which we consider in this section are by assumption admissible.
\begin{proposition}\cite[p.\ 699-700]{FU2}\label{prop:u1u2localpacket}
Each pair $(\theta_1, \theta_2)$ of characters of $\Un(1, E/F)$ lifts to a local packet
$\rho(\theta_1, \theta_2)$ which consists of two irreducible representations $\pi^+$, $\pi^-$ of $H = \Un(2, E/F)$.  They are
cuspidal if $\theta_1 \neq \theta_2$.  
\end{proposition}

\noindent
Note that if we fix the central character of the representations of $\Un(2, E/F)$, 
then the lifting from $\Un(1)\times\Un(1)$ is equivalent to a lifting from $\Un(1)$.

Fix a character $\kappa$ of $E^\times/\N_{E/F}E^\times$ which is nontrivial on $F^\times$.
Defined in \cite{FU2} are two local liftings from $\Un(2, E/F)$ to $\GL(2, E)$, 
called stable and unstable base change, respectively.
These two liftings are related to each other as follows:
If a local packet $\{\rho\}$ of $\Un(2, E/F)$ lifts via stable base change to a representation 
$\tilde{\rho}(\{\rho\})$ of $\GL(2, E)$, 
then $\{\rho\}$ lifts via unstable base change to $\tilde{\rho}(\{\rho\})\otimes\kappa$ (see \cite[p.\ 716]{FU2}).

\begin{proposition}{\rm (Base Change \cite[Sec.\ 7.\ Prop.\ 3]{FU2})}
Each $\sigma$-invariant, cuspidal, irreducible representation of $\GL(2, E)$ 
is the lift of a cuspidal, irreducible representation of $\Un(2, E/F)$ via either the stable or the unstable base change.
\end{proposition}

\begin{proposition}{\rm (Endoscopy \cite[Part 2. Sec.\ III.2.3.\  Cor.]{F})}\label{prop:u2u3localpacket}
Each local packet $\{\rho\}$ of square-integrable $\Un(2, E/F)$-modules 
lifts to a local packet $\pi(\{\rho\})$ of $\Un(3, E/F)$.  The cardinality of $\pi(\{\rho\})$ is twice that of $\{\rho\}$.
\end{proposition}

\begin{proposition}{\rm (Base Change \cite[Part 2. Sec.\ III.3.9]{F})}\label{prop:u3gl3localpacket}
Each local packet of $\Un(3, E/F)$ lifts to a representation of $\GL(3, E)$.
This lifting defines a one-to-one correspondence between the set of local packets of $\Un(3, E/F)$, and
the set of $\sigma$-stable, $\sigma$-invariant, irreducible representations of $\GL(3, E)$.
\end{proposition}

\noindent
We customarily let $\tilde{\pi}(\{\pi\})$ 
denote the $\GL(3, E)$-module to which a local packet $\{\pi\}$ of $\Un(3, E/F)$ lifts.

\renewcommand{\labelenumi}{\arabic{enumi}.}
\begin{proposition}{\rm (\cite[Part 2. Sec.\ III.3.9]{F})}\label{prop:u3singletonlocalpacket}
Let $\{\pi\}$ be a local packet of square-integrable representations of $\Un(m, E/F)$ $(m = 2, 3)$.
The following statements are equivalent:
\begin{enumerate}
\item $\{\pi\}$ is a singleton.
\item 
$\{\pi\}$ lifts to a square-integrable, irreducible $\GL(m, E)$-module.
\item $\{\pi\}$ is not the lift of a packet/representation from $\Un(m - 1, E/F)$.
\end{enumerate}
\end{proposition}

\subsection{Global results}\label{sec:summaryFglobal}
Let $F$, $E$ now be number fields.
In \cite{FU2} and \cite{F},
the global (quasi-)packets of $\Un(2, E/F)$ and $\Un(3, E/F)$ are restricted tensor products of (quasi-)local packets.
The global liftings considered below are uniquely defined by their analogues in the local case.

A (quasi-)packet is said to be {\bf discrete spectrum} if it contains a discrete spectrum automorphic representation.
We say that a discrete spectrum (quasi-)packet is {\bf stable} if each of its members appears with the same multiplicity in the 
discrete spectrum, and {\bf unstable} otherwise.

Each pair of characters $(\theta_1, \theta_2)$ of $\Un(1, E/F)(F)\bs\Un(1, E/F)(\Af)$ 
lifts to a global packet $\rho(\theta_1, \theta_2)$ of $\Un(2, E/F)(\Af)$.  
A discrete spectrum packet of $\Un(2, E/F)(\Af)$ is stable if and only if it is not of the form $\rho(\theta_1, \theta_2)$.

As in the local case,
there are stable and unstable global base-change liftings from $\Un(2, E/F)$ to $\GL(2, \Ae)$.
The unstable lifting is associated with a character $\kappa$ of $\idc{E}/\N_{E/F}\idc{E}$ which is nontrivial on $\idc{F}$.
The stable and unstable liftings are related to each other as follows:
If a global packet $\{\rho\}$ of $\Un(2, E/F)$ lifts via stable base change to an automorphic representation 
$\tilde{\rho}(\{\rho\})$ of $\GL(2, \Ae)$, 
then $\{\rho\}$ lifts via unstable base change to $\tilde{\rho}(\{\rho\})\otimes\kappa$.

\begin{proposition}{\rm(\cite[Sec.\ 5.\ Lemma 8; Sec.\ 7.\ Prop.\ 4]{FU2})}\label{prop:u2gl2globalpacket}
A $\sigma$-invariant, irreducible, discrete spectrum, automorphic representation of $\GL(2, \Ae)$ 
is either the stable or unstable base change of a packet of $\Un(2, E/F)(\Af)$.

A discrete spectrum packet of $\Un(2, E/F)(\Af)$ 
is stable if and only if it lifts to a discrete spectrum automorphic representation of $\GL(2, \Ae)$.
\end{proposition}

\begin{proposition}{\rm (\cite[p.\ 217, 218; Part 2.\ Thm.\ III.5.2.1]{F})}\label{prop:u2u3gl3globalpacket}
Each discrete spectrum global packet $\{\rho\}$ of $\Un(2, E/F)(\Af)$
lifts to an unstable packet $\pi(\{\rho\})$ $($if $\sum_{\rho \in \{\rho\}}\dim\rho = \infty)$,
or a quasi-packet $\pi(\{\rho\})$ $($if $\sum_{\rho \in \{\rho\}} \dim\rho = 1)$,
of $\Un(3, E/F)(\Af)$.
All unstable packets and quasi-packets of $\Un(3, E/F)(\Af)$ are so obtained.
Moreover, $\pi(\rho)$ weakly lifts to the parabolically induced representation 
$I_{(2, 1)}(\tilde{\rho}(\{\rho\})\otimes\kappa, \eta)$ of $\GL(3, \Ae)$, where $\eta$ is a character of $\idc{E}$.
\end{proposition}

\begin{theorem}
{\rm (\cite[Part 2.\ Thm.\ III.5.2.1]{F})}
\label{thm:u3gl3stableglobalpacket}
Each $\sigma$-invariant, automorphic representation of $\GL(3, \Ae)$ is the lift of an automorphic representation of $\Un(3, E/F)(\Af)$.
The global lifting from $\Un(3, E/F)(\Af)$ to $\GL(3, \Ae)$ 
gives a one-to-one correspondence between the set of stable discrete spectrum packets of $\Un(3, E/F)(\Af)$, 
and the set of $\sigma$-invariant, irreducible, discrete spectrum, automorphic representations of $\GL(3, \Ae)$.
\end{theorem}

\begin{theorem}{\rm (Multiplicity One Theorem \cite[Part 2.\ Cor.\ III.5.2.2(1)]{F})}
\label{thm:multoneu3}
If $\pi$ is a discrete spectrum automorphic representation of $\mb{G}(\Af)$
each of whose dyadic local components lies in a local packet containing a constituent of a parabolically induced representation,
then $\pi$ occurs in the discrete spectrum of $L^2(\mb{G}(F)\bs\mb{G}(\Af))$ with multiplicity one.
\end{theorem}

\noindent
{\sc Remark.}
The restriction at the dyadic places is due only to a shortcoming in the
current proof.

\begin{theorem}{\rm (Rigidity Theorem \cite[Part 2.\ Cor.\ III.5.2.2(2)]{F})}
\label{thm:rigidity}
If $\pi$ and $\pi'$ are discrete spectrum representations of $\mb{G}(\Af)$ whose local components $\pi_v$ and $\pi'_v$
are equivalent for almost all places $v$ of $F$, then they lie in the same packet, or quasi-packet.
\end{theorem}

\end{document}